\documentclass[a4paper,10pt]{amsart}
\usepackage[backend=bibtex8]{biblatex}
\bibliography{Bibliography}

\usepackage[english]{babel} 
\usepackage[perpage,symbol*]{footmisc} 
\usepackage{esvect} 
\usepackage{subfig} 
\usepackage{placeins} 

\usepackage{amssymb}

\usepackage{array}
\usepackage{arydshln} 
\usepackage{color}

\usepackage[colorlinks]{hyperref} 

\usepackage[T1]{fontenc}
\usepackage[utf8]{inputenc} 

\usepackage{pgf,tikz} 
\usetikzlibrary{matrix,arrows,calc}

\title[MHS on the rational homotopy of intersection spaces]{Mixed Hodge Structures on the rational homotopy type of intersection spaces}

\author{Mathieu Klimczak}
\address{Université des Sciences et Technologies, Lille}
\email{mathieu.klimczak@ed.univ-lille1.fr}

\keywords{Intersection spaces, rational homotopy, mixed Hodge theory, weight spectral sequence, isolated singularities, formality}
\subjclass[2010]{55P62, 32S35, 55N33}

\newtheorem{thm}{Theorem}[section]
\newtheorem{defi}{Definition}[thm]
\newtheorem{cor}{Corollary}[thm]
\newtheorem{lem}{Lemma}[thm]
\newtheorem{propo}{Proposition}[thm]
\theoremstyle{definition} \newtheorem{exmp}{Example}[section]
\theoremstyle{remark} \newtheorem{remark}[thm]{Remark}

\newcommand{\tr}[2]{t_{#1}#2} 
\newcommand{\redHI}[3]{\widetilde{H}I_{\overline{#1}}^{#2}(#3)} 
\newcommand{\redhI}[3]{\widetilde{H}I^{\overline{#1}}_{#2}(#3)} 
\newcommand{\IH}[3]{IH_{\overline{#1}}^{#2}(#3)} 
\newcommand{\I}[2]{I^{\overline{#1}}#2} 
\newcommand{\Apl}[1]{\mathsf{A}_{\mathsf{PL}}(#1)} 
\newcommand{\Ho}[1]{\mathsf{Ho}(#1)} 
\newcommand{\cotr}[2]{t^{#1}#2} 
\newcommand{\AI}[2]{AI_{\overline{#1}}(#2)} 
\newcommand{\CDGA}[1]{\mathsf{CDGA}_{\mathbf{#1}}} 
\newcommand{\Top}{\mathsf{Top}} 
\newcommand{\dgaucotr}[2]{\xi_{+}^{\overline{#1}}#2} 
\newcommand{\dgacotr}[2]{\xi^{\overline{#1}}#2} 
\newcommand{\Dec}[1]{\mathrm{Dec}(#1)} 
\newcommand{\coim}[1]{\mathsf{C}_{\overline{#1}}} 
\newcommand{\E}[5]{E_{#2}^{#3,#4}(#5,#1)} 
\newcommand{\gr}[3]{\mathrm{gr}_{#1}^{#2}(#3)} 
\newcommand{\pb}[2]{\mathcal{J}_{\overline{#1}}(#2)} 
\newcommand{\cdga}[1]{\mathcal{P}_{n}^{op}\mathsf{CDGA}_{\mathbf{#1}}} 
\newcommand{\mhcdga}[1]{\mathcal{P}_{n}^{op}\mathbf{MH}\mathsf{CDGA}_{\mathbf{#1}}} 
\newcommand{\super}{\mathsf{Super}\mathcal{V}_{\mathbf{C}}} 
\newcommand{\MI}[2]{MI_{\overline{#1}}(#2)} 
\newcommand{\EI}[3]{EI_{#1,\overline{#2}}(#3)} 
\newcommand{\EIbidg}[4]{EI_{#1,\overline{#2}}^{#3}(#4)} 
\newcommand{\FI}[3]{FI_{\overline{#1}}^{#2}(#3)} 
\newcommand{\HI}[3]{HI_{\overline{#1}}^{#2}(#3)} 
\newcommand{\coker}[1]{\mathrm{coker\,}#1} 
\newcommand{\im}[1]{\mathrm{im\,}#1} 
\newcommand{\D}[1]{#1^{\vee}} 
\newcommand{\phip}[2]{\phi_{\overline{#1}}^{#2}} 
\newcommand{\psip}[2]{\psi_{\overline{#1}}^{#2}} 
\newcommand{\diff}[2]{\partial_{\overline{#1}}^{#2}} 
\newcommand{\difb}[2]{d_{\overline{#1}}^{#2}} 
\newcommand{\pomap}[2]{\varphi_{\overline{#1}, \overline{#2}}} 
\newcommand{\Pos}[1]{\mathcal{P}_{#1}}
\newcommand{\col}{\colon\thinspace} 
\newcommand{\IA}[2]{IA_{\overline{#1}}(#2)} 
\newcommand{\IE}[3]{IE_{#1,\overline{#2}}(#3)} 

\usepackage{pgf,tikz} 
\usetikzlibrary{matrix,arrows,calc}

\begin{document}

\begin{abstract}  
Let $X$ be a complex projective variety of complex dimension $n$ with only isolated singularities of simply connected links. We show that we can endow the rational cohomology of the family of the $\overline{p}$-perverse intersection spaces $\lbrace \I{p}{X} \rbrace_{(\overline{p})}$ with compatible mixed Hodge structures.
\end{abstract}
\maketitle

\section{Introduction}
This paper deals with the notion of mixed Hodge structure associated to the intersection spaces of a complex projective variety $X$ of complex dimension $n$ with only isolated singularities and simply connected links. 

Intersection spaces were defined by Markus Banagl in \cite{Banagl2010} as a way to spatialize Poincaré duality for singular spaces. Suppose given a compact, connected pseudomanifold of dimension $n$ with only isolated singularities and simply connected links. We assign to this space a family of topological spaces $\I{p}{X}$, its intersection spaces, where $\overline{p}$ is an element called a perversity varying in a poset $\Pos{n}$ called the poset of perversities. We then have for complementary perversities a generalized Poincaré duality isomorphism
\[
\redHI{p}{k}{X} \cong \D{\redhI{q}{n-k}{X}}.
\]
with $\D{\redhI{q}{n-k}{X}}=\hom(\redhI{q}{n-k}{X}, \mathbf{Q}) $

The theory of intersection spaces can be seen as an enrichment of intersection homology since they both gives complementary informations about $X$.

The aim of this paper is twofold. First we want to get a better understanding of the family of cohomology algebras $\lbrace \HI{\bullet}{\ast}{X} \rbrace_{\overline{p} \in \Pos{n}}$ when we take all the spaces into consideration. We then want to put a mixed Hodge structure on these algebras and get results about formality of intersections spaces.

Formality is a notion tied to the rational homotopy theory of topological spaces. The rational homotopy type of a topological space $X$ is given by the commutative differential graded algebra $\Apl{X}$ in the homotopy category $\mathsf{Ho}(\CDGA{Q})$ defined by formally inverting quasi-isomorphisms and where $\Apl{-} \col \Top \rightarrow \CDGA{Q}$ is  the polynomial De Rham functor defined by Sullivan. The space $X$ is then formal if there is a string of quasi-isomorphisms from the cdga $\Apl{X}$ to its cohomology with rational coefficients $H^{\ast}(\Apl{X}) \cong H^{\ast}(X, \mathbf{Q})$ seen as a cdga with trivial differential. In particular is $X$ is formal then its rational homotopy type is a formal consequence of its cohomology ring, its higher order Massey products vanish. 

The combination of rational homotopy theory and Hodge theory has already been showed to be fruitful. Using Hodge theory, Deligne, Griffiths, Morgan and Sullivan proved in \cite{Deligne1975} that compact Kähler manifolds, in particular smooth projective varieties, are formal. It was also shown by Simpson in \cite{Simpson2011} that every finitely presented group $G$ is the fundamental group of a singular projective variety $X$ and then Kapovich and Koll\'ar showed in \cite{Kapovich2014} that this $X$ could be chosen to be complex projective with only simple normal crossing singularities. More recently, Chataur and Cirici proved in \cite{Chataur2015} that every complex projective variety of dimension $n$ with only isolated singularities $\Sigma = \lbrace\sigma_{1}, \dots, \sigma_{\nu} \rbrace$ such that the link $L_{i}$ of each singularities $\sigma_{i}$ is $(n-2)$-connected is then a formal topological space.

The intersection spaces $\I{p}{X}$ of $X$ are not complex nor algebraic varieties, even if $X$ is. Thus at first glance there should be no reasons the cohomology of these spaces carry a mixed Hodge structure. On second thought, when $X$ is a complex projective variety of complex dimension $n$ with only isolated singularities and that we look at the rational cohomology of their intersection spaces
\[
\HI{p}{k}{X} =
\begin{cases}
\mathbf{Q} & k=0 \\
H^{k}(X) & 1 \leq k \leq p \\
H^{k}(X) \oplus \im{H^{k}(X_{reg}) \rightarrow H^{k}(L)} & k = p+1 \\
H^{k}(X_{reg}) & k > p+1
\end{cases}
\]
it becomes a bit more natural to think that there is a mixed Hodge structure since each part of their rational cohomology can be endowed with a natural mixed Hodge structure coming from $X$. We show here that in fact all these structures naturally come from a mixed Hodge structure at the algebraic models level and that this structure is compatible with the different operations defined on intersection spaces. Note that our definition of intersection spaces \ref{def:intersectionspace2} differs slightly from the original definition given in \cite{Banagl2010}.

It must be pointed out here that the question of a Hodge structure on the intersection spaces as already been looked at in the work of Banagl and Hunsicker \cite{Banagl2015} where they use $L^{2}$-cohomology to provide a Hodge theoretic structure. We do not follow this path here and rather modify the rational homotopy theory tools developed in \cite{Chataur} for the mixed Hodge structures in intersection cohomology.

We explain the contents of this paper. 

The section \ref{section:background} is devoted to collect the different definitions needed. We recall what we call a perversity, the definition of the intersection spaces and the convention we use to construct them. We also introduce the notion of a coperverse cdga which is the main tool for the rational algebraic models of the intersection spaces. We then define a model category structure on the category of coperverse cdga's \ref{thm:model_structure}.

The section \ref{section:coperv_algebras} is a direct application of the previous section. We define the notion of a coperverse cdga associated to a morphism of cdga's. As a result we show that the whole family of algebraic model $\AI{\bullet}{X}$ computing the rational cohomology of intersection spaces carry a structure of coperverse algebra and that we have a external product on that family, extending the cup product that each $\I{p}{X}$ naturally has as a topological space.

The section \ref{section:Hodge_th} is the main section of this paper, we extend our notion of coperverse cdga to the notion of coperverse mixed Hodge cdga. These coperverse mixed Hodge algebras carry a mixed Hodge structure which is compatible the differential, product and poset maps of the underlying coperverse cdga. After developing their algebraic definitions we show in theorem \ref{thm:coperverse_MI} that given a complex projective variety $X$ of complex dimension $n$ with only isolated singularities and simply connected links, there is a coperverse mixed Hodge cdga $\MI{\bullet}{X}$ quasi-isomorphic to the coperverse cdga $\AI{\bullet}{X}$. As a result the whole family $\HI{\bullet}{\ast}{X}$ carry a well defined mixed Hodge structure defined at the algebraic models level.

The section \ref{section:Weight_SS} is devoted to the computation of the associated weight spectral sequence. If $X$ is a complex projective algebraic variety with only isolated singularities and such that $X$ admits a resolution of singularities where the exceptional divisor is smooth, we are able to compute the weight spectral sequence associated to the mixed Hodge structure. We then use this spectral sequence to show a result of "purity implies formality" in theorem \ref{thm:pure_is_formal}.

The section \ref{sec:formality_3folds} is completely devoted to the proof of the theorem \ref{thm:3fold_formal} : suppose $X$ to be a complex projective algebraic threefold with isolated singularities such that there exist a resolution of singularities with a smooth exceptional divisor, then if the links are simply connected the intersection spaces $\I{p}{X}$ are formal topological spaces for any perversity $\overline{p}$. The proof being rather long and intricate, we made the choice of giving it its own section. This result goes well with the result of \cite[Theorem E p.76]{Chataur2012} stating that any nodal hypersurface in $\mathbf{C}P^{4}$ is intersection-formal.

The last section \ref{section:Eg} deals with computations, with for instance the computations for the Calabi-Yau generic quintic 3-fold \ref{subsec:CY_generic_quintic} and the Calabi-Yau quintic 3-fold \ref{subsec:CY_quintic} where we are able to retrieve the cohomology of the associated smooth deformation as stated in \cite{Banagl2012}.

\section{Background, intersection spaces and coperverse algebras}
\label{section:background}
\subsection{Perversities and intersection spaces}
Unless stated otherwise, all cohomology groups will be considered with rational coefficients and they will be omitted.

Since we are concerned about complex algebraic varieties of complex dimension $n$ with only isolated singularities we use the following definition of a perversity. A perversity $\overline{p}$ is determined by a integer $0 \leq p \leq 2n-2$, we then denote by $\Pos{n}^{op}$ the poset $\lbrace 0, \dots, 2n-2 ; \leq \rbrace$ with the reverse order and $\widehat{\Pos{n}}^{op} := \Pos{n}^{op} \cup \lbrace \infty \rbrace$. The posets $\Pos{n}^{op}$ and $\widehat{\Pos{n}}^{op}$ are then totally ordered and look like
\[
2n-2 \rightarrow 2n-3 \rightarrow \cdots \rightarrow 2 \rightarrow 1 \rightarrow 0.
\]
\[
\overline{\infty} \rightarrow 2n-2 \rightarrow 2n-3 \rightarrow \cdots \rightarrow 2 \rightarrow 1 \rightarrow 0.
\]

The maximal element is the zero perversity $\overline{0}=0$, the minimal element is the top perversity $\overline{t}=2n-2$ for $\Pos{n}^{op}$ and $\infty$ for $\widehat{\Pos{n}}^{op}$. The partial addition $\oplus$ is just the classical addition and we put $\overline{p} \oplus \overline{q} := \overline{p+q}$ if $p+q \leq 2n-2$ for $\Pos{n}^{op}$ and $\widehat{\Pos{n}}^{op}$. The complementary perversity $\overline{q}$ of $\overline{p}$ is then $\overline{q} = \overline{t}-\overline{p} = \overline{t-p}$.

If we do not consider complex varieties but just pseudomanifold of dimension $n$ with only isolated singularities, we will still use a linear poset 
\[
n-2 \rightarrow n-3 \rightarrow \cdots \rightarrow 2 \rightarrow 1 \rightarrow 0.
\]

Throughout this paper, every equation involving perversities will be considered in $\mathcal{P}^{op}$. For example $\max(\overline{p}, \overline{0}) = \overline{0}$ for all $\overline{p}$ and if $\overline{p}=2$ and $\overline{q}=1$, then $\overline{p} < \overline{q}$.

Intersection spaces were defined by Markus Banagl in \cite{Banagl2010} in an attempt to spatialize Poincaré duality for singular spaces. The construction of these spaces rely on the notion of \textit{spatial homology truncation} also introduced in \cite{Banagl2010}. 

\begin{defi}
Given a simply connected CW-complex $K$ of dimension $n$ and an integer $k \leq n$. A spatial homology truncation of cut-off degree $k$ of $K$ is a CW-complex $\tr{k}{K}$ together with a comparison map
\[
f \col \tr{k}{K} \longrightarrow K
\]
such that 

\begin{equation}
H_{r}(\tr{k}{K}) \cong
\begin{cases}
H_{r}(K)      &  r < k,\\
0             &  r \geq k.
\end{cases}
\end{equation}

The integer $k$ is called the cut off degree of the homological truncation.
\end{defi}

\begin{remark}
Such a truncation always exists provided that $K$ is simply connected and this truncation is in fact defined on $\mathbf{Z}$ and not just on $\mathbf{Q}$, see \cite{Banagl2010}.
\end{remark}

\begin{defi}
Let $X$ be a compact, connected, oriented pseudomanifold of dimension $n$ and denote by $\Sigma = \lbrace \sigma_{1}, \dots, \sigma_{\nu} \rbrace$ the singular locus of $X$. The pseudomanifold $X$ is called supernormal if the link $L_{i}$ of each singularity $\sigma_{i} \in \Sigma$ is simply connected. 

We denote by $\super$ the category of supernormal complex projective varieties with only isolated singularities together with the morphisms $f \col X \rightarrow Y$ such that $f(X_{reg}) \subset Y_{reg}$.
\end{defi}

For the rest of this paper, we assume that the definition of a supernormal pseudomanifold $X$ includes the fact that $X$ is a connected pseudomanifold of dimension $n$ (the compacity and orientability assumptions being automatic since we work in projective spaces $\mathbf{C}P^{n}$).

Before giving our definition of intersection spaces, let us define which cut off degree we use with respect to the perversities for the spatial homological truncation. This definition will be different from the one in \cite{Banagl2010} and will be more suited to our notion of coperverse cdga we will introduce in definition \ref{def:coperverse_cdga}.

Let $K$ be a simply connected CW-complex of dimension $n$ and suppose given a perversity $\overline{p}$. We set that the cut-off degree is directly given by the perversity $\overline{p}$ and we denote it by $\tr{\overline{p}}{K}$. That is

\begin{equation}
H_{r}(\tr{\overline{p}}{K}) \cong
\begin{cases}
H_{r}(K)      & \text{if } r \leq p\\
0                        & \text{if } r > p.
\end{cases}
\end{equation}

Note that we also swap the strict and large inequalities in the definition. We will use this convention for the rest of this paper. 

By convention we also define $\tr{\overline{\infty}}{K} = K$.

Given a supernormal pseudomanifold $X$ with isolated singularities, 
\[
L(\Sigma, X) := \sqcup_{\sigma_{i}} L_{i}
\]
is then the disjoint union of simply connected topological manifold of dimension $n-1$. Denote by $X_{reg} := X - \Sigma$ the regular part of $X$. We denote by $\cotr{\overline{p}}{L_{i}}$ the homotopy cofiber of the map 
\[
f_{i} \col \tr{\overline{p}}{L_{i}} \rightarrow L_{i}.
\] 
We have maps 
\[
f^{i} \col L_{i} \longrightarrow \cotr{\overline{p}}{L_{i}}.
\]

\begin{defi}
\label{def:intersectionspace2}
The intersection space $\I{p}{X}$ of the space $X$ is defined by the following homotopy pushout diagram 
\[
\begin{tikzpicture}
\matrix (m)[matrix of math nodes, row sep=2em, column sep=5em, text height=1.5ex, text depth=0.25ex]
{ L(\Sigma,X)                                 & X_{reg}       \\
  \bigsqcup_{i} \cotr{\overline{p}}{L_{i}} & \I{p}{X} \\};

\path[->]
(m-1-1) edge node[] {} (m-1-2);
\path[dashed,->]
(m-2-1) edge node[] {} (m-2-2);

\path[->]
(m-1-1) edge node[] {} (m-2-1);
\path[dashed,->]
(m-1-2) edge node[] {} (m-2-2);
\end{tikzpicture}
\]
\end{defi}

We shall use this definition of intersection spaces for the rest of the paper. Note that with this definition we have $\I{\infty}{X} = \overline{X}$ which is the normalization of $X$. We will denote by $\HI{p}{\ast}{X} := H^{\ast}(\I{p}{X})$ and by $\redHI{p}{\ast}{X}$ the reduced cohomology. We then have 

\[
\HI{p}{r}{X} =
\begin{cases}
\mathbf{Q} & r=0 \\
H^{r}(X) & 1 \leq r \leq p \\
H^{r}(X) \oplus \im{H^{r}(X_{reg}) \rightarrow H^{r}(L)} & r = p+1 \\
H^{r}(X_{reg}) & r > p+1
\end{cases}
\]

In particular, we have $\HI{0}{\ast}{X} = H^{\ast}(X_{reg})$ and $\HI{\infty}{\ast}{X} = H^{\ast}(\overline{X})$.
\begin{remark}
\begin{enumerate}
\item Our intersection spaces $\I{p}{X}$ are different from the intersection spaces originally defined in \cite{Banagl2010} since they are not defined as a homotopy cofiber. When there is only one isolated singularity, there is no difference between the two definitions. Differences arise only for the first cohomology group when there is more than one isolated singularity. 
\item This convention also has to be compared at the level of algebraic models with \cite{Chataur2012}, where a $\overline{p}$-perverse rational model of a cone $cL$ on a topological space $L$ of dimension $n$ is given by a truncation in degree $\overline{p}(n)$ of the rational model of $L$. In our case, a rational model of the intersection space $\I{p}{cL}$ is then given by a unital cotruncation in degree $\overline{p}(n)$ of the rational model of $L$.
\end{enumerate}
\end{remark}

Let's compute the bounds of the different weight filtrations involved in $\HI{p}{r}{X}$ for a general perversity $\overline{p}$. Denote by $R^{r}(X_{reg},L) := \im{H^{r}(X_{reg}) \rightarrow H^{r}(L)}$.

\begin{lem}
For $r < n$, $R^{r}(X_{reg},L)$ is pure of weight $r$. For $r \geq n$, we have  
\[
0 = W_{r} \subset W_{r+1} \subset \cdots \subset W_{2r} = R^{r}(X_{reg},L).
\]
\end{lem}

\begin{proof}
This follows from the semi purity of the link, see \cite{Steenbrink1983}. Since $\dim(\Sigma) =0$, the weight filtration on the cohomology of the link is semi-pure, meaning :
\begin{itemize}
\item the weights on $H^{r}(L)$ are less than or equal to $r$ for $r <n$,
\item the weights on $H^{r}(L)$ are greater or equal to $r+1$ for $r \geq n$.
\end{itemize}

Combined with the two following facts 
\begin{itemize}
\item The filtration $0 \subset W_{r} \subset \cdots \subset W_{2r} = H^{r}(X_{reg})$.
\item $H^{r}(X_{reg}) \rightarrow H^{r}(L)$ is a morphism of mixed Hodge structures.
\end{itemize} 
\end{proof}

We have three cases

\noindent\makebox[\textwidth]{
\(
\begin{array}{|c||c | c c c| c : c|} 
\hline
\multicolumn{7}{|c|}{\text{First case : \,}\overline{p} < \overline{m} =n-1} \\
\hline
       & 1 \leq r \leq \overline{p} &   \multicolumn{3}{c|}{r=p+1}                             & \overline{p}+1 <r<n       & n \leq r    \\
\hline
-1     &            0               &         0           &        &                           &                           &    \\
0      &      W_{0}                 &       W_{0}         &        &                           &                           &    \\
1      & W_{1}                      &       W_{1}         &        &                           &                           &    \\
\vdots & \vdots                     &       \vdots        &        &                           &                           &        \\
r-1    & W_{r-1}                    &       W_{r-1}       &        &         W_{r-1}=0         &      W_{r-1}=0            & W_{r-1}=0    \\
\hline
r      & W_{r}                      &       W_{r}         & \oplus &         W_{r}             &      W_{r}                & W_{r}    \\
\hline
r+1    &                            &                     &        &                           &                           & W_{r+1}    \\
\vdots &                            &                     &        &                           &                           & \vdots   \\
2r-1   &                            &                     &        &                           &                           & W_{2r-1}    \\
2r     &                            &                     &        &                           &                           & W_{2r}    \\
\hline
       & H^{r}(X)                   & H^{r}(X)            & \oplus & R^{r}(X_{reg},L)          & H^{r}(X_{reg})            & H^{r}(X_{reg})  \\
\hline
\end{array}  
\)}

\noindent\makebox[\textwidth]{
\(
\begin{array}{|c||c | c c c|  c|} 
\hline
\multicolumn{6}{|c|}{\text{Second case : \,}\overline{p} = \overline{m} =n-1} \\
\hline
       & 1 \leq r \leq n-1          &   \multicolumn{3}{c|}{r=p+1=n}                                                                          & n \leq r    \\
\hline
-1     &  0                         &           0         &        &                                                                          &    \\
0      & W_{0}                      &        W_{0}        &        &                                                                          &    \\
1      & W_{1}                      &        W_{1}        &        &                                                                          &    \\
\vdots & \vdots                     &         \vdots      &        &                                                                          &        \\
r-1    & W_{r-1}                    &         W_{r-1}     &        &                                                                          & W_{r-1}=0    \\
\hline
r      & W_{r}                      &         W_{r}       & \oplus &         W_{r}=0                                                          & W_{r}    \\
\hline
r+1    &                            &                     &        &         W_{r+1}                                                          & W_{r+1}    \\
\vdots &                            &                     &        &         \vdots                                                           & \vdots   \\
2r-1   &                            &                     &        &         W_{2r-1}                                                         & W_{2r-1}    \\
2r     &                            &                     &        &         W_{2r}                                                           & W_{2r}    \\
\hline
       & H^{r}(X)                   & H^{r}(X)            & \oplus & R^{r}(X_{reg},L)                                                         & H^{r}(X_{reg})  \\
\hline  
\end{array}  
\)}

\noindent\makebox[\textwidth]{
\(
\begin{array}{|c||c : c | c c c| c|} 
\hline
\multicolumn{7}{|c|}{\text{Third case : \,}\overline{p} > \overline{m} = n-1} \\
\hline
       & 1 \leq r \leq n &  n  < r \leq p        &  \multicolumn{3}{c|}{ r=p+1}                                                                                & p+1<r    \\
\hline
-1     & 0               &                       &                    &        &                                                                              &    \\
0      & W_{0}           &                       &                     &        &                                                                              &    \\
1      & W_{1}           &                       &                     &        &                                                                              &    \\
\vdots & \vdots          &                       &                     &        &                                                                              &        \\
r-1    & W_{r-1}         &       W_{r-1}=0       &            W_{r-1}=0&        &                                                                              & W_{r-1}=0    \\
\hline
r      & W_{r}           &       W_{r}           &            W_{r}    & \oplus &         W_{r}=0                                                              & W_{r}    \\
\hline
r+1    &                 &                       &                     &        &         W_{r+1}                                                              & W_{r+1}    \\
\vdots &                 &                       &                     &        &         \vdots                                                               & \vdots   \\
2r-1   &                 &                       &                     &        &         W_{2r-1}                                                             & W_{2r-1}    \\
2r     &                 &                       &                     &        &         W_{2r}                                                               & W_{2r}    \\
\hline
       & H^{r}(X)        & H^{r}(X)              & H^{r}(X)            & \oplus & R^{r}(X_{reg},L)                                                             & H^{r}(X_{reg})  \\
\hline
\end{array}  
\)}

\subsection{Coperverse algebras and their homotopy theory}
\subsubsection{Coperverse algebras}

Let $\mathbf{k}$ be a fixed field of characteristic zero.

\begin{defi}
\label{def:coperverse_cdga}
A $n$-coperverse commutative differential graded algebra over $\mathbf{k}$, coperverse cdga for short, is a functor
\[
A_{\overline{\bullet}} \col \Pos{n}^{op} \longrightarrow \mathsf{CDGA}_{\mathbf{k}}.
\]

That is for all perversities $\overline{p} \in \Pos{n}^{op}$, $A_{\overline{p}}$ is a bigraded $\mathbf{k}$-algebra $(A_{\overline{p}}^{k})_{k \in \mathbf{N}}$, together with a linear differential $d \col A_{\overline{p}}^{k} \rightarrow A_{\overline{p}}^{k+1}$ and an associative product $\mu \col A_{\overline{p}}^{i} \times A_{\overline{p}}^{j} \rightarrow A_{\overline{p}}^{i+j}$.

We assume that products and differentials satisfy graded commutativity, Leibniz rules, and are compatible with poset maps. That is for every $\overline{p} \leq \overline{q}$ in $\Pos{n}^{op}$ we have the following commutative diagrams.

\[
\begin{tikzpicture}
\matrix (m)[matrix of math nodes, row sep=3em, column sep=5em, text height=2ex, text depth=0.25ex]
{  A_{\overline{p}} \times A_{\overline{p}}    &  A_{\overline{p}}  \\
   A_{\overline{q}} \times A_{\overline{q}}    &  A_{\overline{q}} \\};

\path[->]
(m-1-1) edge node[auto] {$\mu$} (m-1-2);
\path[->]
(m-2-1) edge node[auto,swap] {$\mu$} (m-2-2);
\path[->]
(m-1-1) edge node[auto,swap] {$(\pomap{p}{q}, \pomap{p}{q})$} (m-2-1);
\path[->]
(m-1-2) edge node[auto] {$\pomap{p}{q}$} (m-2-2);
\end{tikzpicture}
\quad
\begin{tikzpicture}
\matrix (m)[matrix of math nodes, row sep=3em, column sep=5em, text height=2ex, text depth=0.25ex]
{  A_{\overline{p}}     &  A_{\overline{p}}  \\
   A_{\overline{q}}     &  A_{\overline{q}} \\};

\path[->]
(m-1-1) edge node[auto] {$d$} (m-1-2);
\path[->]
(m-2-1) edge node[auto,swap] {$d$} (m-2-2);
\path[->]
(m-1-1) edge node[auto,swap] {$\pomap{p}{q}$} (m-2-1);
\path[->]
(m-1-2) edge node[auto] {$\pomap{p}{q}$} (m-2-2);
\end{tikzpicture}
\]

We denote by $H_{\overline{\bullet}}(A, \mathbf{k}) := H(A_{\overline{\bullet}},d)$.
\end{defi}

We denote by $\cdga{k}$ the category of coperverse cdga's over $\mathbf{k}$. 

Note that with this definition, we have an extended product over the whole family $(A_{\overline{p}})_{\overline{p} \in \Pos{n}^{op}}$. Indeed, for every $\overline{p} \leq \overline{q}$ in $\Pos{n}^{op}$, denote by $\mu_{\overline{p},\overline{q}}$ the following composition

\[
\mu_{\overline{p},\overline{q}} \col A_{\overline{p}} \times A_{\overline{q}} \overset{(\pomap{p}{q}, \mathrm{id})}{\longrightarrow} A_{\overline{q}} \times A_{\overline{q}} \overset{\mu}{\longrightarrow} A_{\overline{q}}.
\]

\begin{defi}
The map $\mu_{\overline{\bullet},\overline{\bullet}}$ defined for all $\overline{p} \leq \overline{q}$ in $\mathcal{P}^{op}$ by the above composition is called the extended product over the family $(A_{\overline{p}})_{\overline{p} \in \mathcal{P}^{op}}$.
\end{defi}

\begin{remark}
\begin{enumerate}
\item 
The following diagram, where $T$ is the twist isomorphism $T(a,b) := (-1)^{|a|\cdot|b|}(b,a)$, commutes. Because of that and for the sake of simplicity, we will then adopt the following convention. Each time a product $A_{\overline{p}} \times \cdots \times A_{\overline{q}}$ will appear, we will consider that the perversities are put in order, that is $\overline{p} \leq \cdots \leq \overline{q}$ in $\Pos{n}^{op}$.

\[
\begin{tikzpicture}
\matrix (m)[matrix of math nodes, row sep=3em, column sep=5em, text height=2ex, text depth=0.25ex]
{ A_{\overline{p}} \times A_{\overline{q}}  &  A_{\overline{q}} \times A_{\overline{q}} & A_{\overline{q}}\\
  A_{\overline{q}} \times A_{\overline{p}}  &  A_{\overline{q}} \times A_{\overline{q}} & \\};

\path[->]
(m-1-1) edge node[auto] {$(\pomap{p}{q}, \mathrm{id})$} (m-1-2);
\path[->]
(m-2-1) edge node[auto] {$(\mathrm{id},\pomap{p}{q})$} (m-2-2);
\path[->]
(m-1-1) edge node[auto,swap] {$T$} (m-2-1);
\path[->]
(m-1-2) edge node[auto] {$T$} (m-2-2);
\path[->]
(m-1-2) edge node[auto] {$\mu$} (m-1-3);
\path[->]
(m-2-2) edge node[auto,swap] {$\mu$} (m-1-3);
\end{tikzpicture}
\]

\item The extended product $\mu_{\overline{\bullet},\overline{\bullet}}$ verifies Leibniz rule, is associative and compatible with poset maps and morphisms of coperverse algebras. That is all $\overline{p} \leq \overline{q} \leq \overline{r}$ in $\Pos{n}^{op}$ we have the commutative diagram,

\[
\begin{tikzpicture}
\matrix (m)[matrix of math nodes, row sep=3em, column sep=5em, text height=2ex, text depth=0.25ex]
{  A_{\overline{p}} \times A_{\overline{q}} \times A_{\overline{r}}    &  A_{\overline{p}} \times A_{\overline{r}}  \\
   A_{\overline{q}} \times A_{\overline{r}}                            &  A_{\overline{r}} \\};

\path[->]
(m-1-1) edge node[auto] {$(\mathrm{id}, \mu_{\overline{q},\overline{r}})$} (m-1-2);
\path[->]
(m-1-1) edge node[auto,swap] {$(\mu_{\overline{p},\overline{q}}, \mathrm{id})$} (m-2-1);
\path[->]
(m-2-1) edge node[auto,swap] {$\mu_{\overline{q},\overline{r}}$} (m-2-2);
\path[->]
(m-1-2) edge node[auto] {$\mu_{\overline{p},\overline{r}}$} (m-2-2);
\end{tikzpicture}
\]
and for all $\overline{p_{1}} \leq \overline{p_{2}} \leq \overline{q_{1}} \leq \overline{q_{2}}$ in $\Pos{n}^{op}$ we have the commutative diagram.
\[
\begin{tikzpicture}
\matrix (m)[matrix of math nodes, row sep=3em, column sep=5em, text height=2ex, text depth=0.25ex]
{  A_{\overline{p_{1}}} \times A_{\overline{q_{1}}} & A_{\overline{q_{1}}} \\
   A_{\overline{p_{2}}} \times A_{\overline{q_{2}}} & A_{\overline{q_{2}}} \\};

\path[->]
(m-1-1) edge node[auto] {$\mu_{\overline{p_{1}},\overline{q_{1}}}$} (m-1-2);
\path[->]
(m-1-1) edge node[auto,swap] {$\pomap{p_{1}}{p_{2}} \times \pomap{q_{1}}{q_{2}}$} (m-2-1);
\path[->]
(m-2-1) edge node[auto,swap] {$\mu_{\overline{p_{2}},\overline{q_{2}}}$} (m-2-2);
\path[->]
(m-1-2) edge node[auto] {$\pomap{q_{1}}{q_{2}}$} (m-2-2);
\end{tikzpicture}
\]
\end{enumerate}
\end{remark}

Since $\mu_{\overline{p},\overline{p}} = \mu$ for all $\overline{p}$ we will always consider the family $(A_{\overline{p}})_{\overline{p} \in \Pos{n}^{op}}$ endowed with the extended product. We then denote a coperverse cdga by $(A_{\overline{\bullet}}, \mu_{\overline{\bullet},\overline{\bullet}})$.

\subsubsection{Homotopy theory of coperverse algebras}

We now define a model structure on the category of coperverse cdga's by using the formalism of Reedy categories. The definitions and results involving Reedy categories can be found in \cite{Hovey1999}.

First, recall the model structure of $\mathsf{CDGA}_{\mathbf{k}}$. The projective model structure on $\mathsf{CDGA}_{\mathbf{k}}$ is given by the following
\begin{itemize}
\item the weak equivalences are the quasi-isomorphims,
\item the fibrations are the degreewise surjections,
\item the cofibrations are the retracts of relative Sullivan algebras.
\end{itemize}

For $n \in \mathbf{N}$, consider the semifree dga's
\[
S(n) := (\wedge \mathbf{k}[n],d=0)
\]
where $\mathbf{k}[n]$ denotes the graded vector space which is $\mathbf{k}$ in degree $n$ and $0$ otherwise. For $n \geq 1$, consider the semifree dga's
\[
D(n) := 
\begin{cases}
0 & n=0, \\
(\wedge(\mathbf{k}[n+1] \oplus \mathbf{k}[n]),d=0) & n>0
\end{cases}
\]
and write 
\[
i_{n} \col S(n) \rightarrow D(n)
\]
for the morphism that send the generator of degree $n$ to the generator of degree $n$. If $n=0$ then this is the unique morphism $0 \rightarrow 0$, and for $n>0$
\[
j_{n} \col 0 \rightarrow D(n).
\]

\begin{propo}
\label{prop:gen_cofib_cdga}
The sets $I := \lbrace i_{n} \rbrace_{n} \cup \lbrace S(0) \rightarrow 0\rbrace$, and $J:= \lbrace j_{n} \rbrace_{n>0}$ are the sets of generating cofibrations and acyclic cofibrations, respectively, of $\mathsf{CDGA}_{\mathbf{k}}$. The category $\mathsf{CDGA}_{\mathbf{k}}$ is then cofibrantly generated.
\end{propo}

Before talking about Reedy categories, note that we have an exact evaluation functor
\[
Ev_{\overline{p}} \col \cdga{k} \longrightarrow \mathsf{CDGA}_{\mathbf{k}}
\]
that send $A_{\overline{\bullet}}$ to $A_{\overline{p}}$, this functor admits an exact left adjoint $F_{\overline{p}}$ defined by $F_{\overline{p}}(A)_{\overline{q}} = A$ if $\overline{p} \leq \overline{q}$ and zero otherwise.

\begin{defi}
Let $\mathcal{C}$ be a small category and $\mathcal{C}' \subset \mathcal{C}$ a subcategory. The subcategory $\mathcal{C}'$ is said to be a lluf subcategory if the objects of $\mathcal{C}'$ and $\mathcal{C}$ are the same.
\end{defi}

\begin{defi}[Reedy category]
Let $\mathcal{C}$ be a small category together with a degree function $\deg \col \mathcal{C} \longrightarrow \mathbf{N}$ defined on the objects and suppose that we have two lluf subcategories $\overrightarrow{\mathcal{C}}$ and $\overleftarrow{\mathcal{C}}$. We say that $(\mathcal{C}, \overrightarrow{\mathcal{C}}, \overleftarrow{\mathcal{C}})$ is a Reedy category if the two following conditions are satisfied.
\begin{enumerate}
\item If $\alpha \col c \rightarrow c'$ is a non-identity map in $\overrightarrow{\mathcal{C}}$ (resp. in $\overleftarrow{\mathcal{C}}$) then $\deg (c) < \deg(c')$ (resp. $\deg (c) > \deg(c')$).
\item Every map $\alpha$ in $\mathcal{C}$ has a unique factorization
\[
\begin{cases}
\alpha & = \overrightarrow{\alpha} \circ \overleftarrow{\alpha}, \\
\overrightarrow{\alpha} & \in \overrightarrow{\mathcal{C}}, \\
\overleftarrow{\alpha} & \in \overleftarrow{\mathcal{C}}. 
\end{cases}
\]
\end{enumerate}
\end{defi}

\begin{exmp}
\begin{enumerate}
\item A discrete category $\mathcal{C}$, that is a category where $\mathcal{C}(x,y) = \lbrace \mathrm{id}_{x} \rbrace$ if and only if $x=y$ and the empty set otherwise, is a Reedy category where all the objects are of degree 0.
\item \label{exmp:Poset_reedy_cat} Let $\mathcal{P}$ be a finite poset. We define every minimal element to be of degree 0 and we define the degree of an element $p \in \mathcal{P}$ to be the length of the longest path of non-identity maps from an element of degree zero to $p$. If we have $p \rightarrow p'$ with $p \neq p'$ then necessarily we have $\deg p < \deg p'$. The poset $\mathcal{P}$ is then endowed with a structure of Reedy category with
\[
\begin{cases}
\overrightarrow{\mathcal{P}}& = \mathcal{P}, \\
\overleftarrow{\mathcal{P}} & = \mathrm{Disc}(\mathcal{P}). 
\end{cases}
\]
where $\mathrm{Disc}(\mathcal{P})$ is the discrete category underlying the poset $\mathcal{P}$, every elements of $\mathrm{Disc}(\mathcal{P})$ are of degree 0.
\end{enumerate}
\end{exmp}

For every Reedy category $\mathcal{C}$ there exist subcategories $\mathcal{C}_{<k}$ of objects of degree strictly inferior to $n$. Let then $F \col \mathcal{C} \rightarrow \mathcal{M}$ a functor which we suppose covariant, consider $c \in \mathcal{C}$ with $\deg c=n$, we have the two objects and maps
\[
L^{c}F \overset{\ell_{c}}{\longrightarrow} F(c) \overset{m_{c}}{\longrightarrow} M^{c}X,
\]
where
\[
L^{c}F := \mathrm{colim}(\partial(\overrightarrow{\mathcal{C}}_{<k}/c) \overset{U_{c}}{\longrightarrow} \mathcal{C} \overset{F}{\longrightarrow} \mathcal{M}),
\] 
\[
M^{c}F := \lim(\partial(c/\overleftarrow{\mathcal{C}}_{<k}) \overset{U_{c}}{\longrightarrow} \mathcal{C} \overset{F}{\longrightarrow} \mathcal{M}).
\] 
with $\partial(\overrightarrow{\mathcal{C}}_{<k}/c)$ and $\partial(c/\overleftarrow{\mathcal{C}}_{<k})$ are the two full subcategories of respectively $\overrightarrow{\mathcal{C}}_{<k}/c$ and $c/\overleftarrow{\mathcal{C}}_{<k}$ where we have removed the identity object $c \rightarrow c$. 

\begin{defi}
The objects $L^{c}F$ and $M^{c}F$ are respectively called the $c$-th latching and $c$-th matching objects. The maps $\ell_{c}$ and $m_{c}$ are then the $c$-th latching and $c$-th matching maps.
\end{defi}

Given a map $F \rightarrow G$ in $\mathsf{Fun}(\mathcal{C}, \mathcal{M})$, we define the $c$-th relative latching map by the following diagram of pushout

\[
\begin{tikzpicture}
\matrix (m)[matrix of math nodes, row sep=3em, column sep=5em, text height=2ex, text depth=0.25ex]
{ L^{c}F  & F(c)  &      \\
  L^{c}G  & \cdot &      \\ 
          &       & G(c) \\};

\path[->]
(m-1-1) edge node[auto] {} (m-1-2);
\path[->]
(m-1-1) edge node[auto] {} (m-2-1);
\path[->]
(m-1-2) edge node[below=1em, left=1em] {$\lrcorner$} (m-2-2);
\path[->]
(m-2-1) edge node[auto] {} (m-2-2);
\path[->]
(m-1-2) edge node[auto] {} (m-3-3);
\path[->]
(m-2-1) edge node[auto] {} (m-3-3);
\path[->,dashed]
(m-2-2) edge node[auto] {} (m-3-3);
\end{tikzpicture}
\]

\noindent and the $c$-th relative matching map by the following diagram of pullback

\[
\begin{tikzpicture}
\matrix (m)[matrix of math nodes, row sep=3em, column sep=5em, text height=2ex, text depth=0.25ex]
{ F(c)    &       &      \\
          & \cdot & M^{c}F     \\ 
          & G(c)  & M^{c}G \\};

\path[->,dashed]
(m-1-1) edge node[auto] {} (m-2-2);
\path[->]
(m-1-1) edge node[auto] {} (m-2-3);
\path[->]
(m-1-1) edge node[auto] {} (m-3-2);
\path[->]
(m-2-2) edge node[auto] {} (m-2-3);
\path[->]
(m-2-2) edge node[above=1em, right=1em] {$\ulcorner$} (m-3-2);
\path[->]
(m-3-2) edge node[auto] {} (m-3-3);
\path[->]
(m-2-3) edge node[auto] {} (m-3-3);
\end{tikzpicture}
\]

\begin{thm}[\cite{Hovey1999}, 5.2.5]
Let $\mathcal{M}$ be a model category et let $\mathcal{C}$ be a Reedy category. Then there is a model category on $\mathsf{Fun}(\mathcal{C}, \mathcal{M})$ such that :
\begin{enumerate}
\item the weak equivalences are defined pointwise,
\item the cofibrations are the maps $F \rightarrow G$ such that each relative latching map
\[
L^{c}G \coprod_{L^{c}F} F(c) \longrightarrow G(c)
\]
is a cofibration in $\mathcal{M}$,
\item the fibrations are the maps $F \rightarrow G$ such that each relative matching map
\[
F(c) \longrightarrow G(c) \times_{M^{c}G} M^{c}F
\]
is a fibration in $\mathcal{M}$.
\end{enumerate}
\end{thm}

We now apply this result to our context. We endow $\Pos{n}^{op}$ with the structure of a Reedy category defined in the item \ref{exmp:Poset_reedy_cat} of the last example.

Let $A_{\overline{\bullet}} \col \Pos{n}^{op} \rightarrow \mathsf{CDGA}_{\mathbf{k}}$ be a coperverse cdga and $\overline{p} \in \Pos{n}^{op}$ such that $\deg \overline{p} = k$. We have
\[
L^{\overline{p}}A_{\overline{\bullet}} := \mathrm{colim}(\partial(\mathcal{P}^{op}_{<k}/\overline{p}) \overset{U_{\overline{p}}}{\longrightarrow} \mathcal{P}^{op} \overset{A_{\overline{\bullet}}}{\longrightarrow} \mathcal{M}) = \mathrm{colim}_{\overline{p} < \overline{q}} A_{\overline{q}}
\] 
and
\[
M^{\overline{p}}A_{\overline{\bullet}} := \lim(\partial(\overline{p}/\mathrm{Disc}(\mathcal{P})) \overset{U_{\overline{p}}}{\longrightarrow} \mathcal{P}^{op} \overset{A_{\overline{\bullet}}}{\longrightarrow} \mathcal{M}) = 0.
\] 

Computing the relative latching and matching map we get the following result

\begin{thm}
\label{thm:model_structure}
The category $\cdga{k}$ has a structure of a cofibrantly generated model category which we call the projective model structure. In this model category, the weak equivalences are the quasi-isomorphisms and the fibrations are the surjections.
\end{thm}

\begin{proof}
The computations of weak equivalences and fibrations are clear.

The fact that $\cdga{k}$ is cofibrantly generated comes from \cite[Remark 5.1.8]{Hovey1999}, the generating cofibrations are the $\lbrace F_{\overline{p}}(i) \rbrace_{i \in I, \overline{p} \in \Pos{n}^{op}}$ and the generating acyclic cofibrations are the $\lbrace F_{\overline{p}}(j) \rbrace_{i \in J, \overline{p} \in \Pos{n}^{op}}$ where $I$ and $J$ are the sets defined in the proposition \ref{prop:gen_cofib_cdga}.
\end{proof}

For clarity, we give the following definition as a result of the previous theorem.

\begin{defi}
Let $f_{\overline{\bullet}} \col A_{\overline{\bullet}} \rightarrow B_{\overline{\bullet}}$ be a morphism of coperverse algebras. The morphism $f_{\overline{\bullet}}$ is
\begin{enumerate}
\item A quasi-isomorphism if, for every perversity $\overline{p} \in \Pos{n}^{op}$, the induced map $H^{\ast}_{\overline{p}}(A) \rightarrow H^{\ast}_{\overline{p}}(B)$ is an isomorphism.
\item A fibration if, for every perversity $\overline{p} \in \Pos{n}^{op}$, the induced map $f_{\overline{p}} \col A_{\overline{p}} \rightarrow B_{\overline{p}}$ is a degreewise surjection.
\end{enumerate}
\end{defi}

We denote by $\Ho{\cdga{\mathbf{k}}}$ the homotopy category associated to the model category structure on $\cdga{\mathbf{k}}$. That is the category defined by formally inverting quasi-isomorphisms.

\begin{remark}
There are many ways to put a model structures on $\cdga{k}$. Indeed the category $\mathsf{CDGA}_{\mathbf{k}}$ also has an injective model structure where the weak equivalences are the quasi-isomorphisms and the cofibrations are the injections and we could have choose this model structure to do the computations. 

On the other hand we could have chose the projective or injective model structure on $\cdga{k}$ coming from $\mathsf{CDGA}_{\mathbf{k}}$ rather than doing computations using Reedy categories. But since $\mathsf{CDGA}_{\mathbf{k}}$ is a combinatorial model category all the ways mentioned above are guaranteed to be Quillen equivalent to the projective model structure on $\cdga{k}$. 

By the way, all these model structures share the same weak equivalences.
\end{remark}

\section{Coperverse rational models}
\label{section:coperv_algebras}
\subsection{Coperverse cdga's associated with a morphism of cdga's}

The tools in this section are modified versions of the one appearing the work of Chataur and Cirici \cite{Chataur} on the interactions between intersection cohomology and mixed Hodge structures. 

Let $(A,d) \in \mathsf{CDGA}_{\mathbf{k}}$. We denote by $\mathbf{k}(t, dt) := \wedge (t,dt)$ the free cdga generated by $t$ and $dt$ with $\deg t =0$, $\deg dt =1$ and $d(t)=dt$.

\begin{defi}
\label{def:At,dt}
We denote by $A(t,dt) := A \otimes_{\mathbf{k}} \mathbf{k}(t, dt)$. For $\lambda \in \mathbf{k}$ we also define the evaluation map
\[
\delta_{\lambda} \col A(t,dt) \longrightarrow A
\]
by $\delta_{\lambda}(t) = \lambda$ and $\delta_{\lambda}(dt)=0$.
\end{defi}

For all $r \geq 0$, we have the following short exact sequence
\[
0 \longrightarrow \ker d^{r} \longrightarrow A^{r} \longrightarrow \mathrm{Coim\,}d^{r} \longrightarrow 0
\]
where $\mathrm{Coim\,}d^{r} := A^{r} / \ker d^{r}$. Denote by $s_{r} \col \mathrm{Coim\,}d^{r} \rightarrow A^{r}$ a choice of section. For all $r \geq 0$, we denote by $\coim{r} := \im{s_{r}}$, the differential $d^{r}$ induces the isomorphism $\coim{r} \rightarrow \im{d^{r}}$.

\begin{defi}
The unital $\overline{p}$-cotruncation of $A(t,dt)$ is defined by
\[
\dgaucotr{p}{A(t,dt)} := A^{0} \oplus \dgacotr{p}{A(t,dt)}.
\]
where $\dgacotr{p}{A(t,dt)}$ is defined by 
\[
\dgacotr{p}{A(t,dt)}^{r} :=
\begin{cases}
A^{r}\otimes \mathbf{k}[t]t \oplus A^{r-1}\otimes \mathbf{k}[t]dt & r < p \\
A^{p-1}\otimes \mathbf{k}[t]dt \oplus A^{p} \otimes \mathbf{k}[t]t \oplus \coim{p} & r=p \\
A^{r-1}\otimes \mathbf{k}[t]dt \oplus A^{r}\otimes \mathbf{k}[t] & r> p
\end{cases}
\]
\end{defi}

\begin{lem}
\label{lem:coperv_alg}
$\dgaucotr{\bullet}{A(t,dt)}$ is a coperverse cdga.
\end{lem}

\begin{proof}
Consider first $\dgacotr{p}{A(t,dt)}$.

The compatibility of $\dgacotr{\bullet}{A(t,dt)}$ with the differential $d(\dgacotr{p}{A(t,dt)}) \subset \dgacotr{p}{A(t,dt)}$ and product $\dgacotr{p}{A(t,dt)} \times \dgacotr{p}{A(t,dt)} \rightarrow \dgacotr{p}{A(t,dt)}$ is clear by construction. We detail the compatibility with the poset maps. By unicity of the maps $\pomap{p}{q}$, every  $\pomap{p}{q}$ is a composition of poset maps $\pomap{k+1}{k}$ so we only detail these ones. We have 

\[
\dgacotr{k+1}{A(t,dt)}^{r} :=
\begin{cases}
A^{r}  \otimes \mathbf{k}[t]t  \oplus A^{r-1} \otimes \mathbf{k}[t]dt & r < k+1 \\
A^{k}  \otimes \mathbf{k}[t]dt \oplus A^{k+1} \otimes \mathbf{k}[t]t \oplus \coim{k+1} & r=k+1 \\
A^{r-1}\otimes \mathbf{k}[t]dt \oplus A^{r}   \otimes \mathbf{k}[t] & r> k+1
\end{cases}
\]

and 

\[
\dgacotr{k}{A(t,dt)}^{r} :=
\begin{cases}
A^{r}  \otimes \mathbf{k}[t]t  \oplus A^{r-1} \otimes \mathbf{k}[t]dt & r < k \\
A^{k-1}\otimes \mathbf{k}[t]dt \oplus A^{k}   \otimes \mathbf{k}[t]t \oplus \coim{k} & r=k \\
A^{r-1}\otimes \mathbf{k}[t]dt \oplus A^{r}   \otimes \mathbf{k}[t] & r> k.
\end{cases}
\]

For $r \leq k$ or $r >k+1$, $\pomap{k+1}{k}$ is the identity map. For $r=k+1$, since  $A^{k+1}\otimes \mathbf{k}[t] = A^{k+1} \oplus A^{k+1}\otimes \mathbf{k}[t]t$, $\pomap{k+1}{k}$ is an injection. 

Now for $\dgaucotr{p}{A(t,dt)} := A^{0} \oplus \dgacotr{p}{A(t,dt)}$ the compatibility with the differential and the poset maps is clear by the same arguments than above. The product $\dgaucotr{p}{A(t,dt)} \times \dgaucotr{p}{A(t,dt)} \rightarrow \dgaucotr{p}{A(t,dt)}$ is also clear by construction.
\end{proof}

Let now $f \col A \longrightarrow B$ be a morphism of cdga's. Given a perversity $\overline{p} \in \Pos{n}^{op}$, we consider the following pull-back diagram in the category $\CDGA{k}$.

\[
\begin{tikzpicture}
\matrix (m)[matrix of math nodes, row sep=3em, column sep=5em, text height=2ex, text depth=0.25ex]
{ \pb{p}{f}  &  \dgaucotr{p}{B(t,dt)}  \\
  A  &  B \\};

\path[->]
(m-1-1) edge node[auto] {} (m-1-2);
\path[->]
(m-2-1) edge node[auto,swap] {$f$} (m-2-2);
\path[->]
(m-1-1) edge node[above=1em, right=1em] {$\ulcorner$} (m-2-1);
\path[->]
(m-1-2) edge node[auto] {$\delta_{1}$} (m-2-2);
\end{tikzpicture}
\]

With the product, the differential defined component-wise and the compatibility with poset maps.

\begin{propo}
The pull-back $\pb{\bullet}{f}$ is a  coperverse cdga.
\end{propo} 

\begin{defi}
\label{def:pb_coperv_cdga}
$\pb{\bullet}{f}$ is the coperverse cdga associated to the morphism of cdga's $f \col A \longrightarrow B$.
\end{defi}

\begin{defi}\label{def:sharp_pcdga}
Let $(A_{\overline{\bullet}}, \mu_{\overline{\bullet},\overline{\bullet}})$ be a coperverse cdga and $r \in \mathbf{Z}$. We say that $(A_{\overline{\bullet}}, \mu_{\overline{\bullet},\overline{\bullet}})$  is a $r$-sharp coperverse cdga if the product satisfies the two following conditions
\begin{enumerate}
\item \textbf{Unity} For $A_{\overline{p}}^{i} \times A_{\overline{0}}^{j} \rightarrow  A_{\overline{0}}^{i+j}$ the product lifts to
\[
\begin{tikzpicture}
\matrix (m)[matrix of math nodes, row sep=3em, column sep=5em, text height=2ex, text depth=0.25ex]
{                                                  &  A_{\overline{p}}^{i+j}  \\
  A_{\overline{p}}^{i} \times A_{\overline{0}}^{j} &  A_{\overline{0}}^{i+j}  \\};

\path[->]
(m-2-1) edge node[auto] {$\mu_{\overline{p},\overline{0}}$} (m-2-2);
\path[->,dashed]
(m-2-1) edge node[auto] {} (m-1-2);
\path[->]
(m-1-2) edge node[auto] {$\pomap{p}{0}$} (m-2-2);
\end{tikzpicture}
\]
\item \textbf{Factorization} For $\overline{p}, \overline{q} \neq \overline{0}$ and $i,j \neq 0$ the product lifts to
\[
\begin{tikzpicture}
\matrix (m)[matrix of math nodes, row sep=3em, column sep=5em, text height=2ex, text depth=0.25ex]
{                                                  &  A_{\overline{p+q+r}}^{i+j}  \\
  A_{\overline{p}}^{i} \times A_{\overline{q}}^{j} &  A_{\overline{q}}^{i+j}  \\};

\path[->]
(m-2-1) edge node[auto] {$\mu_{\overline{p},\overline{q}}$} (m-2-2);
\path[->,dashed]
(m-2-1) edge node[auto] {} (m-1-2);
\path[->]
(m-1-2) edge node[auto] {$\pomap{p+q+r}{q}$} (m-2-2);
\end{tikzpicture}
\]
\end{enumerate}
We assume that this lift satisfies all the properties of the product $\mu$. That is Leibniz rule with respect to the differential, graded commutativity and compatibility with poset maps and morphisms of cdga's.
\end{defi}

\begin{lem}
$\dgaucotr{\bullet}{A(t,dt)}$ is a $(-1)$-sharp coperverse cdga.
\end{lem}

\begin{cor}\label{cor:J(f)_-1_pcdga}
Let $f \col A \longrightarrow B$ be a morphism of cdga's, then $\pb{\bullet}{f}$ is a $(-1)$-sharp coperverse cdga
\end{cor}

\begin{remark}
\begin{enumerate}
\item The first condition means that the final cdga $A_{\overline{0}}$, since $\overline{0}$ is the maximal element of $\Pos{n}^{op}$, plays the role of the unit for the family $(A_{\overline{p}})_{\overline{p} \in \Pos{n}^{op}}$ and in particular for the unit $\eta_{\overline{0}} \col \mathbf{k} \rightarrow A_{\overline{0}}^{0}$ we have $A_{\overline{p}}^{i} \times A_{\overline{0}}^{0} \rightarrow A_{\overline{p}}^{i}$ for every $\overline{p}$ and every $i \geq 0$.

\item  coperverse cdga's are meant to model the rational cohomology of intersection spaces $\HI{p}{k}{X}$. Since the $\I{p}{X}$ are topological spaces their cohomology bear an inner cup-product which is reflected in the definition of the coperverse cdga's. The lift is here to show the interactions between the different $\HI{p}{k}{X}$.
\end{enumerate}
\end{remark}

\subsection{Coperverse rational model of intersection spaces}

Let $X \in \super$ of complex dimension $n$, we denote by $\Sigma$ the singular locus of $X$. 

Let $T$ be a closed algebraic neighbourhood of the singular locus in $X$ such that the inclusion $\Sigma \subset T$ is a homotopy equivalence. Such a neighbourhood exists and is constructed with "rug functions", see \cite[p.144]{Peters2007} or \cite{Durfee1983}.

The link $L:=L(\Sigma,X)$ of $\Sigma$ in $X$ is defined by $L := \partial T \simeq T^{\ast} := T - \Sigma$. The inclusion $i \col L \hookrightarrow X_{reg}$ of the link into the regular part of $X$ induces a morphism of cdga's over $\mathbf{Q}$
\[
i^{\ast} \col \Apl{X_{reg}} \longrightarrow \Apl{L}.
\]

Let $\overline{p} \in \Pos{n}^{op}$ be a perversity, the rational model of the intersection space $\I{p}{X}$ is given by $\AI{p}{X} := \pb{p}{i^{\ast}}$, which is the following pull-back diagram, see \cite{Klimczak2015}.

\[
\begin{tikzpicture}
\matrix (m)[matrix of math nodes, row sep=3em, column sep=5em, text height=2ex, text depth=0.25ex]
{ \pb{p}{i^{\ast}}  &  \dgaucotr{p}{\Apl{L}(t,dt)}  \\
  \Apl{X_{reg}}  &  \Apl{L} \\};

\path[->]
(m-1-1) edge node[auto] {} (m-1-2);
\path[->]
(m-2-1) edge node[auto] {$i^{\ast}$} (m-2-2);
\path[->]
(m-1-1) edge node[above=1em, right=1em] {$\ulcorner$} (m-2-1);
\path[->]
(m-1-2) edge node[auto] {$\delta_{1}$} (m-2-2);
\end{tikzpicture}
\]

\begin{defi}
The coperverse cdga $\AI{\bullet}{X}$ is called the coperverse rational model of the intersection spaces $\I{\bullet}{X}$.
\end{defi}

If $A_{\overline{\bullet}}$ is a coperverse cdga, its cohomology is also a coperverse cdga. We then have the following proposition.
\begin{propo}
$\HI{\bullet}{\ast}{X}$ is a coperverse cdga.
\end{propo}

We have an isomorphism of coperverse cdga $H^{\ast}(\AI{p}{X}) \cong \HI{\bullet}{\ast}{X}$. 

This then defines a functor 
\[
AI_{\overline{\bullet}} \col \super \longrightarrow \Ho{\cdga{k}}.
\]

If we only consider the coperverse rational model of $X \in \super$, we then have that $AI_{\overline{\bullet}}(X)$ is a $(-1)$-sharp coperverse cdga by corollary \ref{cor:J(f)_-1_pcdga}. But if we only want to consider the cohomology coperverse algebra $\HI{\bullet}{\ast}{X}$, we can have an even sharper result.

\begin{propo}
Let $X \in \super$ with only isolated singularities. Then $(\HI{\bullet}{\ast}{X},0)$ is a 1-sharp coperverse cdga. That is we have
\[
\begin{cases}
\HI{0}{i}{X} \otimes \redHI{p}{j}{X} \longrightarrow \redHI{p}{i+j}{X} & \\
\redHI{p}{i}{X} \otimes \redHI{q}{j}{X} \longrightarrow \redHI{p+q+1}{i+j}{X} & p+q+1 \leq 2n-2.
\end{cases}
\]
\end{propo}

\begin{remark}
It is important to make a difference between the extended product $\mu_{\overline{\bullet}, \overline{\bullet}}$ and the property of sharpness. The existence of the extended product is a consequence of the definition \ref{def:coperverse_cdga} and as such every coperverse cdga defined in the same way naturally has an extended product.

The property of sharpness of our coperverse algebras defined in \ref{def:pb_coperv_cdga} is a consequence of our methods of construction. There might be coperverse algebras which do not have any property of sharpness, but still have an extended product.
\end{remark}

\section{Hodge Theory}
\label{section:Hodge_th}
\subsection{Coperverse mixed Hodge algebras}

We now put a mixed Hodge structure on the coperverse rational model of $X \in \super$.

\begin{defi}
A coperverse filtered cdga $(A_{\overline{\bullet}}, W)$ is a coperverse cdga $A_{\overline{\bullet}}$ together with a filtration $\lbrace W_{m}A_{\overline{\bullet}} \rbrace_{m \in \mathbf{Z}}$ such that
\begin{enumerate}
\item $W_{m-1}A_{\overline{p}} \subset W_{m}A_{\overline{p}}$ and $d(W_{m}A_{\overline{p}}) \subset W_{m}A_{\overline{p}}$, for all $m \in \mathbf{Z}$ and all $\overline{p} \in \Pos{n}$,
\item $W_{m}A_{\overline{p}}.W_{n}A_{\overline{p}} \subset W_{m+n}A_{\overline{p}}$,
\item $W_{m}A_{\overline{p}} \subset W_{m}A_{\overline{q}}$ for all $\overline{p} \leq \overline{q}$ in $\Pos{n}^{op}$,
\item The filtration $W$ is exhaustive and biregular : for all $n \geq 0$ and all $\overline{p} \in \mathcal{P}^{op}$ there exist integers $m$ and $l$ such that $W_{m}A^{n}_{\overline{p}} =0$ and $W_{l}A^{n}_{\overline{p}} = A^{n}_{\overline{p}}$.
\end{enumerate}
\end{defi}

\begin{defi}
A coperverse mixed Hodge cdga over $\mathbf{Q}$ is a coperverse filtered cdga $(A_{\overline{\bullet}}, W)$ with a filtration $F$ on $A_{\overline{\bullet}} \otimes \mathbf{C}$ such that for all $n \geq 0$ and all $\overline{p} \in \Pos{n}^{op}$, 
\begin{enumerate}
\item the triple $(A^{n}_{\overline{p}}, \Dec{W}, F)$ is a mixed Hodge structure,
\item the differential $d \col A_{\overline{p}}^{k} \rightarrow A_{\overline{p}}^{k+1}$, the product $\mu \col A_{\overline{p}}^{i} \times A_{\overline{p}}^{j} \rightarrow A_{\overline{p}}^{i+j}$ and the poset maps $\pomap{p}{q} \col A_{\overline{p}}^{k} \rightarrow A_{\overline{q}}^{k}$ are morphisms of mixed Hodge structures.
\end{enumerate}
The filtration $W$ is called the weight filtration and the filtration $F$ is called the Hodge filtration.
\end{defi}

We will denote, by an abuse of notations, such a mixed Hodge cdga by the triple $(A_{\overline{\bullet}}, W, F)$ with in mind the fact that $F$ is not defined on $A_{\overline{\bullet}}$ but on its complexification $A_{\overline{\bullet}} \otimes \mathbf{C}$. The filtration $\Dec{W}$ is the Deligne's décalage of the weight filtration defined in \cite[15]{Deligne1971} which is given by

\[
\Dec{W_{p}}A_{\overline{\bullet}}^{n} := W_{p-n}A_{\overline{\bullet}}^{n} \cap d^{-1}(W_{p-n-1}A_{\overline{\bullet}}^{n+1}).
\]

We denote by $\mhcdga{Q}$ the category of coperverse mixed Hodge cdga's over $\mathbf{Q}$.

\begin{lem}
Let $(A_{\overline{\bullet}}, W, F)$ be a coperverse mixed Hodge cdga, then the extended product $\mu_{\overline{\bullet},\overline{\bullet}}$ is a morphism of mixed Hodge structure.
\end{lem}

\begin{defi}
A coperverse filtered cdga $(A_{\overline{\bullet}}, W)$ is said to be $r$-sharp if $A_{\overline{\bullet}}$ is a filtered coperverse cdga such that the lift is compatible with the filtration $\lbrace W_{m}A_{\overline{\bullet}} \rbrace_{m \in \mathbf{Z}}$. That is we have the two following conditions

\begin{enumerate}
\item \textbf{Filtered unity} For $W_{m}A_{\overline{p}}^{i} \times W_{n}A_{\overline{0}}^{j} \rightarrow  W_{m+n}A_{\overline{0}}^{i+j}$ the product lifts to
\[
\begin{tikzpicture}
\matrix (m)[matrix of math nodes, row sep=3em, column sep=5em, text height=2ex, text depth=0.25ex]
{                                                            &  W_{m+n}A_{\overline{p}}^{i+j}  \\
  W_{m}A_{\overline{p}}^{i} \times W_{n}A_{\overline{0}}^{j} &  W_{m+n}A_{\overline{0}}^{i+j}  \\};

\path[->]
(m-2-1) edge node[auto] {$\mu_{\overline{p},\overline{0}}$} (m-2-2);
\path[->,dashed]
(m-2-1) edge node[auto] {} (m-1-2.south west);
\path[->]
(m-1-2) edge node[auto] {$\pomap{p}{0}$} (m-2-2);
\end{tikzpicture}
\]
\item \textbf{Filtered factorization} For $\overline{p}, \overline{q} \neq \overline{0}$ and $i,j \neq 0$ the product lifts to
\[
\begin{tikzpicture}
\matrix (m)[matrix of math nodes, row sep=3em, column sep=5em, text height=2ex, text depth=0.25ex]
{                                                            &  W_{m+n}A_{\overline{p+q+r}}^{i+j}  \\
  W_{m}A_{\overline{p}}^{i} \times W_{n}A_{\overline{q}}^{j} &  W_{m+n}A_{\overline{q}}^{i+j}  \\};
  
\path[->]
(m-2-1) edge node[auto] {$\mu_{\overline{p},\overline{q}}$} (m-2-2);
\path[->,dashed]
(m-2-1) edge node[auto] {} (m-1-2.south west);
\path[->]
(m-1-2) edge node[auto] {$\pomap{p+q+r}{q}$} (m-2-2);
\end{tikzpicture}
\]
\end{enumerate}
\end{defi}

\begin{defi}
A $r$-sharp coperverse mixed Hodge cdga over $\mathbf{Q}$ is a coperverse mixed Hodge cdga $(A_{\overline{\bullet}}, W,F)$ such that the lift is a morphism of mixed Hodge structure.
\end{defi}

Consider $\mathbf{Q}(t,dt)$ together with the \textit{bête} filtration $\sigma$, that is the multiplicative filtration with $t$ of weight $0$ and $dt$ of weight $-1$. We endow $\mathbf{C}(t,dt) := \mathbf{Q}(t,dt) \otimes \mathbf{C}$ with the bête filtration $\sigma$ and the trivial filtration $t$, that is decreasing filtration given by 
\[
0=t^{1}\mathbf{C}(t,dt) \subset t^{0}\mathbf{C}(t,dt) =\mathbf{C}(t,dt).
\]
Since $\Dec{\sigma} = t$ the triple $(\mathbf{Q}(t,dt), \sigma, t)$ is a mixed Hodge cdga.

Given another mixed Hogde cdga $(A,W,F)$, since the category of mixed Hodge structure is abelian the triple
\[
(A(t,dt), W \ast \sigma, F \ast t)
\]
is again a mixed Hodge cdga where the filtrations are defined by convolution. That is we have
\[
(W \ast \sigma)_{m}A(t,dt)^{n} := W_{m}A^{n} \otimes \mathbf{Q}[t] \oplus W_{m+1}A^{n-1} \otimes \mathbf{Q}[t]dt
\]
and
\[
(F \ast t)^{k}A(t,dt) := F^{k}A \otimes \mathbf{C}(t,dt).
\]
The evaluation map $\delta_{1}$ is strictly compatible with filtrations.

\begin{lem}
Let $(A,W,F)$ be a mixed Hodge cdga. Then for all $\overline{p} \in \Pos{n}^{op}$, $\dgaucotr{p}{A(t,dt)}$ is a $(-1)$-sharp mixed Hodge cdga.
\end{lem}

\begin{proof}
The triple $(A(t,dt), W \ast \sigma, F \ast t)$ is a mixed Hodge cdga, for all $\overline{p} \in \mathcal{P}^{op}$, $\dgaucotr{p}{A(t,dt)}$ is a sub-algebra with the filtrations induced by restriction. 

The differential is a morphism of mixed Hodge structure since the differential on $(A(t,dt), W \ast \sigma, F \ast t)$ is and $d(\dgaucotr{p}{A(t,dt)}) \subset \dgaucotr{p}{A(t,dt)}$.

The poset maps $\pomap{k+1}{k}$, $k \geq 0$, are the identity everywhere but at the cut-off degree $k+1$ where they are canonical inclusions, $\pomap{k+1}{k}$ in then compatible with both filtrations and by composition so are the $\pomap{p}{q}$.

The extended product $\dgaucotr{p}{A(t,dt)}^{i} \times \dgaucotr{q}{A(t,dt)}^{j}  \rightarrow \dgaucotr{q}{A(t,dt)}^{i+j}$ being defined as the composition of $\mu$ with poset maps $\pomap{p}{q}$, it is a morphism of mixed Hodge structure. 

The sharpness comes from the fact that $\dgaucotr{p}{A(t,dt)}$ is $(-1)$-sharp and that the product is a morphism of mixed Hodge structure.
\end{proof}

Let then $f \col (A, W, F) \rightarrow (B, W, F)$ be a morphism of mixed Hodge cdga. Since the category of mixed Hodge structures is abelian, see \cite[Theorem 2.3.5]{Deligne1971}, we have the following proposition.
\begin{propo}
The coperverse cdga $\pb{\bullet}{f}$ is a  coperverse mixed Hodge cdga. 
\end{propo} 

\subsection{Mixed Hodge structure on the coperverse rational model of the intersection spaces \texorpdfstring{$\I{\bullet}{X}$}{IpX}}

\begin{defi}[\cite{Cirici2014}]
A mixed Hodge diagram of cdga's over $\mathbf{Q}$ consists of a filtered cdga $(A_{\mathbf{Q}}, W)$ over $\mathbf{Q}$, a bifiltered cdga $(A_{\mathbf{C}}, W, F)$ over $\mathbf{C}$, together with a string of filtered $E_{1}$-quasi-isomorphisms from $(A_{\mathbf{Q}}, W)\otimes \mathbf{C}$ to $(A_{\mathbf{C}}, W)$. in addition, the following axioms must hold :
\begin{itemize}
\item The weight filtrations $W$ are regular and exhaustive. The Hodge filtration $F$ is biregular. The cohomology $H(A_{\mathbf{Q}})$ has finite type.
\item For all $p \in \mathbf{Z}$, the differential of $\gr{p}{W}{A_{\mathbf{C}}}$ is strictly compatible with $F$.
\item For all $n \geq 0$ and all $p \in \mathbf{Z}$, the filtration $F$ induced on $H^{n}(\gr{p}{W}{A_{\mathbf{C}}})$ defines a pure Hodge structure of weight $p+n$ on $H^{n}(\gr{p}{W}{A_{\mathbf{Q}}})$.
\end{itemize}
\end{defi}

Morphisms of mixed Hodge diagrams are defined by level-wise morphisms of bifiltered cdga's such that the associated diagram is strictly commutative. Forgetting the multiplicative structure gives back the notion of mixed Hodge complex defined by Deligne in \cite[section 8.1]{Deligne1974}.

\begin{defi}
Let $X$ be a topological space. A mixed Hodge diagram for $X$ is a mixed Hodge diagram $M(X)$ such that $M(X)_{\mathbf{Q}} \simeq \Apl{X}$, that is its rational component is quasi-isomorphic to the rational algebra of piecewise linear forms on $X$.
\end{defi}

The following theorem is a modified version of a theorem appearing in \cite[theorem 3.10]{Chataur} stating that the intersection homotopy type of a complex variety $X$ with only isolated singularities carries well-defined mixed Hodge structures.

\begin{thm}
\label{thm:coperverse_MI}
Let $X \in \super$ of complex dimension $n$. There exist a coperverse mixed Hodge cdga $\MI{\bullet}{X}$ together with a string of quasi-isomorphisms
\[
\MI{\bullet}{X} \leftarrow \ast \rightarrow AI_{\overline{\bullet}}(X)
\]
such that :
\begin{enumerate}
\item $\MI{\bullet}{X} = \pb{\bullet}{\tilde{\iota}}$ where $\tilde{\iota} \col M(X_{reg}) \rightarrow M(L)$ is a model of mixed Hodge cdga's for the rational homotopy type of the inclusion $i \col L \hookrightarrow X_{reg}$.
\item there is an isomorphism of coperverse mixed Hodge cdga's 
\[
H^{\ast}(\MI{\bullet}{X}) \cong \HI{\bullet}{\ast}{X}.
\]
\item The mixed Hodge cdga's $\MI{0}{X}$ and $\MI{\infty}{X}$ defines respectively the mixed Hodge structure on the rational homotopy type of the regular part $X_{reg}$ of $X$ and on the normalisation $\overline{X}$ of $X$.
\item The differential of $\MI{\bullet}{X}$ satisfies $d(W_{p}\MI{\bullet}{X}) \subset W_{p-1}\MI{\bullet}{X}$.
\end{enumerate}

This defines a functor 
\[
MI_{\overline{\bullet}} \col \super \longrightarrow \Ho{\mhcdga{Q}}.
\]
\end{thm}

\begin{proof}
The proof is similar to \cite[theorem 3.10]{Chataur}.
By \cite[theorem 3.2.1]{Durfee1988}, there is a morphism of mixed Hodge diagrams $M(X_{reg}) \rightarrow M(L)$ induced by the inclusion $i \col L \hookrightarrow X_{reg}$. The rational component of this morphism is the morphism $i^{\ast} \col \Apl{X_{reg}} \rightarrow \Apl{L}$ of rational piecewise linear forms induced by the inclusion $i \col L \hookrightarrow X_{reg}$. By \cite[theorem 3.19]{Cirici2014}, there is a commutative diagram of mixed Hodge diagrams

\[
\begin{tikzpicture}
\matrix (m)[matrix of math nodes, row sep=3em, column sep=2.5em, text height=1.5ex, text depth=0.25ex]
{ \Apl{X_{reg}}  & \Apl{L}  \\
  \ast           & \ast  \\
  M(X_{reg})     & M(L) \\};

\path[->]
(m-1-1) edge node[auto] {$i^{\ast}$} (m-1-2);
\path[->]
(m-2-1) edge node[style] {} (m-2-2);
\path[->]
(m-3-1) edge node[auto] {$\tilde{\iota}$} (m-3-2);
\path[->]
(m-2-1) edge node[auto] {} (m-1-1);
\path[->]
(m-2-2) edge node[auto] {} (m-1-2);
\path[->]
(m-2-1) edge node[auto] {} (m-3-1);
\path[->]
(m-2-2) edge node[auto] {} (m-3-2);
\end{tikzpicture}
\]
where the vertical maps are quasi-isomorphisms and $\tilde{\iota}$ is a map of mixed Hodge cdga's whose differential satisfies $d(W_{p}) \subset W_{p-1}$. We then let $\MI{\bullet}{X} := \pb{\bullet}{\tilde{\iota}}$. This construction is functorial for stratified morphisms. The above commutative diagram defines a string of quasi-isomorphisms from $\MI{\bullet}{X}$ to $\AI{\bullet}{X}$.

Let now show that $\MI{\bullet}{X}$ is a coperverse mixed Hodge cdga. Consider the mixed Hodge cdga $M(L)(t,dt)$ defined as in definition \ref{def:At,dt}. Then $\dgaucotr{p}{M(L)(t,dt)}$ is a complex of mixed Hodge structure for every perversities $\overline{p} \in \Pos{n}^{op}$. The product 
\[
\dgaucotr{p}{M(L)(t,dt)} \times \dgaucotr{q}{M(L)(t,dt)} \longrightarrow \dgaucotr{q}{M(L)(t,dt)}
\]
and the poset maps
\[
\dgaucotr{p}{M(L)(t,dt)} \longrightarrow \dgaucotr{q}{M(L)(t,dt)}
\]
for $\overline{p} \leq \overline{q} \in \Pos{n}^{op}$ are strictly compatible with filtrations. Since the category of mixed Hodge structures is abelian, for each $n \geq 0$ and each $\overline{p} \in \Pos{n}^{op}$, the vector space $\MI{p}{X}^{n}$ carries a mixed Hodge structure. The compatibility with product and poset maps is a matter of verifications. This proves the first three properties. 

The differential on $\MI{p}{X}$ being defined via the pull-back of cdga's whose differential satisfies $d(W_{p}) \subset W_{p-1}$, this also holds for $\MI{p}{X}$.

Functoriality follows by construction.
\end{proof}

From this result we can deduce the two following product structure.

\begin{cor}
Let $X \in \super$ with only isolated singularities, then the family $\lbrace \MI{p}{X} \rbrace_{(\overline{p})}$ is a $(-1)$-sharp mixed Hodge coperverse cdga.
\end{cor} 

\begin{cor}
Let $X \in \super$ with only isolated singularities, then the family of algebras $\lbrace \HI{0}{\ast}{X},\redHI{1}{\ast}{X}, \dots, \redHI{2n-2}{\ast}{X} \rbrace$ is endowed with a product
\[
\begin{cases}
\HI{0}{i}{X} \otimes \redHI{p}{j}{X} \longrightarrow \redHI{p}{i+j}{X} & \\
\redHI{p}{i}{X} \otimes \redHI{q}{j}{X} \longrightarrow \redHI{p+q+1}{i+j}{X} & p+q+1 \leq 2n-2.
\end{cases}
\]
This product is a morphism of mixed Hodge structure.
\end{cor}

Due to the method of construction of the coperverse mixed Hodge cdga $\MI{\bullet}{X}$, we have the following commutative diagram of mixed Hodge cdga's.
\[
\begin{tikzpicture}[injection/.style={right hook->,fill=white, inner sep=2pt}]
\matrix (m)[matrix of math nodes, row sep=3em, column sep=5em, text height=2ex, text depth=0.25ex]
{ M(X_{reg})  &  M(L)  & \dgaucotr{k+1}{M(L)(t,dt)} \\
  M(X_{reg})  &  M(L)  & \dgaucotr{k}{M(L)(t,dt)}\\
  0           &  0     & M(L,k)\\};

\path[->]
(m-1-1) edge node[auto] {$\tilde{\iota}$} (m-1-2);
\path[->]
(m-1-3) edge node[auto,swap] {$\delta_{1}$} (m-1-2);
\path[->]
(m-2-1) edge node[auto] {$\tilde{\iota}$} (m-2-2);
\path[->]
(m-2-3) edge node[auto,swap] {$\delta_{1}$} (m-2-2);
\path[->]
(m-3-1) edge node[auto] {} (m-3-2);
\path[->]
(m-3-3) edge node[auto] {} (m-3-2);

\path[->]
(m-1-1) edge node[auto] {} (m-2-1);
\path[->]
(m-1-2) edge node[auto] {} (m-2-2);
\path[injection]
(m-1-3) edge node[auto] {$\pomap{k+1}{k}$} (m-2-3);
\path[->]
(m-2-1) edge node[auto] {} (m-3-1);
\path[->]
(m-2-2) edge node[auto] {} (m-3-2);
\path[->>]
(m-2-3) edge node[auto] {} (m-3-3);
\end{tikzpicture}
\]
Where each elements of the last row is the quotient of the previous elements in the same column. That is $M(L,k)$ is the mixed Hodge cdga quotient such that $H^{i}(M(L,k)) = H^{k}(L)$ for $i =k$ and zero otherwise. Taking the pullback on each rows we then have a short exact sequence of mixed Hodge structure
\[
0 \longrightarrow \MI{k+1}{X} \longrightarrow \MI{k}{X} \longrightarrow M(L,k) \longrightarrow 0.
\]
This short exact sequence induces a long exact sequence of mixed Hodge structure and extends to arbitrary perversities. That is we have
\begin{cor}
Let $X \in \super$ with only isolated singularities and two perversities $\overline{p} \leq \overline{q} \in \widehat{\Pos{n}}^{op}$. We have a long exact sequence of mixed Hodge structures
\[
\cdots \rightarrow \HI{p}{i}{X} \rightarrow \HI{q}{i}{X} \rightarrow H^{i}(M(L,q,p)) \rightarrow \HI{p}{i+1}{X} \rightarrow \cdots
\]
where 
\[
H^{i}(M(L,q,p)) =
\begin{cases}
H^{i}(L) & q \leq i < p, \\
0        & \text{otherwise.}
\end{cases}
\]
\end{cor}

\section{Weight spectral sequence}
\label{section:Weight_SS}

Let $(B,W,F)$ a mixed Hodge cdga, then $(B(t,dt),W \ast \sigma,F \ast t)$ is again a mixed Hodge cdga where the filtrations are given by
\[
(W \ast \sigma)_{m}B(t,dt)^{n} := W_{m}B^{n} \otimes \mathbf{Q}[t] \oplus W_{m+1}B^{n-1} \otimes \mathbf{Q}[t]dt
\]
and
\[
(F \ast t)^{k}B(t,dt) := F^{k}B \otimes \mathbf{C}(t,dt).
\]

The graded subspace associated to the the weight filtration is then given by
\[
\gr{m}{W \ast \sigma}{B(t,dt)^{n}} = \gr{m}{W}{B^{n}} \otimes \mathbf{Q}[t]  \oplus \gr{m+1}{W}{B^{n-1}} \otimes \mathbf{Q}[t]dt.
\]

Given a mixed Hodge cdga $(B,W,F)$, we then have a cohomological weight spectral sequence $E(B,W)$ whose $E_{1}$ page is defined by
\[
\E{W}{1}{p}{q}{B} := H^{p+q}(\gr{-p}{W}{B^{p+q}}).
\]
The spectral sequence associated to a coperverse filtered cdga $(A_{\overline{\bullet}}, W)$ is compatible with the multiplicative structure. Thus, for all $r \geq 0$, The term $E_{r}(A_{\overline{\bullet}}, W)$ is a coperverse bigraded algebra with differential $d_{r}$ of degree $(r,1-r)$.

\begin{lem}
\label{lem:E_tdt_commutes}
Let $(B,W,F)$ a mixed Hodge cdga, we have a canonical isomorphism of differential bigraded algebras
\[
E_{1}(B(t,dt), W \ast \sigma) \cong E_{1}(B,W)(t,dt)
\]
\end{lem}

\begin{lem}
\label{lem:pb_E1_commutes}
Let $f \col (A,W,F) \rightarrow (B,W,F)$ be a morphism of coperverse mixed Hodge cdga's. There is a quasi-isomorphism of coperverse differential bigraded algebras 
\[
E_{1}(\pb{\bullet}{f},W) \overset{\sim}{\longrightarrow} \pb{\bullet}{(E_{1}(f,W))}.
\]
\end{lem}

\begin{proof}
The proof is similar to \cite[lemma 3.7]{Chataur} unless for the map 
\[
E_{1}(\dgaucotr{\bullet}{B(t,dt)}, W \ast \sigma) \overset{\sim}{\rightarrow} \dgaucotr{\bullet}{E_{1}(B,W)(t,dt)}
\]
which is not an isomorphism but a quasi-isomorphism.
\end{proof}

\begin{lem}
\label{lem:A_to_E1}
Let $(A_{\overline{\bullet}}, W, F)$ be a coperverse mixed Hodge cdga such that 
\[
d(W_{p}A_{\overline{\bullet}}) \subset W_{p-1}A_{\overline{\bullet}}.
\]
There is an isomorphism of complex coperverse cdga's 
\[
A_{\overline{\bullet}} \otimes \mathbf{C} \cong E_{1}(A_{\overline{\bullet}} \otimes \mathbf{C}, W).
\]
\end{lem}

\begin{proof}
The proof is is the same as the proof of \cite[lemma 3.4]{Chataur} for perverse mixed Hodge cdga's.
\end{proof}

Let $(A,W)$ be a filtered cdga of finite type over a field $\mathbf{k}$ and $\mathbf{k} \subset \mathbf{K}$ a field extension. By \cite[theorem 2.26]{Cirici2014} we have that $A \cong E_{r}(A,W)$ if and only if $A \otimes_{\mathbf{k}} \mathbf{K} \cong E_{r}(A \otimes_{\mathbf{k}} \mathbf{K},W)$. For a coperverse cdga of finite type the same proof is valid. This implies the isomorphism of lemma \ref{lem:A_to_E1} descends to an isomorphism over $\mathbf{Q}$.

Let $X \in \super$ of complex dimension $n$. The inclusion $i \col L \hookrightarrow X_{reg}$ of the link into the regular part induces a morphism of multiplicative weight spectral sequence $E_{1}(i^{\ast}) \col E_{1}(X_{reg}) \rightarrow E_{1}(L)$. We define 
\[
\EI{1}{\bullet}{X} := \pb{\bullet}{E_{1}(i^{\ast})}.
\]
This is a coperverse differential bigraded algebra whose cohomology satifies
\[
\EIbidg{2}{p}{r,s}{X} := H^{r,s}(\EI{1}{p}{X}) \cong \gr{s}{W}{\HI{p}{r+s}{X}}
\]

\begin{defi}
Let $X \in \super$ of complex dimension $n$. The spectral sequence $\EI{1}{\bullet}{X}$ defined by
\[
\EI{1}{\bullet}{X} := \pb{\bullet}{E_{1}(i^{\ast})}
\]
is called the coperverse weight spectral sequence associated to $X$.
\end{defi}

In \cite[theorem 3.12]{Chataur}, Chataur and Cirici prove the existence of a quasi-isomorphism between the rational perverse model $\IA{\bullet}{X}$ of a complex projective variety with only isolated singularities and the first term of its perverse weight spectral sequence $\IE{1}{\bullet}{X}$. This theorem can be modified to get the following one. 

\begin{thm}
\label{thm:AI_to_EI1}
Let $X \in \super$ with only isolated singularities. There is a string of quasi-isomorphisms of coperverse cdga's from $\MI{\bullet}{X} \otimes \mathbf{C}$ to $\EI{1}{\bullet}{X} \otimes \mathbf{C}$. In particular, there is an isomorphism in $\mathrm{Ho}(\cdga{C})$ from $AI_{\overline{\bullet}}(X) \otimes \mathbf{C}$ to $\EI{1}{\bullet}{X} \otimes \mathbf{C}$.
\end{thm}

\begin{proof}
The proof is similar to \cite[theorem 3.12]{Chataur}.

Let $(\MI{\bullet}{X}, W,F)$ be the coperverse mixed Hodge cdga given by the theorem \ref{thm:coperverse_MI}. Since the differential satisfies 
\[
d(W_{p}\MI{\bullet}{X}) \subset W_{p-1}\MI{\bullet}{X}
\]
by the lemma \ref{lem:A_to_E1} we have an isomorphism of complex coperverse cdga's $\MI{\bullet}{X} \otimes \mathbf{C} \cong E_{1}(\MI{\bullet}{X} \otimes \mathbf{C},W)$.

By construction, we have $\MI{\bullet}{X} := \pb{\bullet}{\tilde{\iota}}$, where
\[
\tilde{\iota} \col (M(X_{reg}),W,F) \rightarrow (M(L),W,F)
\]
is a morphism of mixed Hodge cdga's which computes the rational homotopy type of $\iota \col L \rightarrow X_{reg}$. Thus by lemma \ref{lem:pb_E1_commutes} we have a quasi-isomorphism of coperverse cdga's $E_{1}(\MI{\bullet}{X},W) \longrightarrow \pb{\bullet}{E_{1}(\tilde{\iota},W)}$. It remains to note that we have a string of quasi-isomorphisms from $\pb{\bullet}{{E_{1}(\tilde{\iota})}}$ to $\EI{1}{\bullet}{X} := \pb{\bullet}{E_{1}(i^{\ast})}$
\end{proof}

\begin{remark}
Suppose we have a topological space $X$ such that its rational model is endowed with an increasing filtration $W$, then one can consider the associated spectral sequence $E_{1}(X,W)$. The existence of a string of quasi-isomorphisms between the rational model of $X$ and the first page $E_{1}(X,W)$ is called the $E_{1}$-formality and is a property of complex algebraic varieties, see \cite{Cirici2014} and \cite{Chataur2015}. It is an interesting result that the intersection spaces of complex projective varieties have this property although they are not algebraic varieties.
\end{remark}

\begin{defi}
Let $X$ be a compact, connected oriented pseudomanifold of dimension $n$ with only isolated singularities. We say that $X$ is a $EI_{r,\overline{\bullet}}$-formal topological space if its coperverse rational model $\AI{\bullet}{X}$ can be endowed with an increasing bounded filtration $W$ such that there exists a string of quasi-isomorphisms between $\AI{\bullet}{X}$ and the $r$-th term of its associated spectral sequence $\EI{r}{\bullet}{X,W}$.
\end{defi}

With this definition, the theorem \ref{thm:AI_to_EI1} can be rephrased in the following corollary.

\begin{cor}
Let $X \in \super$ with only isolated singularities. The space $X$ is $EI_{1,\overline{\bullet}}$-formal with respect to the weight filtration.
\end{cor}

\subsection{The case of a smooth exceptional divisor}
\label{subsec:smooth_div}
\subsubsection{Notations}

Let $X$ be a complex projective variety of complex dimension $n$ with only normal isolated singularities. We denote by $\Sigma = \lbrace \sigma_{1},\dots, \sigma_{\nu} \rbrace
$ the singular locus of $X$ and by $X_{reg} := X - \Sigma$ its regular part. We also denote by $L := L(\Sigma, X)$ the link of $\Sigma$ in $X$ and by $i \col L \hookrightarrow X_{reg}$ the natural inclusion of the link into the regular part.

Since $\Sigma$ is discrete, we can write $L$ as a disjoint union $L= \sqcup_{\sigma_{i}} L_{i}$ where $L_{i}:= L(\sigma_{i}, X)$ is the link of $\sigma_{i} \in \Sigma$ in $X$. The assumption that $X$ is normal implies that $L_{i}$ is connected for all $\sigma_{i} \in \Sigma$.

From now on, we will always assume $X$ admits a resolution of singularities
\[
\begin{tikzpicture}[injection/.style={right hook->,fill=white, inner sep=2pt}]
\matrix (m)[matrix of math nodes, row sep=3em, column sep=5em, text height=2ex, text depth=0.25ex]
{ D       &  \widetilde{X}  \\
  \Sigma  &  X \\};

\path[injection]
(m-1-1) edge node[auto] {$j$} (m-1-2);
\path[injection]
(m-2-1) edge node[auto,swap] {} (m-2-2);
\path[->]
(m-1-1) edge node[auto,swap] {} (m-2-1);
\path[->]
(m-1-2) edge node[auto] {$f$} (m-2-2);
\end{tikzpicture}
\]
such that the exceptional divisor $D := f^{-1}(\Sigma)$ is smooth.

We denote by 
\[
j^{k} \col H^{k}(\widetilde{X}) \longrightarrow H^{k}(D) \text{  and  } \gamma^{k} \col H^{k-2}(D) \longrightarrow H^{k}(\widetilde{X})
\]
the restriction maps and the Gysin maps induced by the inclusion $j$. 

For all $k \geq 2$ we also denote by 
\[
j^{k}_{\sharp} \col H^{k-2}(D) \overset{\gamma^{k}}{\longrightarrow} H^{k}(\widetilde{X}) \overset{j^{k}}{\longrightarrow} H^{k}(D)
\]
the composition of the two maps.

The morphism $E_{1}(i^{\ast}) \col E_{1}^{\ast,\ast}(X_{reg}) \rightarrow E_{1}^{\ast,\ast}(L)$ of weight spectral sequence induced by the inclusion $i \col L \hookrightarrow X_{reg}$ is defined by 

\[
\begin{tikzpicture}[injection/.style={right hook->,fill=white, inner sep=2pt}]
\matrix (m)[matrix of math nodes, row sep=3em, column sep=3em, text height=2ex, text depth=0.25ex]
{E_{1}^{-1,s}(X_{reg})  & E_{1}^{0,s}(X_{reg}) & H^{s-2}(D)   &  H^{s}(\widetilde{X})  \\
 E_{1}^{-1,s}(L)        & E_{1}^{0,s}(L)       & H^{s-2}(D)   &  H^{s}(D) \\};

\path[->]
(m-1-1) edge node[auto,swap] {$E_{1}^{-1,s}(i^{\ast})$} (m-2-1);
\path[->]
(m-1-2) edge node[auto] {$E_{1}^{0,s}(i^{\ast}) \, =$} (m-2-2);
\path[->]
(m-1-3) edge node[auto,swap] {$\mathrm{id}$} (m-2-3);
\path[->]
(m-1-4) edge node[auto] {$j^{s}$} (m-2-4);

\path[->]
(m-1-1) edge node[auto] {$d$} (m-1-2);
\path[->]
(m-2-1) edge node[auto,swap] {$d$} (m-2-2);
\path[->]
(m-1-3) edge node[auto] {$\gamma^{s}$} (m-1-4);
\path[->]
(m-2-3) edge node[auto,swap] {$j^{s}_{\sharp}$} (m-2-4);
\end{tikzpicture}
\]

The algebra structure on $E_{1}^{\ast,\ast}(X_{reg})$ is given by the cup product of $H^{\ast}(\widetilde{X})$, together with the map 
\[ 
\begin{array}{ccc}
H^{s}(\widetilde{X}) \times H^{s'}(D) & \longrightarrow & H^{s+s'}(D)\\
  (x,a)                               & \longmapsto     & j^{s}(x)\cdot a.\\
\end{array}
\]
This algebra structure is compatible with the differential $\gamma$ because $\gamma(j^{s}(x)\cdot a) = x \cdot \gamma(a)$.

The non-trivial products on $E_{1}^{\ast,\ast}(L)$ are the maps
\[
E_{1}^{0,s}(L) \times E_{1}^{r,s'}(L) \longrightarrow E_{1}^{r,s+s'}(L) \quad r \in \lbrace 0,1 \rbrace, s,s' \geq 0
\]
induced by the cup-product on $H^{\ast}(D)$.

The coperverse weight spectral sequence $\EI{1}{\bullet}{X} := \pb{\bullet}{E_{1}(i^{\ast})}$ for $X$ is then given by

\noindent\makebox[\textwidth]{
\(
\begin{array}{| c || c c c c c |} 
\hline
s > p+1        & H^{s-2}(D)\otimes \mathbf{Q}[t] & \rightarrow & \mathcal{I}^{s}_{0} \oplus H^{s-2}(D)\otimes \mathbf{Q}[t]dt & \rightarrow & H^{s}(D)\otimes \mathbf{Q}[t]dt  \\
\hline
s = p+1        & \coim{p} \oplus H^{s-2}(D)\otimes \mathbf{Q}[t]t & \rightarrow & \mathcal{I}^{s}_{0} \oplus H^{s-2}(D)\otimes \mathbf{Q}[t]dt & \rightarrow & H^{s}(D)\otimes \mathbf{Q}[t]dt  \\
\hline
1 \leq s < p+1 & H^{s-2}(D)\otimes \mathbf{Q}[t]t & \rightarrow & \mathcal{I}^{s}_{1} \oplus H^{s-2}(D)\otimes \mathbf{Q}[t]dt & \rightarrow & H^{s}(D)\otimes \mathbf{Q}[t]dt  \\
\hline
\hline
s = 0          & 0                                            & & \mathcal{I}^{0}_{0} & \rightarrow & H^{0}(D)\otimes \mathbf{Q}[t]dt  \\
\hline
\hline
\EIbidg{1}{p}{r,s}{X}               & r=-1                                         &             & r=0                              &             & r=1 \\
\hline 
\end{array}  
\)}

Where
\begin{enumerate}
\item $\coim{p}$ is the image of the section of $d^{-1,s}_{1} \col E^{-1,s}_{1}(L) \rightarrow E^{0,s}_{1}(L)$, ie a section of $j^{s}_{\sharp} \col H^{s-2}(D) \rightarrow H^{s}(D)$. Note that $\coim{p}$ is just a computational tool and does not impact the value of the $EI_{2}$ term since it has been shown in \cite[theorem 2.18]{Banagl2010} that the values of $\HI{p}{k}{X}$ for rational coefficients are independent of the choices made during the construction.
\item $\mathcal{I}^{s}_{k}$, $k \in \lbrace 0,1 \rbrace$, is the vector space given by the following pullback square.
\[
\begin{tikzpicture}
\matrix (m)[matrix of math nodes, row sep=3em, column sep=5em, text height=2ex, text depth=0.25ex]
{ \mathcal{I}^{s}_{k}  &  H^{s}(D)\otimes \mathbf{Q}[t]t^{k}  \\
  H^{s}(\widetilde{X}) &  H^{s}(D) \\};

\path[->]
(m-1-1) edge node[auto] {} (m-1-2);
\path[->]
(m-2-1) edge node[auto] {$j^{s}$} (m-2-2);
\path[->]
(m-1-1) edge node[above=1em, right=1em] {$\ulcorner$} (m-2-1);
\path[->]
(m-1-2) edge node[auto] {$\delta_{1}$} (m-2-2);
\end{tikzpicture}
\]
\item The differential $\difb{p}{-1,s} \col \EIbidg{1}{p}{-1,s}{X} \rightarrow \EIbidg{1}{p}{0,s}{X}$ is defined by 
\[
\sum a_{i}t^{i} \mapsto \left((\sum \gamma^{s}(a_{i}), \sum j^{s}_{\sharp}(a_{i})t^{i}), \sum i a_{i} t^{i-1}dt \right) \quad a_{i} \in H^{s-2}(D).
\]
\item The differential $\difb{p}{0,s} \col \EIbidg{1}{p}{0,s}{X} \rightarrow \EIbidg{1}{p}{1,s}{X}$ is defined by 
\[
\left((x, \sum a_{i}t^{i}), \sum b_{i} t^{i}dt \right) \mapsto \sum ia_{i}t^{i-1}dt + \sum j^{s}_{\sharp}(b_{i})t^{i}dt
\]
with
\[
\begin{cases}
a_{i} \in H^{s}(D), b_{i} \in H^{s-2}(D), & \\
x \in H^{s}(\widetilde{X}), j^{s}(x) = \sum a_{i}. & \\
\end{cases}
\]
\end{enumerate} 

We describe the internal algebra structure of the coperverse weight spectral sequence $\EIbidg{1}{p}{r,s}{X}$. Due to the method of construction, this algebra structure is similar to the external one on the perverse weight spectral sequence for intersection cohomology in \cite{Chataur}.

The algebra structure is described by the following maps. We set $x,x' \in H^{\ast}(\widetilde{X})$ and $a,a',b,b' \in H^{\ast}(D)\otimes \mathbf{Q}[t]$.
\[ 
\begin{array}{ccc}
\EIbidg{1}{p}{0,s}{X} \times \EIbidg{1}{p}{0,s'}{X} & \longrightarrow & \EIbidg{1}{p}{0,s+s'}{X} \\
((x,a+b\cdot dt),(x',a' + b' \cdot dt))             & \longmapsto     & (xx', aa' + (a'b + b'a)dt) \\
\end{array}
\]
\[ 
\begin{array}{ccc}
\EIbidg{1}{p}{0,s}{X} \times \EIbidg{1}{p}{1,s'}{X} & \longrightarrow & \EIbidg{1}{p}{1,s+s'}{X} \\
((x,a+b\cdot dt),(a' \cdot dt))                     & \longmapsto     & aa' \cdot dt \\
\end{array}
\]
\[ 
\begin{array}{ccc}
\EIbidg{1}{p}{-1,s}{X} \times \EIbidg{1}{p}{1,s'}{X} & \longrightarrow & \EIbidg{1}{p}{0,s+s'}{X} \\
(a,a'\cdot dt)                                       & \longmapsto     & aa' \cdot dt \\
\end{array}
\]
\[ 
\begin{array}{ccc}
\EIbidg{1}{p}{-1,s}{X} \times \EIbidg{1}{p}{0,s'}{X} & \longrightarrow & \EIbidg{1}{p}{-1,s+s'}{X} \\
(a,(x,a' + b' \cdot dt))                             & \longmapsto     & aa' \\
\end{array}
\]

Note that since $\coim{p} \subset H^{s-2}(D)$ and $\mathcal{I}^{s}_{1} \subset \mathcal{I}^{s}_{0}$, $\pomap{k+1}{k}$ induces a morphism of spectral sequences of bidegree $(0,0)$
\[
EI_{1}(\pomap{k+1}{k}) \col \EI{1}{k+1}{X} \rightarrow \EI{1}{k}{X}.
\]
This poset map extends the internal structure structure into an external one, meaning we then have an extended product
\[
\EIbidg{1}{p}{r,s}{X} \times \EIbidg{1}{q}{r',s'}{X} \longrightarrow \EIbidg{1}{q}{r+r',s+s'}{X}
\]
defined with the same map as before for the internal structure and following the same rules for $r,r',s,s'$.

By computing the cohomology of $\EI{1}{p}{X}$ we have
\[
\begin{array}{| c || c c c |} 
\hline
s > p+1        & \ker \gamma^{s} &  \coker{\gamma^{s}} &  0  \\
\hline
s = p+1        & 0               &  \coker{\gamma^{s}_{|\coim{p}}} & 0  \\
\hline
1 \leq s < p+1 & 0               &  \ker j^{s}         & \coker{j^{s}}  \\
\hline
\hline
s = 0          & 0               & H^{0}(\widetilde{X})          & 0 \\
\hline
\hline
\EIbidg{2}{p}{r,s}{X} & r=-1            & r=0                              & r=1 \\
\hline 
\end{array}  
\]

Where $\gamma^{s}_{|\coim{p}}$ is the restriction of $\gamma^{s}$ to 
\[
\coim{p} \rightarrow H^{s}(\widetilde{X}).
\]

We then have the following isomorphisms

\noindent\makebox[\textwidth]{
\(
\HI{p}{k}{X} =
\begin{cases}
H^{0}(\widetilde{X}) = \mathbf{Q} & k=0 \\
H^{k}(X) \cong  \ker j^{k} \oplus \coker{j^{k-1}} & 1 \leq k < p+1 \\
H^{k}(X) \oplus \im{H^{k}(X_{reg}) \rightarrow H^{k}(L)} \cong \coker{\gamma^{k}_{|\coim{p}}} \oplus \coker{j^{k-1}} \oplus \ker \gamma^{k+1} & k = p+1 \\
H^{k}(X_{reg}) \cong \ker \gamma^{k+1} \oplus \coker{\gamma^{k}} & k > p+1
\end{cases}
\)}

\subsubsection{Remark on \texorpdfstring{$\coker{j^{0}}$}{coker j0}}
\label{subsubsec:coker}
It is important to note here that the values of $\ker j^{s}$ and $\coker{j^{s}}$ recorded in the array of the $EI_{2}$ term above start with $s=1$, meaning we don't take into account $\ker j^{0}$ and $\coker{j^{0}}$, this is intended.

Indeed, $\coker{j^{0}}$ accounts for the number of loops created when the intersection spaces are defined as a homotopy pushout over a single point, like in the original definition of \cite{Banagl2010}, this not the definition we use.

As a consequence, when we have multiple isolated singularities, the generalised Poincaré duality of the intersection spaces fails for $\redHI{p}{1}{X} \cong \redHI{q}{n-1}{X}$. 

This is also one of the reason we modified the definition of intersection spaces. If we used the original definition of \cite{Banagl2010}, the mixed Hodge structure on $\HI{p}{k}{X}$ would never be pure unless if there is only one isolated singularity, which is the case where $\coker{j^{0}}=0$.

\subsubsection{Remark on the zero perversity} 
The intersection space for the zero perversity is by definition \ref{def:intersectionspace2} the regular part $X_{reg}$ of the complex projective variety $X \in \super$ involved. The isomorphism given above by the $EI_{2}$ term gives 
\[
\HI{0}{1}{X} = \coker{\gamma^{1}_{|\coim{0}}} \oplus \ker \gamma^{2}.
\]
Let's see that this coincides with $H^{1}(X_{reg})$.

Consider the term $\coker{\gamma^{1}_{|\coim{0}}}$, by definition $\coim{0}$ is defined as the image of a section of $j^{0}_{\sharp} \col H^{-1}(D)=0 \rightarrow H^{1}(D)$. So we have $\coim{0}=0$, and we then have $\coker{\gamma^{1}_{|\coim{0}}} \cong \coker{\gamma^{1}}.$

We then have what we wanted
\[
\HI{0}{1}{X} = \coker{\gamma^{1}} \oplus \ker \gamma^{2} = H^{1}(X_{reg}).
\]

\subsection{\texorpdfstring{$(\overline{p},r)$}{(p,r)}-purity implies \texorpdfstring{$(\overline{p},r)$}{(p,r)}-formality}

\begin{defi}
Let $0 \leq r \leq \infty$ be an integer and $\overline{p}$ a perversity. A morphism of coperverse cdga's $f_{\overline{\bullet}} \col A_{\overline{\bullet}} \rightarrow B_{\overline{\bullet}}$ is a $(\overline{p},r)$-quasi-isomorphism if for all perversities $\overline{s} \leq \overline{p}$ in $\mathcal{P}^{op}$ the induced morphism
\[
H_{\overline{s}}^{i}(f) \col H_{\overline{s}}^{i}(A) \longrightarrow H_{\overline{s}}^{i}(B)
\]
is an isomorphism for all $i \leq r$ and a monomorphism for $i = r+1$.
\end{defi}

\begin{defi}
\begin{enumerate}
\item A coperverse cdga $(A_{\overline{\bullet}},d)$ over $\mathbf{k}$ is said to be $(\overline{p},r)$-formal if there exist a string of $(\overline{p},r)$-quasi-isomorphisms from $(A_{\overline{\bullet}},d)$ to its cohomology $(H_{\overline{\bullet}}(A,\mathbf{k}),0)$ seen as a coperverse cdga with zero differential.
\item Let $X \in \super$, $\I{\bullet}{X}$ is said to be $(\overline{p},r)$-formal if its coperverse rational model $AI_{\overline{\bullet}}(X)$ is $(\overline{p},r)$-formal.
\item Let $X \in \super$, $\I{\bullet}{X}$ is said to be $(\overline{p},r)$-pure if the weight filtration $\HI{s}{k}{X}$ is pure of weight $k$ for all $k \leq r$ and for all perversities $\overline{s} \leq \overline{p}$ in $\mathcal{P}^{op}$.
\end{enumerate}
\end{defi}

\begin{thm}\label{thm:pure_is_formal}
Let $X \in \super$ of dimension $n$ with only isolated singularities. Let $r \geq 0$ be an integer and $\overline{p}$ a perversity. Suppose that $\I{\bullet}{X}$ is $(\overline{p},r)$-pure, then $\I{\bullet}{X}$ is $(\overline{p},r)$-formal. 
\end{thm}

\begin{proof}
By theorem \ref{thm:AI_to_EI1}, we need to define a string of $(\overline{p},r)$-quasi-isomorphisms of differential bigraded algebras from  $(\EIbidg{1}{s}{i,j}{X}, \difb{s}{i,j})$ to $(\EIbidg{2}{s}{i,j}{X}, 0)$ for $i+j \leq r$ and $\overline{s} \leq \overline{p}$ in $\Pos{n}^{op}$.

Given $X \in \super$ of dimension $n$ with only isolated singularities, the terms $EI_{1}$ and $EI_{2}$ of the spectral sequence look like.

\[
\begin{array}{| c || c : c : c |}
\hline
j=5            & \vdots  & \vdots   & \vdots  \\ 
\hline
j=4            & \EIbidg{1}{\bullet}{-1,4}{X}  & \EIbidg{1}{\bullet}{0,4}{X}   & \EIbidg{1}{\bullet}{1,4}{X}  \\
\hline
j=3            & \EIbidg{1}{\bullet}{-1,3}{X}  & \EIbidg{1}{\bullet}{0,3}{X}   & \EIbidg{1}{\bullet}{1,3}{X}  \\
\hline
j=2            & \EIbidg{1}{\bullet}{-1,2}{X}  & \EIbidg{1}{\bullet}{0,2}{X}   & \EIbidg{1}{\bullet}{1,2}{X}  \\
\hline
j=1            & \EIbidg{1}{\bullet}{-1,1}{X}  & \EIbidg{1}{\bullet}{0,1}{X}   & \EIbidg{1}{\bullet}{1,1}{X}  \\
\hline
j=0            &                       0       & \EIbidg{1}{\bullet}{0,0}{X}   & \EIbidg{1}{\bullet}{1,0}{X}  \\
\hline
\hline
\EIbidg{1}{\bullet}{i,j}{X} & i=-1                          & i=0                                      & i=1 \\
\hline 
\end{array}  
\]

\[
\begin{array}{| c || c : c : c |}
\hline
j=5            & \vdots  & \vdots   & \vdots  \\ 
\hline
j=4            & \gr{4}{W}{\HI{\bullet}{3}{X}}  & \gr{4}{W}{\HI{\bullet}{4}{X}}   & \gr{4}{W}{\HI{\bullet}{5}{X}}  \\
\hline
j=3            & \gr{3}{W}{\HI{\bullet}{2}{X}}  & \gr{3}{W}{\HI{\bullet}{3}{X}}   & \gr{3}{W}{\HI{\bullet}{4}{X}}  \\
\hline
j=2            & \gr{2}{W}{\HI{\bullet}{1}{X}}  & \gr{2}{W}{\HI{\bullet}{2}{X}}   & \gr{2}{W}{\HI{\bullet}{3}{X}}  \\
\hline
j=1            & \gr{1}{W}{\HI{\bullet}{0}{X}}  & \gr{1}{W}{\HI{\bullet}{1}{X}}   & \gr{1}{W}{\HI{\bullet}{2}{X}}  \\
\hline
j=0            &            0                   & \gr{0}{W}{\HI{\bullet}{0}{X}}   & \gr{0}{W}{\HI{\bullet}{1}{X}}  \\
\hline
\hline
\EIbidg{2}{\bullet}{i,j}{X} & i=-1              & i=0                             & i=1 \\
\hline 
\end{array}  
\]

The $(\overline{p},r)$-purity assumption implies that $\gr{j}{W}{\HI{s}{j-1}{X}}=0$ for all $j \leq r+1$ and $\gr{j}{W}{\HI{s}{j+1}{X}}=0$ for all $j \leq r-1$. This means that $\ker \difb{s}{-1,j} =0$ for all $j \leq r+1$ and $\im{\difb{s}{0,j}} = \EIbidg{1}{s}{1,j}{X}$ for all $j \leq r-1$.

Denote by $\FI{s}{i,j}{X}$ the bigraded differential algebra defined by, for all $\overline{s} \leq \overline{p}$ in $\Pos{n}^{op}$
\[
\begin{cases}
\FI{s}{-1,j}{X} := \EIbidg{1}{s}{-1,j}{X}         & j \leq r+1, \\
\FI{s}{-1,j}{X} := 0                              & j > r+1, \\
\FI{s}{0,j}{X}  := \ker \difb{s}{0,j}             & \forall \,j, \\
\FI{s}{1,j}{X}  := 0                              & \forall \, j. 
\end{cases}
\]
The differential being $\difb{s}{i,j}$.

The bigraded differential algebra $\FI{s}{\ast,\ast}{X}$ has the following product structure 
\[
\begin{cases}
\FI{s}{-1,j}{X} \times \FI{s}{-1,j}{X} \longrightarrow 0                       & \forall \, j, \\
\FI{s}{-1,j}{X} \times \FI{s}{0,j'}{X}  \longrightarrow  \FI{s}{-1,j+j'}{X} & \forall \, j,j', \\
\FI{s}{0,j}{X}  \times \FI{s}{0,j'}{X}  \longrightarrow  \FI{s}{0,j+j'}{X}  & \forall \, j,j'.
\end{cases}
\]
which is well defined and is compatible with $\difb{s}{i,j}$ and poset maps $EI_{1}(\pomap{s+1}{s})$ for all $\overline{s} \leq \overline{p}$.

We then clearly have a inclusion $(\FI{s}{i,j}{X}, \difb{s}{i,j}) \hookrightarrow (\EIbidg{1}{s}{i,j}{X}, \difb{s}{i,j})$, the map $(\FI{s}{i,j}{X}, \difb{s}{i,j}) \rightarrow (\EIbidg{2}{s}{i,j}{X},0)$ is defined by the following commutative diagram where the dashed arrows are the zero map.

\[
\begin{tikzpicture}[injection/.style={right hook->,fill=white, inner sep=2pt}]
\matrix (m)[matrix of math nodes, row sep=3em, column sep=-1em, text height=2ex, text depth=0.25ex]
{\FI{s+1}{-1,j}{X}    &                              & \ker d^{0,j}_{1, \overline{s+1}} &  \\
                             & \FI{s}{-1,j}{X}       &                                  & \ker d^{0,j}_{1, \overline{s}} \\
 \gr{j}{W}{\HI{s+1}{j-1}{X}} &                              & \gr{j}{W}{\HI{s+1}{j}{X}}        & \\
                             & \gr{j}{W}{\HI{s}{j-1}{X}}    &                                  & \gr{j}{W}{\HI{s}{j}{X}}\\};
                             
\path[injection]
(m-1-1) edge node[auto] {$\difb{s+1}{-1,j}$} (m-1-3);
\path[injection]
(m-2-2) edge node[above left] {$\difb{s}{-1,j}$} (m-2-4);
\path[dashed,->]
(m-3-1) edge node[auto,swap] {} (m-3-3);
\path[dashed,->]
(m-4-2) edge node[auto,swap] {} (m-4-4);

\path[dashed,->]
(m-1-1) edge node[auto] {} (m-3-1);
\path[dashed,->]
(m-2-2) edge node[auto] {} (m-4-2);
\path[->>]
(m-1-3) edge node[below right] {$p$} (m-3-3);
\path[->>]
(m-2-4) edge node[auto] {$p$} (m-4-4);

\path[->]
(m-1-1) edge node[auto] {$EI_{1}(\pomap{s+1}{s})$} (m-2-2);
\path[->]
(m-3-1) edge node[auto,swap] {$EI_{2}(\pomap{s+1}{s})$} (m-4-2);
\path[->]
(m-1-3) edge node[auto] {$EI_{1}(\pomap{s+1}{s})$} (m-2-4);
\path[->]
(m-3-3) edge node[auto,swap] {$EI_{2}(\pomap{s+1}{s})$} (m-4-4);
\end{tikzpicture}
\]

The string $(\EIbidg{1}{s}{i,j}{X}, \difb{s}{i,j}) \longleftarrow (\FI{s}{i,j}{X}, \difb{s}{i,j}) \rightarrow (\EIbidg{2}{s}{i,j}{X},0)$ then defines a $(\overline{p},r)$-quasi-isomorphism.
\end{proof}

Regardless of the perversity. The two cases of special interest here are the cases where $r=1$ and $r= \infty$. 

The case $r=1$, the 1-formality, implies that the rational Malcev completion of $\pi_{1}(\I{p}{X})$ can be computed directly from the cohomology group $\HI{p}{1}{X}$, together with the cup product $\HI{p}{1}{X} \otimes \HI{p}{1}{X} \rightarrow \HI{p}{2}{X}$. We then say that $\pi_{1}(\I{p}{X})$ is 1-formal. 

The case $r=\infty$ implies the formality of $\I{p}{X}$ in the usual sense, which in the cases where $\I{p}{X}$ is simply-connected or nilpotent implies that the rational homotopy groups $\pi_{i}(\I{p}{X})\otimes \mathbf{Q}$ can be directly computed from the cohomology ring $\HI{p}{\ast}{X}$. We note that formality implies 1-formality.

Suppose now $X \in \super$ with only normal isolated singularities, that is 
\[
\HI{\infty}{k}{X}=H^{k}(\overline{X})=H^{k}(X)
\]
then by the Van-Kampen theorem and by definition \ref{def:intersectionspace2} for any perversity $\overline{p}$ we have
\[
\pi_{1}(X) = \pi_{1}(\I{p}{X}) = \pi_{1}(X_{reg}).
\]
Morevover, whether $\overline{p}=\overline{0}$ or $\overline{p} \neq \overline{0}$ we have the two following commutative diagrams.
\[
\begin{tikzpicture}[injection/.style={right hook->,fill=white, inner sep=2pt}]
\matrix (m)[matrix of math nodes, row sep=3em, column sep=3em, text height=2ex, text depth=0.25ex]
{  \HI{0}{1}{X} \otimes \HI{0}{1}{X}                               &  \HI{0}{2}{X} \\ 
H^{1}(X_{reg}) \otimes  H^{1}(X_{reg})  &   H^{2}(X_{reg})  \\};

\path[->]
(m-1-1) edge node[auto] {$- \cup -$} (m-1-2);
\path[->]
(m-2-1) edge node[auto,swap] {$- \cup -$} (m-2-2);
\path[->]
(m-1-1) edge node[auto,swap] {=} (m-2-1);
\path[->]
(m-1-2) edge node[auto] {=} (m-2-2);
\end{tikzpicture}
\]
\[
\begin{tikzpicture}[injection/.style={right hook->,fill=white, inner sep=2pt}]
\matrix (m)[matrix of math nodes, row sep=3em, column sep=3em, text height=2ex, text depth=0.25ex]
{ H^{1}(X) \otimes  H^{1}(X)  &   H^{2}(X)  \\
  \HI{p}{1}{X} \otimes \HI{p}{1}{X}                   &  \HI{p}{2}{X} \\};

\path[->]
(m-1-1) edge node[auto] {$- \cup -$} (m-1-2);
\path[->]
(m-2-1) edge node[auto,swap] {$- \cup -$} (m-2-2);
\path[->]
(m-1-1) edge node[auto,swap] {$\cong$} (m-2-1);
\path[injection]
(m-1-2) edge node[auto] {} (m-2-2);
\end{tikzpicture}
\]

Which means that if $X$ is 1-formal then we can compute the rational Malcev completion of $\pi_{1}(\I{p}{X})$ by computing the one from $\pi_{1}(X)$. It is a result from \cite{Arapura2015} that when considering normal projective varieties the fundamental group is always 1-formal, see also \cite[Corollary 3.8]{Chataur2015} for the isolated singularities case. We can then deduce the following result

\begin{propo}
Let $X \in \super$ with only normal isolated singularities. Then for any perversity $\overline{p}$ $\pi_{1}(\I{p}{X})$ is formal.
\end{propo}

We also highlight the case $r=\infty$.

\begin{cor}
Let $X \in \super$ with only isolated singularities. If $\I{\bullet}{X}$ is $(\overline{p},\infty)$-pure then $\I{s}{X}$ is formal for any $\overline{s} \leq \overline{p}$ in $\mathcal{P}^{op}$.
\end{cor}

\begin{remark}
The question of the purity of the weight filtration is also considered in intersection cohomology, where a similar result of "purity implies formality" exists \cite[corollary 3.13]{Chataur}. It must be pointed out that the purity of $X \in \super$ in intersection cohomology does not imply the purity of $\I{\bullet}{X}$. For example the Kummer surface of section \ref{subsec:Kummer}, it is a $\mathbf{Q}$-homology manifold and as such $\IH{p}{k}{X}$ is pure of weight $k$ for any perversities and then is intersection formal. This is not the case of the corresponding intersection space for the middle perversity $\I{1}{X}$ since $\gr{4}{W}{\HI{1}{3}{X}} \neq 0$.

Another and more involved example. It is a consequence of Gabber's purity theorem and the decomposition theorem of intersection homology (see \cite{Steenbrink1983}) that for projective varieties $X$ with isolated singularities and for the middle perversity, the weight filtration $W$ on $\IH{m}{k}{X}$ is pure of weight $k$ for all $k \geq 0$, this is not the case for the Calabi-Yau 3-folds treated in the last part as we see that the weight filtration $W$ on $\HI{m}{k}{X}$ isn't pure.
\end{remark}

\section{Formality of intersection spaces for 3-folds}
\label{sec:formality_3folds}
\subsection{Preparatory work}

Let $X$ be a complex projective algebraic 3-fold with isolated singularities and denote by $\Sigma= \lbrace \sigma_{1}, \dots, \sigma_{\nu} \rbrace$ the singular locus of $X$. Assume that there is a resolution of singularities $f \col \widetilde{X} \rightarrow X$ such that the exceptional divisor $D := f^{-1}(\Sigma)$ is smooth and the the link $L_{i}$ of $\sigma_{i}$ in $X$, for all $\sigma_{i} \in \Sigma$ is simply connected.

First we recall and collect the different properties we will need. We state them in the case of a space of complex dimension 3 but they are completely general and holds for any complex projective variety of complex dimension $n$ with only isolated singularities by replacing $3$ by $n$. The proofs can be found in \cite{Chataur}.

\begin{lem}
\label{lem:DP_ker_coker}
We have the following Poincaré duality isomorphisms for all $0 \leq s \leq 3$,
\[
\coker{\gamma^{3+s}} \cong (\ker j^{3-s})^{\vee} \quad \ker \gamma^{3+s} \cong (\coker{j^{3-s}})^{\vee}
\]
\end{lem}

Recall that since $\dim(\Sigma) =0$, the weight filtration on the cohomology of the link is semi-pure, meaning :
\begin{itemize}
\item the weights on $H^{k}(L)$ are less than or equal to $k$ for $k <3$,
\item the weights on $H^{k}(L)$ are greater or equal to $k+1$ for $k \geq 3$.
\end{itemize}

We than have the following results.
\begin{lem}
\label{lem:properties_1}
With the previous notations we have :
\begin{enumerate}
\item The map $j_{\sharp}^{k} \col H^{k-2}(D) \rightarrow H^{k}(D)$ is injective for $k \leq 3$ and surjective for $k \geq 3$.
\item The Gysin map $\gamma^{k} \col H^{k-2}(D) \rightarrow H^{k}(\widetilde{X})$ is injective for $k \leq 3$ and $\gamma^{6}$ is surjective. 
\item The restriction morphism $j^{k} \col H^{k}(\widetilde{X}) \rightarrow H^{k}(D)$ is surjective for $k \geq 3$.
\end{enumerate}
\end{lem}

\begin{lem}
\label{lem:properties_2}
With the assumption on the links $L$, we have the following :
\begin{enumerate}
\item The map $j_{\sharp}^{2} \col H^{0}(D) \rightarrow H^{2}(D)$ in injective, the map $j_{\sharp}^{4} \col H^{2}(D) \rightarrow H^{4}(D)$ is surjective, $j_{\sharp}^{k} \col H^{k-2}(D) \rightarrow H^{k}(D)$ is an isomorphism for $k =1,3,5$.
\item The map $\gamma^{k} \col H^{k-2}(D) \rightarrow H^{k}(\widetilde{X})$ is injective for all $k \neq 4,6$ and $j^{k} \col H^{k}(\widetilde{X}) \rightarrow H^{k}(D)$ is surjective for all $k \neq 0,2$.
\end{enumerate}
\end{lem}

\begin{lem}
\label{lem:properties_3}
With the above assumptions we have the following :
\begin{enumerate}
\item $H^{k}(\widetilde{X}) \cong \ker j^{k} \oplus \im{\gamma^{k}}$ for $k=1,3,5$.
\item $\ker j^{2} \cap \im{\gamma^{2}} = 0$.
\end{enumerate}
\end{lem}

With the lemmas above the second term of the spectral sequences for the regular part and the links are given by
\[
\begin{array}{| c || c | c |}
\hline
\multicolumn{3}{|c|}{E_{2}^{r,s}(X_{reg})} \\
\hline
s=6            & \ker \gamma^{6}  & 0   \\ 
\hline
s=5            & 0                & \coker{\gamma^{5}}     \\ 
\hline
s=4            & \ker \gamma^{4}  & \coker{\gamma^{4}}    \\
\hline
s=3            & 0                & \coker{\gamma^{3}}   \\
\hline
s=2            & 0                & \coker{\gamma^{2}}    \\
\hline
s=1            & 0                & \coker{\gamma^{1}}    \\
\hline
s=0            & 0                & H^{0}(\widetilde{X})    \\
\hline
\hline
               & r=-1             & r=0                                     \\
\hline 
\end{array}  
\quad 
\begin{array}{| c || c | c |}
\hline
\multicolumn{3}{|c|}{E_{2}^{r,s}(L)} \\
\hline
s=6            & H^{4}(D)            & 0   \\ 
\hline
s=5            & 0                   & 0     \\ 
\hline
s=4            & \ker j^{4}_{\sharp} & 0    \\
\hline
s=3            & 0                   & 0   \\
\hline
s=2            & 0                   & \coker{j^{2}_{\sharp}}    \\
\hline
s=1            & 0                   & 0    \\
\hline
s=0            & 0                   & H^{0}(D)    \\
\hline
\hline
               & r=-1                & r=0 \\
\hline 
\end{array} 
\]

The computation of the cohomology of the intersection spaces involve a choice of complementary subspace $\coim{p}$, we detail here the choice we make.
\begin{itemize}
\item For the perversity $\overline{1}$, the map $j_{\sharp}^{2}$ is injective by lemma \ref{lem:properties_1}, we then have $\coim{1} = H^{0}(D)$ and $\coker{\gamma^{2}_{|\coim{1}}} = \coker{\gamma^{2}}$.
\item For the perversity $\overline{2}$, the map $j_{\sharp}^{3}$ is an isomorphism by lemma \ref{lem:properties_2}, we then also have $\coim{2} = H^{1}(D)$ and $\coker{\gamma^{3}_{|\coim{2}}} = \coker{\gamma^{3}}$.
\item For the perversity $\overline{3}$, there is no assumption on $j_{\sharp}^{4}$ and we chose a complementary subspace of $\ker j_{\sharp}^{4}$ which we denote by $\coim{3}$.
\item For the perversity $\overline{4}$, the map $j_{\sharp}^{5}$ is an isomorphism by lemma \ref{lem:properties_2}, we then also have $\coim{4} = H^{3}(D)$ and $\coker{\gamma^{5}_{|\coim{4}}} = \coker{\gamma^{5}}$.
\end{itemize}

Since the links of the singularities are simply connected five dimensional manifolds, by definition of the intersection spaces we have $\I{0}{X} \simeq \I{1}{X}$ and $\I{3}{X} \simeq \I{4}{X}$. Thus the second terms of the corresponding spectral sequences must be isomorphic, for now the corresponding second term for the associated spectral sequences are the following.
\[
\begin{array}{| c || c c c || c c c|} 
\hline
               & \multicolumn{3}{c||}{\EIbidg{2}{0}{r,s}{X}}          & \multicolumn{3}{c|}{\EIbidg{2}{1}{r,s}{X}} \\
\hline
s = 6          & \ker \gamma^{6} &  0                               & 0   & \ker \gamma^{6} &      0                           & 0\\
\hline
s = 5          & 0               &  \coker{\gamma^{5}}              & 0   & 0               & \coker{\gamma^{5}}               & 0 \\
\hline
s = 4          & \ker \gamma^{4} &  \coker{\gamma^{4}}              & 0   & \ker \gamma^{4} & \coker{\gamma^{4}}               & 0\\
\hline
s = 3          & 0               &  \coker{\gamma^{3}}              & 0   & 0               & \coker{\gamma^{3}}               & 0\\
\hline
s = 2          & 0               &  \coker{\gamma^{2}}              & 0   & 0               & \coker{\gamma^{2}}               & 0\\
\hline
s=1            & 0               &  \coker{\gamma^{1}}              & 0   & 0               & \ker j^{1}                       & 0\\
\hline
\hline
s = 0          & 0               & H^{0}(\widetilde{X}) & 0   & 0               & H^{0}(\widetilde{X}) &  0\\
\hline
\hline
               & r=-1            & r=0                       & r=1 & r=1             & r=0                              & r=1\\
\hline 
\end{array}  
\]

\[
\begin{array}{| c || c c c || c c c|} 
\hline
               & \multicolumn{3}{c||}{\EIbidg{2}{3}{r,s}{X}}        & \multicolumn{3}{c|}{\EIbidg{2}{4}{r,s}{X}} \\
\hline
s = 6          & \ker \gamma^{6} &  0                               & 0             & \ker \gamma^{6} &      0                           & 0\\
\hline
s = 5          & 0               &  \coker{\gamma^{5}}              & 0             & 0               & \coker{\gamma^{5}}               & 0 \\
\hline
s = 4          & 0               &  \coker{\gamma^{4}_{|\coim{3}}}  & 0             & 0               & \ker j^{4}                       & 0\\
\hline
s = 3          & 0               &  \ker j^{3}                      & 0             & 0               & \ker j^{3}                       & 0\\
\hline
s = 2          & 0               &  \ker j^{2}                      & \coker{j^{2}} & 0               & \ker j^{2}                       & \coker{j^{2}}\\
\hline
s=1            & 0               &  \ker j^{1}                      & 0             & 0               & \ker j^{1}                       & 0\\
\hline
\hline
s = 0          & 0               & H^{0}(\widetilde{X}) & 0             & 0               & H^{0}(\widetilde{X}) &  0\\
\hline
\hline
               & r=-1            & r=0                              & r=1           & r=1             & r=0                              & r=1\\
\hline 
\end{array}  
\]
We then need to show that $\EIbidg{2}{0}{r,s}{X} \cong \EIbidg{2}{1}{r,s}{X}$ and $\EIbidg{2}{3}{r,s}{X} \cong \EIbidg{2}{4}{r,s}{X}$. The first isomorphism is given by the isomorphism
\[
H^{1}(\widetilde{X}) \cong \ker j^{1} \oplus \im{\gamma^{1}}
\]
from the lemma \ref{lem:properties_3}, we then have $\coker{\gamma^{1}} \cong \ker j^{1}$. 

For the second isomorphism we need to show that 
\[
\coker{\gamma^{4}_{|\coim{3}}} \cong \ker j^{4}.
\]
Which is given by the following lemma

\begin{lem}
\label{lem:properties_4}
We have the following isomorphism
\[
H^{4}(\widetilde{X}) \cong \ker j^{4} \oplus \im{\gamma^{4}_{|\coim{3}}}.
\]
\end{lem}

\begin{proof}
Denote by $(\ker j^{4})^{\perp}$ a complementary subspace of $\ker j^{4} \subset H^{4}(\widetilde{X})$. The maps $j_{\sharp}^{4}$ and $j^{4}$ are surjective by lemma \ref{lem:properties_1}. We then have the following commutative diagram
\[
\begin{tikzpicture}[description/.style={fill=white,inner sep=2pt}]
\matrix (m)[matrix of math nodes, row sep=3em, column sep=5em, text height=2ex, text depth=0.25ex]
{ H^{2}(D) \cong \ker j_{\sharp}^{4}  \oplus \coim{3} &  H^{4}(\widetilde{X}) \cong\ker j^{4} \oplus (\ker j^{4})^{\perp}  \\
                                       &  H^{4}(D) \\};

\path[->]
(m-1-1) edge node[auto] {$\gamma^{4}$} (m-1-2);
\path[->>]
(m-1-2) edge node[auto] {$j^{4}$} (m-2-2);
\path[->>]
(m-1-1) edge node[description] {$j_{\sharp}^{4}$} (m-2-2);
\end{tikzpicture}
\]

By definition of $\coim{3}$  we have $\gamma^{4}_{|} \col \coim{3} \rightarrow (\ker j^{4})^{\perp}$. The commutative diagram restricts then to the following commutative diagram where the restrictions $j_{\sharp |}^{4}$ and $j^{4}_{|}$ are isomorphisms. Which finishes the proof.
\[
\begin{tikzpicture}[description/.style={fill=white,inner sep=2pt}]
\matrix (m)[matrix of math nodes, row sep=3em, column sep=5em, text height=2ex, text depth=0.25ex]
{ \coim{3} & (\ker j^{4})^{\perp}  \\
           & H^{4}(D) \\};

\path[->]
(m-1-1) edge node[auto] {$\gamma^{4}_{|}$} (m-1-2);
\path[->]
(m-1-2) edge node[auto] {$j^{4}_{|}$} (m-2-2);
\path[->]
(m-1-1) edge node[description] {$j_{\sharp |}^{4}$} (m-2-2);
\end{tikzpicture}
\]
\end{proof}

The second terms of the spectral sequences of $\EIbidg{2}{p}{r,s}{X}$ for $\overline{p} \in \lbrace \overline{0}, \overline{2}, \overline{4} \rbrace$ are finally.

\[
\hspace{-4.25em}
\begin{array}{| c || c c c || c c c|| c  c  c |} 
\hline
               & \multicolumn{3}{c||}{\EIbidg{2}{0}{r,s}{X}}        & \multicolumn{3}{c||}{\EIbidg{2}{2}{r,s}{X}}                              & \multicolumn{3}{c|}{\EIbidg{2}{4}{r,s}{X}} \\
\hline
s = 6          & \ker \gamma^{6} &  0                               & 0   & \ker \gamma^{6} &      0                           & 0             & \ker \gamma^{6} &  0                               &  0 \\
\hline
s = 5          & 0               &  \coker{\gamma^{5}}              & 0   & 0               & \coker{\gamma^{5}}               & 0             & 0               &  \coker{\gamma^{5}}              &  0\\
\hline
s = 4          & \ker \gamma^{4} &  \coker{\gamma^{4}}              & 0   & \ker \gamma^{4} & \coker{\gamma^{4}}               & 0             & 0               &  \ker j^{4}                      & 0 \\
\hline
s = 3          & 0               &  \coker{\gamma^{3}}              & 0   & 0               & \coker{\gamma^{3}}               & 0             & 0               &  \ker j^{3}                      & 0\\
\hline
s = 2          & 0               &  \coker{\gamma^{2}}              & 0   & 0               & \ker j^{2}                       & \coker{j^{2}} & 0               &  \ker j^{2}                      & \coker{j^{2}}\\
\hline
s=1            & 0               &  \coker{\gamma^{1}}              & 0   & 0               & \ker j^{1}                       & 0             & 0               &  \ker j^{1}                      & 0 \\
\hline
\hline
s = 0          & 0               & H^{0}(\widetilde{X}) & 0   & 0               & H^{0}(\widetilde{X}) & 0             & 0               & H^{0}(\widetilde{X}) & 0 \\
\hline
\hline
               & r=-1            & r=0                              & r=1 & r=-1            & r=0                              & r=1           & r=-1            & r=0                              & r=1\\
\hline 
\end{array}  
\]

We are now ready to state the following theorem.
\subsection{Statement and proof}

In \cite[theorem E]{Chataur2012} it is proved that any nodal hypersurface $X$ in $\mathbf{C}P^{4}$ is GM-intersection-formal, meaning that their perverse rational models $\IA{\bullet}{X}$ is quasi-isomorphic to their intersection cohomology algebras $\IH{\bullet}{\ast}{X}$. This result is extended in \cite[theorem 4.5]{Chataur} to the case of complex projective varieties of dimension $n$ with only isolated singularities and $(n-2)$-connected links using mixed Hodge structures. We show, using the same ideas, that for $X$ a complex projective algebraic 3-fold with isolated singularities and simply connected links, the intersection spaces are formal topological spaces.

\begin{thm}
\label{thm:3fold_formal}
Let $X$ be a complex projective algebraic 3-fold with isolated singularities and denote by $\Sigma = \lbrace \sigma_{1}, \dots, \sigma_{\nu}\rbrace$ the singular locus of $X$. Assume that there is a resolution of singularities $f \col \widetilde{X} \rightarrow X$ such that the exceptional divisor $D := f^{-1}(\Sigma)$ is smooth and the link $L_{i}$ of $\sigma_{i}$ in $X$, for all $\sigma_{i} \in \Sigma$, is simply connected. Then $\I{p}{X}$ is formal over $\mathbf{C}$ for all $\overline{p} \in \lbrace \overline{\infty}, \overline{0}, \dots, \overline{4} \rbrace$.
\end{thm}

By theorem \ref{thm:AI_to_EI1} there is a string of quasi-isomorphisms of coperverse cdga's from $\AI{\bullet}{X}\otimes \mathbf{C}$ to $\EI{1}{\bullet}{X}\otimes \mathbf{C}$. Moreover we have $\EIbidg{2}{\bullet}{\ast,\ast}{X} \cong \HI{\bullet}{\ast}{X}$. We follow this pattern
\begin{enumerate}
\item We define a bigraded differential algebra $(\FI{p}{r,s}{X}, \diff{p}{r,s})$ step by step for the perversities $\overline{4}$, $\overline{2}$ and $\overline{0}$.
\begin{itemize}
\item When needed, we then define the poset map $\pomap{p}{q} \col \FI{p}{r,s}{X} \rightarrow \FI{q}{r,s}{X}$ and show its compatibility with the product and the differential.
\end{itemize}
\item We define the quasi-isomorphisms
\[
(\EIbidg{1}{p}{r,s}{X},\difb{p}{r,s}) \overset{\psip{p}{r,s}}{\longleftarrow} (\FI{p}{r,s}{X},\diff{p}{r,s}) \overset{\phip{p}{r,s}}{\longrightarrow} (\EIbidg{2}{p}{r,s}{X},0)
\]
and check their compatibility with the products and differentials.
\begin{itemize}
\item When needed, we then check the compatibility of the maps $\psip{\bullet}{\ast,\ast}$ and $\phip{\bullet}{\ast,\ast}$  with the poset map $\pomap{p}{q} \col \FI{p}{r,s}{X} \rightarrow \FI{q}{r,s}{X}$.
\end{itemize}
\end{enumerate}

\subsubsection{The top perversity}
We begin with the top perversity $\overline{t} = \overline{4}$. We define the bigraded differential algebra $(\FI{4}{r,s}{X}, \diff{4}{r,s})$.
\[
\begin{array}{| c || c c c |} 
\hline
s = 6          & H^{4}(D) & H^{6}(\widetilde{X}) &  0  \\
\hline
s = 5          & H^{3}(D) &  H^{5}(\widetilde{X})&  0  \\
\hline
s = 4          & 0        &  \ker j^{4}          &  0  \\
\hline
s = 3          & 0        &  \ker j^{3}          &  0  \\
\hline
s = 2          & 0        &  \ker j^{2}          & \D{(\ker \gamma^{4})} \otimes dt  \\
\hline
s=1            & 0        &  \ker j^{1}          & 0  \\
\hline
\hline
s = 0          & 0        & H^{0}(\widetilde{X}) & 0 \\
\hline
\hline
\FI{4}{r,s}{X} & r=-1     & r=0                  & r=1 \\
\hline 
\end{array}  
\]
The only non-trivial differentials are $\diff{4}{-1,s} \col H^{s-2}(D) \rightarrow H^{s}(\widetilde{X})$ given by $\diff{4}{-1,s} = \gamma^{s}$ for $s=5,6$. The algebra structure is defined by $\FI{4}{0,s}{X} \times \FI{4}{0,s'}{X} \rightarrow \FI{4}{0,s+s'}{X}$.

Let's now define the map $\psip{4}{\ast,\ast} \col \FI{4}{\ast,\ast}{X} \rightarrow \EIbidg{1}{4}{\ast,\ast}{X}$. Recall that we have the following first term for the weight spectral sequence.

\noindent\makebox[\textwidth]{
\(
\begin{array}{| c || c c c c c |} 
\hline
s \geq 5        & H^{s-2}(D)\otimes \mathbf{Q}[t] & \rightarrow & \mathcal{I}^{s}_{0} \oplus H^{s-2}(D)\otimes \mathbf{Q}[t]dt & \rightarrow & H^{s}(D)\otimes \mathbf{Q}[t]dt  \\
\hline
1 \leq s \leq 4 & H^{s-2}(D)\otimes \mathbf{Q}[t]t & \rightarrow & \mathcal{I}^{s}_{1} \oplus H^{s-2}(D)\otimes \mathbf{Q}[t]dt & \rightarrow & H^{s}(D)\otimes \mathbf{Q}[t]dt  \\
\hline
\hline
s = 0          & 0                                            & & \mathcal{I}^{0}_{0} & \rightarrow & H^{0}(D)\otimes \mathbf{Q}[t]dt  \\
\hline
\hline
\EIbidg{1}{4}{r,s}{X}               & r=-1                                         &             & r=0                              &             & r=1 \\
\hline 
\end{array}  
\)}

For $r=s=0$, the map $\psip{4}{0,0}$ is the identity map.For $r=0$, $s>0$, the map $\psip{4}{0,s} \col \FI{4}{0,s}{X} \rightarrow \EIbidg{1}{4}{0,s}{X}$ is defined to be
\[
\psip{4}{0,s}(x) := (x, j^{s}(x)).
\]

For $r=-1$, $\psip{4}{-1,s}$ is defined to be the canonical inclusion. 

By lemma \ref{lem:DP_ker_coker} we have $\D{(\ker \gamma^{4})} \cong \coker{j^{2}} \subset H^{2}(D)$, we then define $\psip{4}{1,2}$ to be the injective map 
\[
\psip{4}{1,2} \col \D{(\ker \gamma^{4})} \otimes dt \longrightarrow \EIbidg{1}{4}{1,2}{X} = H^{2}(D)\otimes \mathbf{Q}[t]dt.
\]

By definition $\mathcal{I}^{s}_{k}$, $k \in \lbrace 0,1 \rbrace$, is the vector space given by the following pullback square.
\[
\begin{tikzpicture}
\matrix (m)[matrix of math nodes, row sep=3em, column sep=5em, text height=2ex, text depth=0.25ex]
{ \mathcal{I}^{s}_{k}  &  H^{s}(D)\otimes \mathbf{Q}[t]t^{k}  \\
  H^{s}(\widetilde{X}) &  H^{s}(D) \\};

\path[->]
(m-1-1) edge node[auto] {} (m-1-2);
\path[->]
(m-2-1) edge node[auto] {$j^{s}$} (m-2-2);
\path[->]
(m-1-1) edge node[above=1em, right=1em] {$\ulcorner$} (m-2-1);
\path[->]
(m-1-2) edge node[auto] {$\delta_{1}$} (m-2-2);
\end{tikzpicture}
\]
We have $\mathcal{I}^{s}_{1} \subset \mathcal{I}^{s}_{0}$, the map $\psip{4}{0,s}(x) := (x, j^{s}(x))$ is then compatible with the algebra structure of $\FI{4}{\ast,\ast}{X}$. The commutativity of the following diagrams
\[
\begin{tikzpicture}[injection/.style={right hook->,fill=white, inner sep=2pt}]
\matrix (m)[matrix of math nodes, row sep=3em, column sep=5em, text height=2ex, text depth=0.25ex]
{ \FI{4}{-1,s}{X}        &  \FI{4}{0,s}{X}  \\
  \EIbidg{1}{4}{-1,s}{X} &  \EIbidg{1}{4}{0,s}{X} \\};

\path[->]
(m-1-1) edge node[auto] {$\diff{4}{-1,s} = \gamma^{s}$} (m-1-2);
\path[->]
(m-2-1) edge node[auto,swap] {$\difb{4}{-1,s}$} (m-2-2);
\path[injection]
(m-1-1) edge node[auto,swap] {$\psip{4}{-1,s}$} (m-2-1);
\path[->]
(m-1-2) edge node[auto] {$\psip{4}{0,s}$} (m-2-2);
\end{tikzpicture}
\quad
\begin{tikzpicture}[injection/.style={right hook->,fill=white, inner sep=2pt}]
\matrix (m)[matrix of math nodes, row sep=3em, column sep=5em, text height=2ex, text depth=0.25ex]
{ \ker j^{2}             &  \D{(\ker \gamma^{4})} \otimes dt  \\
  \EIbidg{1}{4}{0,2}{X}  &  \EIbidg{1}{4}{1,2}{X} \\};

\path[->]
(m-1-1) edge node[auto] {$0$} (m-1-2);
\path[->]
(m-2-1) edge node[auto,swap] {$\difb{4}{-1,s}$} (m-2-2);
\path[->]
(m-1-1) edge node[auto,swap] {$\psip{4}{0,2}$} (m-2-1);
\path[->]
(m-1-2) edge node[auto] {$\psip{4}{1,2}$} (m-2-2);
\end{tikzpicture}
\]
concludes that we have a quasi-isomorphism $\psip{4}{\ast,\ast} \col \FI{4}{\ast,\ast}{X} \rightarrow \EIbidg{1}{4}{\ast,\ast}{X}$.

We now detail the map $\phip{4}{\ast,\ast} \col \FI{4}{\ast,\ast}{X} \rightarrow \EIbidg{2}{4}{\ast,\ast}{X}$. 

For $r=-1$, $\phip{4}{-1,s}$ is non zero only for $s=6$ where it is the projection $H^{4}(D) \twoheadrightarrow \ker \gamma^{6}$.

For $r=0$, since $\FI{4}{0,s}{X} = \ker \difb{4}{0,s}$ for all $s$, we define the map $\phip{4}{0,s}$ to be the surjection $\phip{4}{0,s} \col \ker \difb{4}{0,s} \twoheadrightarrow \EIbidg{2}{4}{0,s}{X}$.

For $r=1$, the assignation $\D{(\ker \gamma^{4})} \otimes dt \mapsto \coker{j^{2}}$ defines $\phip{4}{1,2}$ and $\phip{4}{1,s}$ is zero for any other $s$.

Since we have $\ker \difb{4}{0,s} \times \ker \difb{4}{0,s'} \rightarrow \ker \difb{4}{0,s+s'}$, the map $\phip{4}{\ast,\ast}$ is compatible with the algebra structure of $\FI{4}{\ast,\ast}{X}$.

The map $\phip{4}{\ast,\ast}$ is also compatible with the two non zero differentials of $\FI{4}{\ast,\ast}{X}$ since the two following diagrams are commutative.
\[
\begin{tikzpicture}[injection/.style={right hook->,fill=white, inner sep=2pt}]
\matrix (m)[matrix of math nodes, row sep=3em, column sep=5em, text height=2ex, text depth=0.25ex]
{ H^{4}(D)        &  H^{6}(\widetilde{X})  \\
  \ker \gamma^{6} &  0 \\};

\path[->]
(m-1-1) edge node[auto] {$\gamma^{6}$} (m-1-2);
\path[->]
(m-2-1) edge  (m-2-2);
\path[->>]
(m-1-1) edge node[auto,swap] {$\phip{4}{-1,6}$} (m-2-1);
\path[->]
(m-1-2) edge (m-2-2);
\end{tikzpicture}
\quad
\begin{tikzpicture}[injection/.style={right hook->,fill=white, inner sep=2pt}]
\matrix (m)[matrix of math nodes, row sep=3em, column sep=5em, text height=2ex, text depth=0.25ex]
{ H^{3}(D) &  H^{5}(\widetilde{X})  \\
  0        &  \EIbidg{2}{4}{0,5}{X} = \coker{\gamma^{5}} \\};

\path[injection]
(m-1-1) edge node[auto] {$\gamma^{5}$} (m-1-2);
\path[->>]
(m-1-2) edge node[auto] {$\phip{4}{0,5}$} (m-2-2);
\path[->]
(m-2-1) edge  (m-2-2);
\path[->]
(m-1-1) edge  (m-2-1);
\end{tikzpicture}
\]

We then have a quasi-isomorphism of algebras
\[
(\EIbidg{1}{4}{r,s}{X},\difb{4}{r,s}) \overset{\psip{4}{r,s}}{\longleftarrow} (\FI{4}{r,s}{X},\diff{4}{r,s}) \overset{\phip{4}{r,s}}{\longrightarrow} (\EIbidg{2}{4}{r,s}{X},0).
\]

\subsubsection{The middle perversity}
We define the bigraded differential algebra $(\FI{2}{r,s}{X},\diff{2}{r,s})$ as the sub-algebra of $(\EI{1}{2}{X},\difb{2}{r,s})$ given by
\[
\hspace{-1em}
\begin{array}{| c || c c c |} 
\hline
s = 6          & H^{4}(D) & H^{6}(\widetilde{X}) &  0  \\
\hline
s = 5          & H^{3}(D) &  H^{5}(\widetilde{X})&  0  \\
\hline
s = 4          & H^{2}(D) &  H^{4}(\widetilde{X})&  0  \\
\hline
s = 3          & 0        &  \ker j^{3}          &  0  \\
\hline
s = 2          & 0        &  \ker j^{2}          & \D{(\ker \gamma^{4})} \otimes dt  \\
\hline
s=1            & 0        &  \ker j^{1}          & 0  \\
\hline
\hline
s = 0          & 0        & H^{0}(\widetilde{X}) & 0 \\
\hline
\hline
\FI{2}{r,s}{X} & r=-1     & r=0                  & r=1 \\
\hline 
\end{array}  
\]
Where $(\EI{1}{2}{X},\difb{2}{r,s})$ is given by

\noindent\makebox[\textwidth]{
\(
\begin{array}{| c || c c c c c |} 
\hline
s \geq 3        & H^{s-2}(D)\otimes \mathbf{Q}[t] & \rightarrow & \mathcal{I}^{s}_{0} \oplus H^{s-2}(D)\otimes \mathbf{Q}[t]dt & \rightarrow & H^{s}(D)\otimes \mathbf{Q}[t]dt  \\
\hline
1 \leq s \leq 2 & H^{s-2}(D)\otimes \mathbf{Q}[t]t & \rightarrow & \mathcal{I}^{s}_{1} \oplus H^{s-2}(D)\otimes \mathbf{Q}[t]dt & \rightarrow & H^{s}(D)\otimes \mathbf{Q}[t]dt  \\
\hline
\hline
s = 0          & 0                                            & & \mathcal{I}^{0}_{0} & \rightarrow  & H^{0}(D)\otimes \mathbf{Q}[t]dt  \\
\hline
\hline
\EIbidg{1}{2}{r,s}{X}               & r=-1                                         &             & r=0                              &             & r=1 \\
\hline 
\end{array}  
\)}

Compared to $\FI{4}{\ast,\ast}{X}$, we added $H^{2}(D)$ in bidegree $(-1,4)$ and replaced $\ker j^{4}$ by $H^{4}(\widetilde{X})$ in bidegree $(0,4)$, both are related by a new non-trivial differential $\diff{2}{-1,4} = \gamma^{4}$. 

The algebra structure is still non-trivial only for $r=0$, with 
\[
\FI{2}{0,s}{X} \times \FI{2}{0,s'}{X} \rightarrow \FI{2}{0,s+s'}{X}.
\]

The map $\pomap{4}{2} \col \FI{4}{\ast,\ast}{X} \rightarrow \FI{2}{\ast,\ast}{X}$ is then the canonical inclusion, which is clearly compatible with the differential and the algebra structure.

To construct $\psip{2}{\ast,\ast} \col \FI{2}{\ast,\ast}{X} \rightarrow \EIbidg{1}{2}{\ast,\ast}{X}$, we extend $\psip{4}{\ast,\ast}$, meaning that $\psip{2}{-1,s}$ is the inclusion, $\psip{2}{0,s}(x) = (x, j^{s}(x))$ and $\psip{2}{1,s} = \psip{4}{1,s}$. The algebra structure is preserved by $\psip{2}{0,s}$ and the following diagram commutes
\[
\begin{tikzpicture}[injection/.style={right hook->,fill=white, inner sep=2pt}]
\matrix (m)[matrix of math nodes, row sep=3em, column sep=5em, text height=2ex, text depth=0.25ex]
{ H^{2}(D)                &  H^{4}(\widetilde{X})  \\
  \EIbidg{1}{2}{-1,4}{X}  &  \EIbidg{1}{2}{0,4}{X} \\};

\path[->]
(m-1-1) edge node[auto] {$\diff{2}{-1,4} = \gamma^{4}$} (m-1-2);
\path[->]
(m-2-1) edge node[auto,swap] {$\difb{2}{-1,4}$} (m-2-2);
\path[injection]
(m-1-1) edge node[auto,swap] {$\psip{2}{-1,4}$} (m-2-1);
\path[->]
(m-1-2) edge node[auto] {$\psip{2}{0,4}$} (m-2-2);
\end{tikzpicture}
\]
The rest being the same as for the top perversity, we have a quasi-isomorphism
\[
\psip{2}{\ast,\ast} \col \FI{2}{\ast,\ast}{X} \longrightarrow \EIbidg{1}{2}{\ast,\ast}{X}.
\]

We now construct $\phip{2}{\ast,\ast} \col \FI{2}{\ast,\ast}{X} \rightarrow \EIbidg{2}{2}{\ast,\ast}{X}$. 

First of all nothing changes for $r=1$ and $\phip{2}{1,s}=\phip{4}{1,s}$.

For $r=-1$, $\phip{2}{-1,s}$ is non zero only for $s=4,6$ where it is the projection $H^{s-2}(D) \twoheadrightarrow \ker \gamma^{s}$.

For $r=0$, since $\FI{2}{0,s}{X} = \ker \difb{2}{0,s}$ for all $s \neq 3$, we define the map $\phip{2}{0,s}$ to be the surjection $\phip{2}{0,s} \col \ker \difb{2}{0,s} \twoheadrightarrow \EIbidg{2}{2}{0,s}{X}$. For $s=3$, by lemma \ref{lem:properties_3} we have $\ker j^{3} \cong \coker{\gamma^{3}}$, this isomorphism defines $\phip{2}{0,3}$.

For $s=4,6$ or $s=5$, the following diagrams commute
\[
\begin{tikzpicture}[injection/.style={right hook->,fill=white, inner sep=2pt}]
\matrix (m)[matrix of math nodes, row sep=3em, column sep=5em, text height=2ex, text depth=0.25ex]
{ H^{s-2}(D)        &  H^{s}(\widetilde{X})  \\
  \ker \gamma^{s}   &  \coker{\gamma^{s}} \\};

\path[->]
(m-1-1) edge node[auto] {$\gamma^{s}$} (m-1-2);
\path[->]
(m-2-1) edge node[auto,swap] {$0$} (m-2-2);
\path[->>]
(m-1-1) edge node[auto,swap] {$\phip{2}{-1,s}$} (m-2-1);
\path[->>]
(m-1-2) edge node[auto] {$\phip{2}{0,s}$} (m-2-2);
\end{tikzpicture}
\quad
\begin{tikzpicture}[injection/.style={right hook->,fill=white, inner sep=2pt}]
\matrix (m)[matrix of math nodes, row sep=3em, column sep=5em, text height=2ex, text depth=0.25ex]
{ H^{3}(D) &  H^{5}(\widetilde{X})  \\
  0        &  \EIbidg{2}{2}{0,5}{X} = \coker{\gamma^{5}} \\};

\path[injection]
(m-1-1) edge node[auto] {$\gamma^{5}$} (m-1-2);
\path[->>]
(m-1-2) edge node[auto] {$\phip{2}{0,5}$} (m-2-2);
\path[->]
(m-2-1) edge  (m-2-2);
\path[->]
(m-1-1) edge  (m-2-1);
\end{tikzpicture}
\]
So $\phip{2}{\ast,\ast}$ is compatible with the differential. 

To see its compatibility with the algebra structure of $\FI{2}{\ast,\ast}{X}$ we have to check the commutativity of the following diagram
\[
\begin{tikzpicture}
\matrix (m)[matrix of math nodes, row sep=3em, column sep=5em, text height=2ex, text depth=0.25ex]
{ \ker j^{1} \times \ker j^{2}        &  \ker j^{3}  \\
  \coker{\gamma^{1}} \times \coker{\gamma^{2}}  &  \coker{\gamma^{3}} \\};

\path[->]
(m-1-1) edge node[auto] {} (m-1-2);
\path[->]
(m-2-1) edge node[auto,swap] {} (m-2-2);
\path[->]
(m-1-1) edge node[auto,swap] {} (m-2-1);
\path[->]
(m-1-2) edge node[auto] {$\cong$} (m-2-2);
\end{tikzpicture}
\]

We then have a quasi-isomorphism of algebras
\[
(\EIbidg{1}{2}{r,s}{X},\difb{2}{r,s}) \overset{\psip{2}{r,s}}{\longleftarrow} (\FI{2}{r,s}{X},\diff{2}{r,s}) \overset{\phip{2}{r,s}}{\longrightarrow} (\EIbidg{2}{2}{r,s}{X},0).
\]

We now check the commutativity of the following diagram
\[
\begin{tikzpicture}[injection/.style={right hook->,fill=white, inner sep=2pt}]
\matrix (m)[matrix of math nodes, row sep=3em, column sep=5em, text height=2ex, text depth=0.25ex]
{ (\EIbidg{1}{4}{r,s}{X},\difb{4}{r,s})  &  (\FI{4}{r,s}{X},\diff{4}{r,s})  & (\EIbidg{2}{4}{r,s}{X},0) \\
  (\EIbidg{1}{2}{r,s}{X},\difb{2}{r,s})  &  (\FI{2}{r,s}{X},\diff{2}{r,s})  & (\EIbidg{2}{2}{r,s}{X},0)\\};

\path[->]
(m-1-2) edge node[auto,swap] {$\psip{4}{r,s}$} (m-1-1);
\path[->]
(m-1-2) edge node[auto] {$\phip{4}{r,s}$} (m-1-3);

\path[->]
(m-2-2) edge node[auto] {$\psip{2}{r,s}$} (m-2-1);
\path[->]
(m-2-2) edge node[auto,swap] {$\phip{2}{r,s}$} (m-2-3);

\path[->]
(m-1-1) edge node[auto,swap] {$EI_{1}(\pomap{4}{2})$} (m-2-1);
\path[injection]
(m-1-2) edge node[auto] {$\pomap{4}{2}$} (m-2-2);
\path[->]
(m-1-3) edge node[auto] {$EI_{2}(\pomap{4}{2})$} (m-2-3);
\end{tikzpicture}
\]
The only differences between $\EIbidg{i}{4}{r,s}{X}$ and $\EIbidg{i}{2}{r,s}{X}$, $i=1,2$, arise for $s=3,4$. We then only check these cases. 

The only square that does not trivially commutes for $s=3$ is the following
\[
\begin{tikzpicture}[injection/.style={right hook->,fill=white, inner sep=2pt}]
\matrix (m)[matrix of math nodes, row sep=3em, column sep=5em, text height=2ex, text depth=0.25ex]
{ \mathcal{I}^{3}_{1} \oplus H^{1}(D)\otimes \mathbf{Q}[t]dt  &  \ker j^{3}  & \ker j^{3} \\
  \mathcal{I}^{3}_{0} \oplus H^{1}(D)\otimes \mathbf{Q}[t]dt  &  \ker j^{3}  & \coker{\gamma^{3}}\\};

\path[->]
(m-1-2) edge node[auto,swap] {$\psip{4}{0,3}$} (m-1-1);
\path[->]
(m-1-2) edge node[auto] {$\phip{4}{0,3}$} (m-1-3);

\path[->]
(m-2-2) edge node[auto] {$\psip{2}{0,3}$} (m-2-1);
\path[->]
(m-2-2) edge node[auto,swap] {$\phip{2}{0,3}$} (m-2-3);

\path[injection]
(m-1-1) edge node[auto,swap] {$EI_{1}(\pomap{4}{2})$} (m-2-1);
\path[->]
(m-1-2) edge node[auto] {$\pomap{4}{2}$} node[auto,swap] {=} (m-2-2);
\path[->]
(m-1-3) edge node[auto] {$EI_{2}(\pomap{4}{2})$} (m-2-3);
\end{tikzpicture}
\]
The left hand square commutes because $\im{\psip{4}{0,3}} \subset \mathcal{I}^{3}_{1}$, $\im{\psip{2}{0,3}} \subset \mathcal{I}^{3}_{0}$ and the fact that $\mathcal{I}^{3}_{1} \subset \mathcal{I}^{3}_{0}$. The right hand square commutes because of the isomorphism $\ker j^{3} \cong \coker{\gamma^{3}}$.

For $s=4$, the only square that does not trivially commutes is the following
\[
\begin{tikzpicture}[injection/.style={right hook->,fill=white, inner sep=2pt}]
\matrix (m)[matrix of math nodes, row sep=3em, column sep=5em, text height=2ex, text depth=0.25ex]
{ \mathcal{I}^{4}_{1} \oplus H^{2}(D)\otimes \mathbf{Q}[t]dt  &  \ker j^{4}  & \ker j^{4} \\
  \mathcal{I}^{4}_{0} \oplus H^{2}(D)\otimes \mathbf{Q}[t]dt  &  H^{4}(\widetilde{X})  & \coker{\gamma^{4}}\\};

\path[->]
(m-1-2) edge node[auto,swap] {$\psip{4}{0,4}$} (m-1-1);
\path[->]
(m-1-2) edge node[auto] {$\phip{4}{0,4}$} (m-1-3);

\path[->]
(m-2-2) edge node[auto] {$\psip{2}{0,4}$} (m-2-1);
\path[->]
(m-2-2) edge node[auto,swap] {$\phip{2}{0,4}$} (m-2-3);

\path[injection]
(m-1-1) edge node[auto,swap] {$EI_{1}(\pomap{4}{2})$} (m-2-1);
\path[injection]
(m-1-2) edge node[auto] {$\pomap{4}{2}$} (m-2-2);
\path[->]
(m-1-3) edge node[auto] {$EI_{2}(\pomap{4}{2})$} (m-2-3);
\end{tikzpicture}
\]
The left hand square commutes for the same reason that for $s=3$. We then consider the right hand square. By lemma \ref{lem:properties_4} we have $H^{4}(\widetilde{X}) \cong \ker j^{4} \oplus \im{\gamma^{4}_{|\coim{3}}}$, moreover we have $\im{\gamma^{4}_{|\coim{3}}} \subset \im{\gamma^{4}}$, this implies that 
\[
\ker j^{4} \cap \im{\gamma^{4}} \neq \lbrace 0 \rbrace.
\]
We may then find a direct sum decomposition 
\[
\ker j^{4} = (\ker j^{4} \cap \im{\gamma^{4}}) \oplus C
\]
and defines a map $\ker j^{4} \rightarrow C$ by projection on the second summand. We then have
\[
H^{4}(\widetilde{X}) \cong (\ker j^{4} \cap \im{\gamma^{4}}) \oplus C \oplus \im{\gamma^{4}_{|\coim{3}}},
\]
the maps $EI_{2}(\pomap{4}{2})$ and $\phip{2}{0,4}$ then send the summand $(\ker j^{4} \cap \im{\gamma^{4}}) \oplus \im{\gamma^{4}_{|\coim{3}}}$ to zero and $C$ to its class in $\coker{\gamma^{4}}$. Which makes the right hand square commute.

\subsubsection{The zero perversity}
We define the bigraded differential algebra $(\FI{0}{r,s}{X},\diff{0}{r,s})$ as the sub-algebra of $(\EI{1}{0}{X},\difb{0}{r,s})$ given by
\[
\begin{array}{| c || c c c |} 
\hline
s = 6          & H^{4}(D) & H^{6}(\widetilde{X}) &  0  \\
\hline
s = 5          & H^{3}(D) &  H^{5}(\widetilde{X})&  0  \\
\hline
s = 4          & H^{2}(D) &  H^{4}(\widetilde{X})&  0  \\
\hline
s = 3          & 0        &  \ker j^{3}          &  0  \\
\hline
s = 2          & 0        &  \D{(\ker j^{4})} \oplus  \D{(\ker \gamma^{4})} \otimes t        & \D{(\ker \gamma^{4})} \otimes dt  \\
\hline
s=1            & 0        &  \ker j^{1}          & 0  \\
\hline
\hline
s = 0          & 0        & H^{0}(\widetilde{X}) & 0 \\
\hline
\hline
\FI{0}{r,s}{X} & r=-1     & r=0                  & r=1 \\
\hline 
\end{array}  
\]
Where $(\EI{1}{0}{X},\difb{0}{r,s})$ is given by

\noindent\makebox[\textwidth]{
\(
\begin{array}{| c || c c c c c |} 
\hline
s \geq 1        & H^{s-2}(D)\otimes \mathbf{Q}[t] & \rightarrow & \mathcal{I}^{s}_{0} \oplus H^{s-2}(D)\otimes \mathbf{Q}[t]dt & \rightarrow & H^{s}(D)\otimes \mathbf{Q}[t]dt  \\
\hline
\hline
s = 0          & 0                                            & & \mathcal{I}^{0}_{0} & \rightarrow & H^{0}(D)\otimes \mathbf{Q}[t]dt  \\
\hline
\hline
\EIbidg{1}{0}{r,s}{X}               & r=-1                                         &             & r=0                              &             & r=1 \\
\hline 
\end{array}  
\)}

Compared to $\FI{2}{\ast,\ast}{X}$, we added $\D{(\ker \gamma^{4})} \otimes t$ and replaced $\ker j^{2}$ by $\D{(\ker j^{4})}$ in bidegree $(0,2)$. There is also a new differential $\diff{0}{0,2} \col \D{(\ker \gamma^{4})} \otimes t \rightarrow  \D{(\ker \gamma^{4})} \otimes dt$ which is differentiation with respect to $t$.

The algebra structure is non trivial only for $r=0$ where we have
\[
\begin{cases}
(\D{(\ker \gamma^{4})} \otimes t) \times \FI{0}{0,s}{X} \longrightarrow 0 & \forall \, s, \\
\FI{0}{0,s}{X} \times \FI{0}{0,s'}{X} \longrightarrow \FI{0}{0,s+s'}{X} & \text{otherwise.} \\
\end{cases}
\]

We now define $\pomap{2}{0} \col \FI{2}{\ast,\ast}{X} \rightarrow \FI{2}{\ast,\ast}{X}$. For $s \geq 3$, there is no changes and $\pomap{2}{0}$ is the identity, same if $s=0,1$. For $s=2$, by lemma \ref{lem:properties_3} we have $\ker j^{2} \cap \im{\gamma^{2}} = 0$ so we have the inclusion $\ker j^{2} \rightarrow \D{(\ker j^{4})}$. The map $\pomap{2}{0}$ is then an inclusion and is compatible with the differential and the algebra structure.

We now construct $\psip{0}{\ast,\ast} \col \FI{0}{\ast,\ast}{X} \rightarrow \EIbidg{1}{0}{\ast,\ast}{X}$. Since we have $\D{(\ker j^{4})} \oplus \D{(\ker \gamma^{4})} \otimes t \subset \mathcal{I}^{2}_{0}$ there is no difference between $\psip{0}{\ast,\ast}$ and $\psip{2}{\ast,\ast}$ and the definition is the same. We then have a quasi-isomorphism
\[
\psip{0}{\ast,\ast} \col \FI{0}{\ast,\ast}{X} \rightarrow \EIbidg{1}{0}{\ast,\ast}{X}.
\]

We define $\phip{0}{\ast,\ast} \col \FI{0}{\ast,\ast}{X} \rightarrow \EIbidg{2}{0}{\ast,\ast}{X}$, for $s \geq 3$ there is no difference with the middle perversity. If $s=2$ then we define $\phip{0}{0,2}$ by $\D{(\ker j^{4})} \mapsto \coker{\gamma^{2}}$ and $ \D{(\ker \gamma^{4})} \mapsto 0$, we then have the following commutative diagram.
\[
\begin{tikzpicture}
\matrix (m)[matrix of math nodes, row sep=3em, column sep=5em, text height=2ex, text depth=0.25ex]
{ \D{(\ker j^{4})} \oplus  \D{(\ker \gamma^{4})} \otimes t  &  \D{(\ker \gamma^{4})} \otimes dt  \\
  \EIbidg{2}{0}{0,2}{X} = \coker{\gamma^{2}}  &  0 \\};

\path[->]
(m-1-1) edge node[auto] {$\diff{0}{0,2}$} (m-1-2);
\path[->]
(m-2-1) edge node[auto,swap] {} (m-2-2);
\path[->]
(m-1-1) edge node[auto,swap] {$\phip{0}{0,2}$} (m-2-1);
\path[->]
(m-1-2) edge node[auto] {} (m-2-2);
\end{tikzpicture}
\]
If $s=1$, the isomorphism $\ker j^{1} \cong \coker{\gamma^{1}}$ defines $\phip{0}{0,1}$. 

We then have a quasi-isomorphism of algebras
\[
(\EIbidg{1}{0}{r,s}{X},\difb{0}{r,s}) \overset{\psip{0}{r,s}}{\longleftarrow} (\FI{0}{r,s}{X},\diff{0}{r,s}) \overset{\phip{0}{r,s}}{\longrightarrow} (\EIbidg{2}{0}{r,s}{X},0).
\]

We now check the commutativity of the following diagram
\[
\begin{tikzpicture}[injection/.style={right hook->,fill=white, inner sep=2pt}]
\matrix (m)[matrix of math nodes, row sep=3em, column sep=5em, text height=2ex, text depth=0.25ex]
{ (\EIbidg{1}{2}{r,s}{X},\difb{2}{r,s})  &  (\FI{2}{r,s}{X},\diff{2}{r,s})  & (\EIbidg{2}{2}{r,s}{X},0) \\
  (\EIbidg{1}{0}{r,s}{X},\difb{0}{r,s})  &  (\FI{0}{r,s}{X},\diff{0}{r,s})  & (\EIbidg{2}{0}{r,s}{X},0)\\};

\path[->]
(m-1-2) edge node[auto,swap] {$\psip{2}{r,s}$} (m-1-1);
\path[->]
(m-1-2) edge node[auto] {$\phip{2}{r,s}$} (m-1-3);

\path[->]
(m-2-2) edge node[auto] {$\psip{0}{r,s}$} (m-2-1);
\path[->]
(m-2-2) edge node[auto,swap] {$\phip{0}{r,s}$} (m-2-3);

\path[->]
(m-1-1) edge node[auto,swap] {$EI_{1}(\pomap{2}{0})$} (m-2-1);
\path[injection]
(m-1-2) edge node[auto] {$\pomap{2}{0}$} (m-2-2);
\path[->]
(m-1-3) edge node[auto] {$EI_{2}(\pomap{2}{0})$} (m-2-3);
\end{tikzpicture}
\]
The only differences between $\EIbidg{i}{2}{r,s}{X}$ and $\EIbidg{i}{0}{r,s}{X}$, $i=1,2$, arise for $s=1,2$. We then only check these cases. 
For $s=1$, there is nothing to check and everything commutes. For $s=2$, the only thing to check is the commutativity of the square
\[
\hspace{-1em}
\begin{tikzpicture}[injection/.style={right hook->,fill=white, inner sep=2pt}]
\matrix (m)[matrix of math nodes, row sep=3em, column sep=5em, text height=2ex, text depth=0.25ex]
{ \mathcal{I}^{2}_{1} \oplus H^{0}(D)\otimes \mathbf{Q}[t]dt  &  \ker j^{2}  & \ker j^{2} \\
  \mathcal{I}^{2}_{0} \oplus H^{0}(D)\otimes \mathbf{Q}[t]dt  &  \D{(\ker j^{4})} \oplus  \D{(\ker \gamma^{4})} \otimes t   & \coker{\gamma^{2}}\\};

\path[->]
(m-1-2) edge node[auto,swap] {$\psip{2}{0,2}$} (m-1-1);
\path[->]
(m-1-2) edge node[auto] {$\phip{2}{0,2}$} (m-1-3);

\path[->]
(m-2-2) edge node[auto] {$\psip{0}{0,2}$} (m-2-1);
\path[->]
(m-2-2) edge node[auto,swap] {$\phip{0}{0,2}$} (m-2-3);

\path[injection]
(m-1-1) edge node[auto,swap] {$EI_{1}(\pomap{2}{0})$} (m-2-1);
\path[injection]
(m-1-2) edge node[auto] {$\pomap{2}{0}$} (m-2-2);
\path[injection]
(m-1-3) edge node[auto] {$EI_{2}(\pomap{2}{0})$} (m-2-3);
\end{tikzpicture}
\]
Which is clear by the previous computations.

\subsubsection{The infinite perversity}
We finish with the perversity $\overline{\infty}$. We define the bigraded differential algebra $(\FI{\infty}{r,s}{X},\diff{\infty}{r,s})$ as the sub-algebra of $(\EI{1}{\infty}{X},\difb{\infty}{r,s})$ given by
\[
\begin{array}{| c || c c c |} 
\hline
s = 6          & 0        & H^{6}(\widetilde{X}) &  0  \\
\hline
s = 5          & 0        & H^{5}(\widetilde{X}) &  0  \\
\hline
s = 4          & 0        &  \ker j^{4}          &  0  \\
\hline
s = 3          & 0        &  \ker j^{3}          &  0  \\
\hline
s = 2          & 0        &  \ker j^{2}          & \D{(\ker \gamma^{4})} \otimes dt  \\
\hline
s=1            & 0        &  \ker j^{1}          & 0  \\
\hline
\hline
s = 0          & 0        & H^{0}(\widetilde{X}) & 0 \\
\hline
\hline
\FI{\infty}{r,s}{X} & r=-1     & r=0                  & r=1 \\
\hline 
\end{array}  
\]
There is no non trivial differentials. The algebra structure is as always concentrated in $r=0$. The map $\pomap{\infty}{4}$ is the canonical inclusion and is compatible the algebra structure.

The maps $\psip{\infty}{\ast,\ast}$ and $\phip{\infty}{\ast,\ast}$ are clear from the previous computations for the top perversity. 

We then have a quasi-isomorphism of algebras
\[
(\EIbidg{1}{\infty}{r,s}{X},\difb{\infty}{r,s}) \overset{\psip{\infty}{r,s}}{\longleftarrow} (\FI{\infty}{r,s}{X},\diff{\infty}{r,s}) \overset{\phip{\infty}{r,s}}{\longrightarrow} (\EIbidg{2}{\infty}{r,s}{X},0).
\]

We then define the coperverse cdga $\FI{\bullet}{\ast,\ast}{X}$ to be
\[
\FI{\bullet}{\ast,\ast}{X} = 
\begin{cases}
\FI{\infty}{\ast,\ast}{X} & \overline{p} = \infty, \\
\FI{4}{\ast,\ast}{X} & \overline{p} \in \lbrace \overline{3} , \overline{4} \rbrace, \\
\FI{2}{\ast,\ast}{X} & \overline{p} = \overline{2}, \\
\FI{0}{\ast,\ast}{X} & \overline{p} \in \lbrace \overline{0} , \overline{1} \rbrace. \\
\end{cases}
\]

We then have a quasi-isomorphism of coperverse cdga's.
\[
(\EIbidg{1}{\bullet}{\ast,\ast}{X},\difb{\bullet}{\ast,\ast}) \overset{\psip{\bullet}{\ast,\ast}}{\longleftarrow} (\FI{\bullet}{\ast,\ast}{X},\diff{\bullet}{\ast,\ast}) \overset{\phip{\bullet}{\ast,\ast}}{\longrightarrow} (\EIbidg{2}{\bullet}{\ast,\ast}{X},0).
\]

Then $\I{\bullet}{X}$ is formal.

\section{Examples and Applications}
\label{section:Eg}

We use the following conventions in the rest of this section :
\begin{itemize}
\item When needed, we will denote by $\lbrace 1_{i}, E_{i} \rbrace$ a basis of $H^{\ast}(\mathbf{C}P^{1}_{(i)})$, we complete it into a basis $\lbrace 1_{i}, E_{i}, \mathcal{E}_{i}, \Lambda_{i} \rbrace$ of $H^{\ast}(\mathbf{C}P^{1}_{(i)}\times \mathbf{C}P^{1}_{(i)})$ with $|E_{i}|=|\mathcal{E}_{i}|=2$, $|\Lambda_{i}|=4$ and where $\mathcal{E}_{i} E_{i} = \Lambda_{i}$.
\item even if we do not take into account the loops in the first cohomology group (see subsection \ref{subsubsec:coker}), we mark them in red 
\end{itemize}

\subsection{Projective cone over a K3 surface}

\begin{defi}
A K3 surface $S$ is a simply connected compact smooth complex surface such that its canonical bundle $K_{S}$ is trivial.
\end{defi}

Denote by $S$ a K3 surface, for example a nonsingular degree 4 hypersurface in $\mathbf{C}P^{3}$, such as the Fermat quartic 
\[
S = \lbrace  [z_{0}:z_{1}:z_{2}:z_{3}] \in \mathbf{C}P^{3} \, : \, z_{0}^{4}+z_{1}^{4}+z_{2}^{4}+z_{3}^{4}=0 \rbrace.
\]
In fact every K3 surface over $\mathbf{C}$ is diffeomorphic to this example, see \cite{1964}. The Hodge diamond of a K3 surface is completely determined and is given by the following.
\[
\begin{tikzpicture}
\matrix (m)[matrix of math nodes, row sep=0.5em, column sep=1em, text height=2ex, text depth=0.25ex]
{       &         & h^{2,2}  &         &   \\
        & h^{2,1} &          & h^{1,2} &   \\
h^{2,0} &         & h^{1,1}  &         &  h^{0,2} \quad =\\
        & h^{1,0} &          & h^{0,1} &   \\
        &         & h^{0,0}  &         &   \\};
\end{tikzpicture}
\begin{tikzpicture}
\matrix (m)[matrix of math nodes, row sep=0.5em, column sep=1.5em, text height=2ex, text depth=0.25ex]
{    &   & 1  &   &   \\
     & 0 &    & 0 &   \\
   1 &   & 20 &   &  1 \\
     & 0 &    & 0 &   \\
     &   &  1 &   &   \\};
\end{tikzpicture}
\]

Which means that we have the following cohomology.
\[
\begin{array}{| c || c | c | c | c | c |}
\hline
s               & 0          & 1 & 2               & 3 & 4          \\
\hline
\hline
H^{s}(S)        & \mathbf{Q} & 0 & \mathbf{Q}^{22} & 0 & \mathbf{Q} \\
\hline 
\end{array} 
\]

Denote by $\mathbb{P}_{\mathbf{C}}S \subset \mathbf{C}P^{4}$ the projective cone over the K3 surface. This is a simply connected hypersurface of complex dimension 3 with only one isolated singularity which is the cone point and defined by the same equation but in $\mathbf{C}P^{4}$
\[
\mathbb{P}_{\mathbf{C}}S = \lbrace  [z_{0}:z_{1}:z_{2}:z_{3}:z_{4}] \in \mathbf{C}P^{4} \, : \, z_{0}^{4}+z_{1}^{4}+z_{2}^{4}+z_{3}^{4}=0 \rbrace.
\]
The cohomology of $\mathbb{P}_{\mathbf{C}}S$ is given by (see \cite[p.169]{Dimca1992})
\[
H^{k}(\mathbb{P}_{\mathbf{C}}S) = H^{k-2}(S) \, \forall \, k \geq 2.
\]
By Hironaka's Theorem on resolution of singularities there exists a cartesian diagram
\[
\begin{tikzpicture}[injection/.style={right hook->,fill=white, inner sep=2pt}]
\matrix (m)[matrix of math nodes, row sep=3em, column sep=5em, text height=2ex, text depth=0.25ex]
{ S     &  \widetilde{P}  \\
  \ast  &  \mathbb{P}_{\mathbf{C}}S \\};

\path[->]
(m-1-1) edge node[auto] {} (m-1-2);
\path[injection]
(m-2-1) edge node[auto,swap] {} (m-2-2);
\path[->]
(m-1-1) edge node[auto,swap] {} (m-2-1);
\path[->]
(m-1-2) edge node[auto] {$f$} (m-2-2);
\end{tikzpicture}
\]
where the exceptional divisor is the K3 surface $S$ and $\widetilde{P}$ is a smooth projective variety of complex dimension 3. We then have the following Mayer-Vietoris sequence 
\[
\cdots \rightarrow H^{k}(\mathbb{P}_{\mathbf{C}}S) \rightarrow H^{k}(\widetilde{P}) \oplus H^{k}(\ast) \rightarrow H^{k}(S) \rightarrow \cdots
\]
which gives the following cohomology for $\widetilde{P}$. 
\[
\begin{array}{| c || c | c | c | c | c | c | c |}
\hline
s                    & 0          & 1 & 2                                 & 3 & 4                                  & 5 & 6 \\
\hline
\hline
H^{s}(\widetilde{P}) & \mathbf{Q} & 0 & \mathbf{Q} \oplus \mathbf{Q}^{22} & 0 & \mathbf{Q} \oplus \mathbf{Q}^{22}  & 0 & \mathbf{Q} \\
\hline 
\end{array} 
\]

We compute the intersection space for the perversities $\lbrace \overline{0}, \overline{1}, \overline{2}, \overline{3}, \overline{4} \rbrace$. 

First of all The intersection space for the zero perversity is by definition the regular part, which is computed by the following spectral sequence
\[
\begin{array}{| c || c | c || c | c |}
\hline
\multicolumn{3}{|c||}{E_{1}^{r,s}((\mathbb{P}_{\mathbf{C}}S)_{reg})}   &  \multicolumn{2}{c|}{E_{2}^{r,s}((\mathbb{P}_{\mathbf{C}}S)_{reg})} \\
\hline
s=6            & \mathbf{Q}       & \mathbf{Q}                        &   0                 &0\\ 
\hline
s=5            & 0                & 0                                 &   0                 &0\\ 
\hline
s=4            & \mathbf{Q}^{22}  & \mathbf{Q} \oplus \mathbf{Q}^{22} &   0                 & \mathbf{Q}\\
\hline
s=3            & 0                & 0                                 &   0                 &0\\
\hline
s=2            & \mathbf{Q}       & \mathbf{Q} \oplus \mathbf{Q}^{22} &   0                 &  \mathbf{Q}^{22}\\
\hline
s=1            & 0                & 0                                 &   0                 &0\\
\hline
s=0            & 0                & \mathbf{Q}                        &   0                 &\mathbf{Q}\\
\hline
\hline
               & r=-1             & r=0                               &  \ker \gamma^{s}    &  \coker{\gamma^{s}}                          \\
\hline 
\end{array}  
\]

Now we need the cohomology of the link, which is given by the spectral sequence defined by $j^{s}_{\sharp} \col H^{s-2}(D) \rightarrow H^{s}(D)$, as in the section \ref{subsec:smooth_div}.
\[
\begin{array}{| c || c | c || c | c |}
\hline
\multicolumn{3}{|c||}{E_{1}^{r,s}(L)}   &  \multicolumn{2}{c|}{E_{2}^{r,s}(L)} \\
\hline
s=6            & \mathbf{Q}       & 0                                 &   \mathbf{Q}        &0\\ 
\hline
s=5            & 0                & 0                                 &   0                 &0\\ 
\hline
s=4            & \mathbf{Q}^{22}  & \mathbf{Q}                        &   \mathbf{Q}^{21}   & 0\\
\hline
s=3            & 0                & 0                                 &   0                 &0\\
\hline
s=2            & \mathbf{Q}       & \mathbf{Q}^{22}                   &   0                 &  \mathbf{Q}^{21}\\
\hline
s=1            & 0                & 0                                 &   0                 &0\\
\hline
s=0            & 0                & \mathbf{Q}                        &   0                 &\mathbf{Q}\\
\hline
\hline
               & r=-1             & r=0                               &  \ker j_{\sharp}^{s}    &  \coker{j_{\sharp}^{s}}                          \\
\hline 
\end{array}  
\]

We then have 
\[
\begin{cases}
H^{0}(L) & = H^{5}(L) = \mathbf{Q}, \\
H^{1}(L) & = H^{4}(L) = 0, \\
H^{2}(L) & = H^{3}(L) = \mathbf{Q}^{21}. \\
\end{cases}
\]

By the $E_{2}$ term of the previous spectral sequence we see that the only sections of $j^{s}_{\sharp}$ for which the image won't be zero correspond to the perversities $\overline{1}$ and $\overline{3}$. Each times the image of the section is equal to $\mathbf{Q}$, we then have the two following map
\[
\gamma^{2}_{|\coim{1}} \col \coim{1}=\mathbf{Q} \longrightarrow H^{2}(\widetilde{P}) = \mathbf{Q} \oplus \mathbf{Q}^{22},
\]
\[
\gamma^{4}_{|\coim{3}} \col \coim{3}=\mathbf{Q} \longrightarrow H^{4}(\widetilde{P}) = \mathbf{Q} \oplus \mathbf{Q}^{22}.
\]
and $\coker{\gamma^{2}_{|\coim{1}}} \cong \coker{\gamma^{4}_{|\coim{3}}} \cong \mathbf{Q}^{22}$.

The last map we need to know is $j^{s} \col H^{s}(\widetilde{P}) \rightarrow H^{s}(S)$, the map induced by the inclusion $S \hookrightarrow \widetilde{P}$.
\[
\begin{array}{| c || c | c |}
\hline
s=6            & \mathbf{Q}       & 0   \\ 
\hline
s=5            & 0                & 0     \\ 
\hline
s=4            & \mathbf{Q}^{22}  & 0    \\
\hline
s=3            & 0                & 0   \\
\hline
s=2            & \mathbf{Q}       & 0    \\
\hline
s=1            & 0                & 0    \\
\hline
s=0            & 0                & 0    \\
\hline
\hline
               & \ker j^{s}       & \coker{j^{s}}    \\
\hline 
\end{array}  
\]

We recall the $EI_{2}$ term of the spectral sequence of $\I{p}{X}$.
\[
\begin{array}{| c || c c c |} 
\hline
s > p+1        & \ker \gamma^{s} &  \coker{\gamma^{s}}             &  0  \\
\hline
s = p+1        & 0               &  \coker{\gamma^{s}_{|\coim{p}}} & 0  \\
\hline
1 \leq s < p+1 & 0               &  \ker j^{s}                     & \coker{j^{s}}  \\
\hline
\hline
s = 0          & 0               & H^{0}(\widetilde{P})            & 0 \\
\hline
\hline
\EIbidg{2}{p}{r,s}{X} & r=-1            & r=0                      & r=1 \\
\hline 
\end{array}  
\]

We then have the following results.

\[
\begin{array}{| c ||  c  |}
\hline
\multicolumn{2}{|c|}{\EIbidg{2}{1}{r,s}{\mathbb{P}_{\mathbf{C}}S}} \\ 
\hline
s \geq 5                      &  0                 \\
\hline
s = 4                         &  \mathbf{Q}       \\
\hline
s = 3                         &  0                 \\
\hline
s = 2                         &  \mathbf{Q}^{22}   \\
\hline
s = 1                         &  0                 \\
\hline
\hline
s = 0                         & H^{0}(\widetilde{P}) \\
\hline
\hline
                             & r=0                               \\
\hline 
\end{array}  
\quad
\begin{array}{| c || c |} 
\hline
\multicolumn{2}{|c|}{\EIbidg{2}{2}{r,s}{\mathbb{P}_{\mathbf{C}}S}} \\ 
\hline
s \geq 5                      &  0                \\
\hline
s = 4                         &  \mathbf{Q}       \\
\hline
s = 3                         &  0                 \\
\hline
s = 2                         &  \mathbf{Q}        \\
\hline
s = 1                         &  0                 \\
\hline
\hline
s = 0                         & H^{0}(\widetilde{P})          \\
\hline
\hline
                             & r=0                               \\
\hline 
\end{array}  
\quad
\begin{array}{| c ||  c|}
\hline
\multicolumn{2}{|c|}{\EIbidg{2}{3}{r,s}{\mathbb{P}_{\mathbf{C}}S}} \\ 
\hline
s \geq 5                      &  0                 \\
\hline
s = 4                         &  \mathbf{Q}^{22}       \\
\hline
s = 3                         &  0                 \\
\hline
s = 2                         &  \mathbf{Q}        \\
\hline
s = 1                         &  0                 \\
\hline
\hline
s = 0                         & H^{0}(\widetilde{P})           \\
\hline
\hline
                              & r=0                               \\
\hline 
\end{array}  
\quad 
\begin{array}{| c || c |} 
\hline
\multicolumn{2}{|c|}{\EIbidg{2}{4}{r,s}{\mathbb{P}_{\mathbf{C}}S}} \\ 
\hline
s \geq 5                      &  0                 \\
\hline
s = 4                         &  \mathbf{Q}^{22}  \\
\hline
s = 3                         &  0                 \\
\hline
s = 2                         &  \mathbf{Q}        \\
\hline
s = 1                         &  0                 \\
\hline
\hline
s = 0                         & H^{0}(\widetilde{P})           \\
\hline
\hline
                              & r=0                               \\
\hline 
\end{array}  
\]

Note that for complementary perversities, such as $\overline{1}$ and $\overline{3}$ or $\overline{0}$ and $\overline{4}$, and for $s \neq 0$ the $EI_{2}$ term gives back the generalized Poincaré duality between the various intersection spaces such as proved in \cite[theorem 2.12]{Banagl2010}. The middle perversity here is $\overline{2}$ and we also get back the self-duality of the space $\I{2}{\mathbb{P}_{\mathbf{C}}S}$.

For any perversity $\overline{p}$ the weight filtration is pure, so by the theorem \ref{thm:pure_is_formal} we get the following proposition.

\begin{propo}
Given any perversity $\overline{p}$, the intersection space $\I{p}{\mathbb{P}_{\mathbf{C}}S}$ is a formal topological space.
\end{propo}

\subsection{Kummer quartic surface}
\label{subsec:Kummer}

Let $K$ be a Kummer quartic surface, that is an irreducible surface of degree 4 in $\mathbf{C}P^{3}$ with 16 ordinary double points, which is the maximum for such surfaces. 

From the algebraic topologist point of view, a Kummer surface is constructed in the following way. Let's consider a 4-dimensional torus
\[
\mathbf{T} = S^{1} \times S^{1} \times S^{1} \times S^{1}
\]  
endowed with the complex involution $\tau \col z \mapsto \bar{z}$ action. This action has 16 fixed point and we define the Kummer surface to be the quotient complex surface
\[
K := \mathbf{T}/\tau.
\]
We have the following cohomology for $K$.
\[
\begin{array}{| c || c | c | c | c | c |}
\hline
s               & 0          & 1 & 2              & 3 & 4     \\
\hline
\hline
H^{s}(K)        & \mathbf{Q} & 0 & \mathbf{Q}^{6} & 0 & \mathbf{Q} \\
\hline 
\end{array} 
\]

The link of each singularity is then a projective space $\mathbf{R}P^{3}$. These singularities are quotients singularities so by \cite{Durfee1979} $K$ admits a resolution where the exceptional set consists of curves of genus zero and self-intersection $-2$. Which means we have the following resolution diagram
\[
\begin{tikzpicture}[injection/.style={right hook->,fill=white, inner sep=2pt}]
\matrix (m)[matrix of math nodes, row sep=3em, column sep=5em, text height=2ex, text depth=0.25ex]
{ \bigsqcup_{i=1}^{16} \mathbf{C}P^{1}_{(i)}     &  \widetilde{K}  \\
  \bigsqcup_{i=1}^{16} \ast_{(i)}                &  K \\};

\path[->]
(m-1-1) edge node[auto] {} (m-1-2);
\path[injection]
(m-2-1) edge node[auto,swap] {} (m-2-2);
\path[->]
(m-1-1) edge node[auto,swap] {} (m-2-1);
\path[->]
(m-1-2) edge node[auto] {$f$} (m-2-2);
\end{tikzpicture}
\]

The Mayer-Vietoris sequence gives the following cohomology for $\widetilde{K}$.
\[
\begin{array}{| c || c | c | c | c | c |}
\hline
s                    & 0          & 1 & 2                                                        & 3 & 4     \\
\hline
\hline
H^{s}(\widetilde{K}) & \mathbf{Q} & 0 & \mathbf{Q}^{6}\oplus \bigoplus_{i=1}^{16}\mathbf{Q}E_{i} & 0 & \mathbf{Q} \\
\hline 
\end{array} 
\]

We have the fairly easy following spectral sequence for the links.
\[
\begin{array}{| c || c | c || c | c |}
\hline
\multicolumn{3}{|c||}{E_{1}^{r,s}(L)}   &  \multicolumn{2}{c|}{E_{2}^{r,s}(L)} \\
\hline
s=4            & \bigoplus_{i=1}^{16}\mathbf{Q}E_{i}  & 0                                    &   \bigoplus_{i=1}^{16}\mathbf{Q}E_{i}   & 0\\
\hline
s=3            & 0                                    & 0                                    &   0                 &0\\
\hline
s=2            & \bigoplus_{i=1}^{16}\mathbf{Q}1_{i}  & \bigoplus_{i=1}^{16}\mathbf{Q}E_{i}  &   0                 &  0\\
\hline
s=1            & 0                                    & 0                                    &   0                 &0\\
\hline
s=0            & 0                                    & \bigoplus_{i=1}^{16}\mathbf{Q}1_{i}  &   0                 &\bigoplus_{i=1}^{16}\mathbf{Q}1_{i}\\
\hline
\hline
               & r=-1                                 & r=0                                  &  \ker j_{\sharp}^{s}    &  \coker{j_{\sharp}^{s}}                          \\
\hline 
\end{array}  
\]
The rational cohomology of link of each singularities is then a 3-sphere, which is the rationalization of $\mathbf{R}P^{3}$.

The only interesting perversity here is the middle perversity $\overline{1}$. We need a $\coim{1}$ for the computation, we have here
\[
\coim{1} = \bigoplus_{i=1}^{16}\mathbf{Q}1_{i}
\]
and $\gamma^{2}_{|\coim{1}} = \gamma^{2}$. 

The following spectral sequence computes the regular part and the second array is the restriction map $j^{s}$.
\[
\begin{array}{| c || c | c || c | c |}
\hline
\multicolumn{3}{|c||}{E_{1}^{r,s}(K_{reg}) \quad \gamma^{s} \col H^{s-2}(D) \longrightarrow H^{s}(\widetilde{K})} &  \multicolumn{2}{c|}{E_{2}^{r,s}(K_{reg})} \\
\hline
s=4            & \bigoplus_{i=1}^{16}\mathbf{Q}E_{i}  & \mathbf{Q}                           &   \bigoplus_{i=1}^{15}\mathbf{Q}E_{i}   & 0\\
\hline
s=3            & 0                                    & 0                                    &   0                 &0\\
\hline
s=2            & \bigoplus_{i=1}^{16}\mathbf{Q}1_{i}  &\mathbf{Q}^{6}\oplus \bigoplus_{i=1}^{16}\mathbf{Q}E_{i}  &   0               &  \mathbf{Q}^{6}\\
\hline
s=1            & 0                                    & 0                                    &   0                 &0\\
\hline
s=0            & 0                                    & \mathbf{Q}                           &   0                 &\mathbf{Q}\\
\hline
\hline
               & r=-1                                 & r=0                                  &  \ker \gamma^{s}    &  \coker{\gamma^{s}}                          \\
\hline 
\end{array}  
\]

\[
\begin{array}{| c || c | c || c | c |}
\hline
s=4            & \mathbf{Q}                           & 0                                    &   \mathbf{Q}        & 0\\
\hline
s=3            & 0                                    & 0                                    &   0                 &0\\
\hline
s=2            & \mathbf{Q}^{6} \oplus \bigoplus_{i=1}^{16}\mathbf{Q}E_{i} &\bigoplus_{i=1}^{16}\mathbf{Q}E_{i}  &   \mathbf{Q}^{6} &  0\\
\hline
s=1            & 0                                    & 0                                    &   0                 &0\\
\hline
s=0            & \mathbf{Q}                           & \bigoplus_{i=1}^{16} \mathbf{Q}1_{i} &   0                 &\bigoplus_{i=1}^{15} \mathbf{Q}1_{i}\\
\hline
\hline
               & H^{s}(\widetilde{K})                 & H^{s}(D)                             &  \ker j^{s}         &  \coker{j^{s}}                          \\
\hline 
\end{array}  
\]

The cohomology of the middle perversity intersection space of a Kummer surface is then given by the following array. Note that the cohomology obtained isn't pure.
\[
\begin{array}{| c || c c c || c |}
\hline
s = 4          & \bigoplus_{i=1}^{15} \mathbf{Q}E_{i} & 0                    & 0 & 0\\
\hline
s = 3          & 0                                    & 0                    & 0 & \bigoplus_{i=1}^{15} \mathbf{Q}E_{i}\\
\hline
s = 2          & 0                                    & \mathbf{Q}^{6}       & 0 & \mathbf{Q}^{6} \\
\hline
s = 1          & 0                                    &   0                  & 0 & \textcolor{red}{\bigoplus_{i=1}^{15} \mathbf{Q}1_{i}}  \\
\hline
\hline
s = 0          & 0                                    & H^{0}(\widetilde{K}) & \textcolor{red}{\bigoplus_{i=1}^{15} \mathbf{Q}1_{i}} &H^{0}(\widetilde{K})  \\
\hline 
\hline
\EIbidg{2}{1}{r,s}{K} & r=-1                          & r=0                  & r=1  &\HI{1}{s}{K}  \\
\hline 
\end{array}
\]

\subsection{The Calabi-Yau generic quintic 3-fold}
\label{subsec:CY_generic_quintic}

Let $Y \subset \mathbf{C}P^{4}$ the singular hypersurface given by the equation
\[
Y := \lbrace [z_{0}:z_{1}:z_{2}:z_{3}:z_{4}] \in \mathbf{C}P^{4} \, : \, z_{3}g(z_{0}, \dots, z_{4}) + z_{4}h(z_{0}, \dots, z_{4}) =0 \rbrace
\]
where $g$ and $h$ are generic homogeneous polynomials of degree 4. $Y$ is the Calabi-Yau generic quintic 3-fold containing the plane 
\[
\pi := \lbrace z_{3}=z_{4}=0 \rbrace \cong \mathbf{C}P^{2}.
\]

The singular locus
\[
\Sigma := \lbrace [x] \in \mathbf{C}P^{4} \, : \, z_{3}=z_{4}=g(z)=h(z)=0  \rbrace \subset \mathbf{C}P^{2}
\]
is given by 16 ordinary double points. That is the link of each singularity $\sigma \in \Sigma$ is topologically equal to $L_{\sigma} = S^{2} \times S^{3}$.

We have the following cohomology for $Y$.
\[
\begin{array}{| c || c | c | c | c | c | c | c |}
\hline
s               & 0          & 1 & 2          & 3                & 4               & 5 & 6 \\
\hline
\hline
H^{s}(Y)        & \mathbf{Q} & 0 & \mathbf{Q} & \mathbf{Q}^{189} & \mathbf{Q}^{2}  & 0 & \mathbf{Q} \\
\hline 
\end{array} 
\]

We consider the following commutative diagram of resolutions
\[
\begin{tikzpicture}
\matrix (m)[matrix of math nodes, row sep=3em, column sep=2.5em, text height=1.5ex, text depth=0.25ex]
{\bigsqcup_{i=1}^{16} \mathbf{C}P^{1}_{(i)} \times \mathbf{C}P^{1}_{(i)}  & \bigsqcup_{i=1}^{16} \mathbf{C}P^{1}_{(i)}  & \bigsqcup_{i=1}^{16} \ast_{(i)}  \\
  \overline{Y}                                                            & \widetilde{Y}                               & Y \\};

\path[->]
(m-1-1) edge node[right=2.5em] {Blow up} (m-2-1);
\path[->]
(m-1-2) edge node[right=1.25em] {small res.} (m-2-2);
\path[->]
(m-1-3) edge node[auto] {} (m-2-3);

\path[->]
(m-1-1) edge node[auto] {} (m-1-2);
\path[->]
(m-1-2) edge node[auto] {} (m-1-3);

\path[->]
(m-2-1) edge node[auto,swap] {$\mathcal{B}\ell$} (m-2-2);
\path[->]
(m-2-2) edge node[auto,swap] {$f$} (m-2-3);
\end{tikzpicture}
\]

The first square is a simultaneous small resolution of the 16 singularities obtained by blowing up $\mathbf{C}P^{4}$ along the plane $\pi \cong \mathbf{C}P^{2}$. The exceptionnal divisor of this blow-up is a $\mathbf{C}P^{1}$-bundle over $\pi \cong \mathbf{C}P^{2}$.

For the second square $\mathcal{B}\ell$ is a blow-up along the $\mathbf{C}P^{1}_{(i)}$'s. 

Denote by $\Psi$ the generator of $H^{2}(Y)$.

By using twice the Mayer-Vietoris long exact sequence, we get the following cohomology for $\overline{Y}$.
\[
\begin{cases}
H^{0}(\overline{Y}) = H^{6}(\overline{Y}) = \mathbf{Q}, & \\
H^{1}(\overline{Y}) = H^{5}(\overline{Y}) = 0, &\\
H^{2}(\overline{Y}) = \mathbf{Q}\Psi \oplus \mathbf{Q}E_{1} \oplus \bigoplus_{i=1}^{16} \mathbf{Q}\D{\Lambda_{i}}, & \\
H^{4}(\overline{Y}) = \mathbf{Q}\D{\Psi} \oplus \mathbf{Q}\D{E_{1}} \oplus \bigoplus_{i=1}^{16} \mathbf{Q}\Lambda_{i}, & \\
H^{3}(\overline{Y}) = \mathbf{Q}^{174}.
\end{cases}
\]

The cohomology of the links of the singularities is given by the spectral sequence
\[
\begin{array}{| c || c | c || c | c |}
\hline
\multicolumn{3}{|c||}{E_{1}^{r,s}(L)} & \multicolumn{2}{c|}{E_{2}^{r,s}(L)} \\ 
\hline
s=6            & \bigoplus_{i=1}^{16} \mathbf{Q}\Lambda_{i}                              & 0                                                                       &\bigoplus_{i=1}^{16} \mathbf{Q}\Lambda_{i}              & 0   \\ 
\hline
s=5            & 0                                                                       & 0                                                                       & 0 & 0\\ 
\hline 
s=4            & \bigoplus_{i=1}^{16} (\mathbf{Q}E_{i} \oplus \mathbf{Q}\mathcal{E}_{i}) &  \bigoplus_{i=1}^{16} \mathbf{Q}\Lambda_{i}                             &\bigoplus_{i=1}^{16} \mathbf{Q}\mathcal{E}_{i} & 0\\
\hline
s=3            & 0                                                                       & 0                                                                       & 0 & 0\\
\hline
s=2            & \bigoplus_{i=1}^{16} \mathbf{Q}1_{i}                                    & \bigoplus_{i=1}^{16} (\mathbf{Q}E_{i} \oplus \mathbf{Q}\mathcal{E}_{i}) & 0 & \bigoplus_{i=1}^{16} \mathbf{Q}E_{i} \\
\hline
s=1            & 0                                                                       & 0                                                                       & 0 & 0 \\
\hline
s=0            & 0                                                                       & \bigoplus_{i=1}^{16} \mathbf{Q}1_{i}                                    & 0 & \bigoplus_{i=1}^{16} \mathbf{Q}1_{i}   \\
\hline
\hline
               & r=-1                                                                    & r=0                                                                     &\ker j^{s}_{\sharp} & \coker{j^{s}_{\sharp}} \\
\hline 
\end{array}  
\]

We here follow the section \ref{sec:formality_3folds} and do the computations for the top, middle and zero perversity.
The spectral sequence of the regular part is given by 
\[
\hspace{-0.5em}
\begin{array}{| c || c | c || c | c |}
\hline
\multicolumn{3}{|c||}{E_{1}^{r,s}(Y_{reg}) \quad \gamma^{s} \col H^{s-2}(D) \longrightarrow H^{s}(\overline{Y})} & \multicolumn{2}{c|}{E_{2}^{r,s}(Y_{reg})} \\
\hline 
s=6            & \bigoplus_{i=1}^{16} \mathbf{Q}\Lambda_{i}                               & \mathbf{Q}  & \bigoplus_{i=1}^{15} \mathbf{Q}\Lambda_{i}                      & 0   \\ 
\hline
s=5            & 0                                                                        & 0           & 0                                                               & 0     \\    
\hline
s=4            &  \bigoplus_{i=1}^{16} (\mathbf{Q}E_{i} \oplus \mathbf{Q}\mathcal{E}_{i}) & \mathbf{Q}\D{\Psi} \oplus \mathbf{Q}\D{E_{1}} \oplus \bigoplus_{i=1}^{16} \mathbf{Q}\Lambda_{i}  & \bigoplus_{i=1}^{15} \mathbf{Q}\mathcal{E}_{i}                  & \mathbf{Q}\D{\Psi}      \\
\hline
s=3            & 0                                                                        & \mathbf{Q}^{174} & 0                                                               & \mathbf{Q}^{174}    \\
\hline  
s=2            & \bigoplus_{i=1}^{16} \mathbf{Q}1_{i}                                     & \mathbf{Q}\Psi \oplus \mathbf{Q}E_{1} \oplus \bigoplus_{i=1}^{16} \mathbf{Q}\D{\Lambda_{i}} & 0                                                               & \mathbf{Q}\Psi \oplus \mathbf{Q}E_{1} \\
\hline
s=1            & 0                                                                        & 0   & 0                                                               & 0   \\
\hline
s=0            & 0                                                                        & \mathbf{Q}   & 0                                                               & \mathbf{Q}    \\
\hline
\hline
               & r=-1                                                                     & r=0       & \ker \gamma^{s}                                        &\coker{\gamma^{s}}             \\
\hline 
\end{array}  
\]

Finally we also need the restriction morphism $j^{s}$.
\[
\hspace{-0.9em}
\begin{array}{| c || c | c || c | c |}
\hline
s=6            & \mathbf{Q}                                                                                          & 0  & \mathbf{Q}                                                                               & 0  \\ 
\hline
s=5            & 0                                                                                                   & 0  & 0                                                                                        & 0     \\ 
\hline
s=4            & \mathbf{Q}\D{\Psi} \oplus \mathbf{Q}\D{E_{1}} \oplus \bigoplus_{i=1}^{16} \mathbf{Q}\Lambda_{i}     & \bigoplus_{i=1}^{16} \mathbf{Q}\Lambda_{i} & \mathbf{Q}\D{\Psi} \oplus \mathbf{Q}\D{E_{1}}                                            & 0   \\
\hline
s=3            & \mathbf{Q}^{174}                                                                                    & 0   & \mathbf{Q}^{174}                                                                         & 0  \\
\hline
s=2            & \mathbf{Q}\Psi \oplus \mathbf{Q}E_{1} \oplus \bigoplus_{i=1}^{16} \mathbf{Q}\D{\Lambda_{i}}         & \bigoplus_{i=1}^{16} (\mathbf{Q}E_{i} \oplus \mathbf{Q}\mathcal{E}_{i}) & \mathbf{Q}\Psi                                                                           & \bigoplus_{i=1}^{15}\mathbf{Q}E_{i}   \\
\hline
s=1            & 0                                                                                                   & 0    & 0                                                                                        & 0   \\
\hline
s=0            & \mathbf{Q}                                                                                          & \bigoplus_{i=1}^{16} \mathbf{Q}1_{i} & 0                                                                                        & \bigoplus_{i=1}^{15} \mathbf{Q}1_{i}   \\
\hline
\hline
               & H^{s}(\overline{Y})                                                                                 & H^{s}(D)    & \ker j^{s}                                                                               & \coker{j^{s}} \\
\hline 
\end{array}  
\]

We then get the following tables for the perversities $\overline{0}, \overline{2}, \overline{4}$. Note here that the generalized Poincaré duality is only partial as we explained in the subsection \ref{subsubsec:coker} since we do not take into accounts the loops of $\coker{j^{0}}$ (marked in red in the arrays).

\[
\begin{array}{| c || c c c || c |}
\hline
s = 6          & \bigoplus_{i=1}^{15} \mathbf{Q}\Lambda_{i}                         &  0                                    &   0  & 0 \\ 
\hline
s = 5          & 0                                                                  &  0                                    &   0  & \bigoplus_{i=1}^{15} \mathbf{Q}\Lambda_{i}  \\
\hline
s = 4          & \bigoplus_{i=1}^{15}\mathbf{Q}\mathcal{E}_{i}                      & \mathbf{Q}\D{\Psi}                    &   0  & \mathbf{Q}\D{\Psi} \\
\hline
s = 3          & 0                                                                  & \mathbf{Q}^{174}                      &   0  &  \mathbf{Q}^{189}   \\
\hline 
s = 2          & 0                                                                  & \mathbf{Q}\Psi \oplus \mathbf{Q}E_{1} &   0  & \mathbf{Q}\Psi \oplus \mathbf{Q}E_{1}  \\
\hline
s = 1          & 0                                                                  &     0                                 &   0  & \textcolor{red}{\bigoplus_{i=1}^{15} \mathbf{Q}1_{i}} \\
\hline
\hline
s = 0          & 0                                                                  & H^{0}(\overline{Y})                   & \textcolor{red}{\bigoplus_{i=1}^{15} \mathbf{Q}1_{i}} & H^{0}(\overline{Y}) \\
\hline
\hline
\EIbidg{2}{0}{r,s}{Y} & r=-1                                                        & r=0                                   & r=1 &  \HI{0}{s}{Y} \\
\hline 
\end{array}
\]

Note here the partial duality for the values $s=2,3,4$ for the perversities $\overline{0}$ and $\overline{4}$.

\[
\begin{array}{| c || c c c || c |}
\hline
s = 6          & \bigoplus_{i=1}^{15} \mathbf{Q}\Lambda_{i}                         &  0                                            &   0  & 0 \\ 
\hline
s = 5          & 0                                                                  &  0                                            &   0  & \bigoplus_{i=1}^{15} \mathbf{Q}\Lambda_{i}  \\
\hline
s = 4          & 0                                                                  & \mathbf{Q}\D{\Psi} \oplus \mathbf{Q}\D{E_{1}} &   0  & \mathbf{Q}\D{\Psi} \oplus \mathbf{Q}\D{E_{1}}\\
\hline
s = 3          & 0                                                                  & \mathbf{Q}^{174}                              &   0  & \mathbf{Q}^{189}\\
\hline 
s = 2          & 0                                                                  & \mathbf{Q}\Psi                                & \bigoplus_{i=1}^{15}\mathbf{Q}E_{i} & \mathbf{Q}\Psi    \\
\hline
s = 1          & 0                                                                  &   0                                           &   0 & \textcolor{red}{\bigoplus_{i=1}^{15} \mathbf{Q}1_{i}} \\
\hline
\hline
s = 0          & 0                                                                  & H^{0}(\overline{Y})                           & \textcolor{red}{\bigoplus_{i=1}^{15} \mathbf{Q}1_{i}} & H^{0}(\overline{Y})\\
\hline
\hline
\EIbidg{2}{4}{r,s}{Y} & r=-1                                                        & r=0                                           & r=1 &  \HI{4}{s}{Y} \\
\hline 
\end{array}
\]

For the perversity $\overline{2}$ we retrieve the values of the smooth deformation as in \cite{Banagl2012}, unless for $s=1$.

\[
\begin{array}{| c || c c c || c |}
\hline
s = 6          & \bigoplus_{i=1}^{15} \mathbf{Q}\Lambda_{i}    &  0                  &   0                                                   &0 \\ 
\hline
s = 5          & 0                                             &  0                  &   0                                                   &\bigoplus_{i=1}^{15} \mathbf{Q}\Lambda_{i} \\
\hline
s = 4          & \bigoplus_{i=1}^{15}\mathbf{Q}\mathcal{E}_{i} & \mathbf{Q}\D{\Psi}  &   0                                                   &\mathbf{Q}\D{\Psi}\\
\hline
s = 3          & 0                                             & \mathbf{Q}^{174}    &   0                                                   &\mathbf{Q}^{204}\\
\hline
s = 2          & 0                                             & \mathbf{Q}\Psi      & \bigoplus_{i=1}^{15}\mathbf{Q}E_{i}                   &\mathbf{Q}\Psi   \\
\hline
s = 1          & 0                                             &   0                 &   0                                                   &\textcolor{red}{\bigoplus_{i=1}^{15} \mathbf{Q}1_{i}}\\
\hline
\hline
s = 0          & 0                                             & H^{0}(\overline{Y}) & \textcolor{red}{\bigoplus_{i=1}^{15} \mathbf{Q}1_{i}} & H^{0}(\overline{Y}) \\
\hline 
\hline
\EIbidg{2}{2}{r,s}{Y} & r=-1                                   & r=0                 & r=1                                                   & \HI{2}{s}{Y}\\
\hline 
\end{array}
\]

\subsection{The Quintic}
\label{subsec:CY_quintic}

Let $\psi$ be a complex number and consider the variety
\[
X_{\psi} := \left\lbrace [z_{0}:z_{1}:z_{2}:z_{3}:z_{4}] \in \mathbf{C}P^{4} \, : \, z_{0}^{5} + z_{1}^{5} + z_{2}^{5} + z_{3}^{5} + z_{4}^{5} -5\psi z_{0}z_{1}z_{2}z_{3}z_{4}=0 \right\rbrace,
\]
which is Calabi-Yau. It is smooth for small $\psi \neq 1$ and becomes singular when $\psi =1$, denote by $X$ the singular degeneration $X_{\psi=1}$. 

The singular locus $\Sigma$ of $X$ is here composed of 125 ordinary double points. That is the link of each singularity $\sigma \in \Sigma$ is topologically equal to $L_{\sigma} = S^{2} \times S^{3}$, just like before.

We get the following cohomology for $X$.
\[
\begin{array}{| c || c | c | c | c | c | c | c |}
\hline
s               & 0          & 1 & 2          & 3                & 4               & 5 & 6 \\
\hline
\hline
H^{s}(X)        & \mathbf{Q} & 0 & \mathbf{Q} & \mathbf{Q}^{103} & \mathbf{Q}^{25} & 0 & \mathbf{Q} \\
\hline 
\end{array} 
\]

Using the same method of resolution that before
\[
\begin{tikzpicture}
\matrix (m)[matrix of math nodes, row sep=3em, column sep=2.5em, text height=1.5ex, text depth=0.25ex]
{\bigsqcup_{i=1}^{125} \mathbf{C}P^{1}_{(i)} \times \mathbf{C}P^{1}_{(i)}  & \bigsqcup_{i=1}^{125} \mathbf{C}P^{1}_{(i)}  & \bigsqcup_{i=1}^{125} \ast_{(i)}  \\
 \overline{X}                                                              & \widetilde{X}                                & X \\};

\path[->]
(m-1-1) edge node[right=2.5em] {Blow up} (m-2-1);
\path[->]
(m-1-2) edge node[right=1.25em] {small res.} (m-2-2);
\path[->]
(m-1-3) edge node[auto] {} (m-2-3);

\path[->]
(m-1-1) edge node[auto] {} (m-1-2);
\path[->]
(m-1-2) edge node[auto] {} (m-1-3);

\path[->]
(m-2-1) edge node[auto,swap] {$\mathcal{B}\ell$} (m-2-2);
\path[->]
(m-2-2) edge node[auto,swap] {$f$} (m-2-3);
\end{tikzpicture}
\]

With the Mayer-Vietoris long exact sequence, we get the following cohomology for $\overline{X}$, we still denote by $\Psi$ the generator of $H^{2}(X)$.
\[
\begin{array}{| c || c |}
\hline  
6      & \mathbf{Q} \\
\hline 
5      & 0 \\
\hline 
4      & \mathbf{Q}\D{\Psi} \oplus \bigoplus_{i=1}^{24}\D{E_{i}} \oplus \bigoplus_{i=1}^{125}\mathbf{Q}\Lambda_{i} \\
\hline  
3      & \mathbf{Q}^{2} \\
\hline 
2      & \mathbf{Q}\Psi \oplus \bigoplus_{i=1}^{24}\mathbf{Q}E_{i} \oplus \bigoplus_{i=1}^{125}\mathbf{Q}\D{\Lambda_{i}} \\
\hline 
1      & 0 \\
\hline 
0      & \mathbf{Q} \\
\hline 
\hline 
s      & H^{s}(\overline{X}) \\
\hline 
\end{array}
\]

The spectral sequences of the regular part if given by
\[
\hspace{-2.7em}
\begin{array}{| c || c | c || c | c |}
\hline
\multicolumn{3}{|c||}{E_{1}^{r,s}(X_{reg}) \quad \gamma^{s} \col H^{s-2}(D) \longrightarrow H^{s}(\overline{X})} & \multicolumn{2}{c|}{E_{2}^{r,s}(X_{reg})} \\
\hline 
s=6            & \bigoplus_{i=1}^{125} \mathbf{Q}\Lambda_{i}                               & \mathbf{Q}          &  \bigoplus_{i=1}^{124} \mathbf{Q}\Lambda_{i}                      & 0  \\ 
\hline
s=5            & 0                                                                         & 0                   &  0                                                                & 0   \\ 
\hline
s=4            &  \bigoplus_{i=1}^{125} (\mathbf{Q}E_{i} \oplus \mathbf{Q}\mathcal{E}_{i}) & \mathbf{Q}\D{\Psi} \oplus \bigoplus_{i=1}^{24}\D{E_{i}} \oplus \bigoplus_{i=1}^{125}\mathbf{Q}\Lambda_{i}    & \bigoplus_{i=1}^{101} \mathbf{Q}\mathcal{E}_{i}                            & \mathbf{Q}\D{\Psi}\\
\hline
s=3            & 0                                                                         & \mathbf{Q}^{2}      & 0                                                                & \mathbf{Q}^{2}  \\
\hline  
s=2            & \bigoplus_{i=1}^{125} \mathbf{Q}1_{i}                                     & \mathbf{Q}\Psi \oplus \bigoplus_{i=1}^{24}\mathbf{Q}E_{i} \oplus \bigoplus_{i=1}^{125}\mathbf{Q}\D{\Lambda_{i}}    & 0                                                                & \mathbf{Q}\Psi \oplus \bigoplus_{i=1}^{24}\mathbf{Q}E_{i}\\
\hline
s=1            & 0                                                                         & 0                   &0                                                                & 0 \\
\hline
s=0            & 0                                                                         & \mathbf{Q}          & 0                                                                & \mathbf{Q} \\
\hline
\hline
               & r=-1                                                                      & r=0                 &   \ker\gamma^{s}                                         &\coker{\gamma^{s}}                      \\
\hline 
\end{array}  
\]

The formulas for the restriction morphism are
\[
\hspace{-3.2em}
\begin{array}{| c || c | c || c | c |}
\hline
s=6            & \mathbf{Q}                                                                                          & 0 & \mathbf{Q}                                                                               & 0   \\ 
\hline
s=5            & 0                                                                                                   & 0  & 0                                                                                        & 0     \\ 
\hline
s=4            & \mathbf{Q}\D{\Psi} \oplus \bigoplus_{i=1}^{24}\D{E_{i}} \oplus \bigoplus_{i=1}^{125}\mathbf{Q}\Lambda_{i}   & \bigoplus_{i=1}^{125} \mathbf{Q}\Lambda_{i}  & \mathbf{Q}\D{\Psi} \oplus \bigoplus_{i=1}^{24}\mathbf{Q}\D{E_{i}}                    & 0  \\
\hline
s=3            & \mathbf{Q}^{2}                                                                                    & 0   & \mathbf{Q}^{2}                                                                           & 0 \\
\hline
s=2            & \mathbf{Q}\Psi \oplus \bigoplus_{i=1}^{24}\mathbf{Q}E_{i} \oplus \bigoplus_{i=1}^{125}\mathbf{Q}\D{\Lambda_{i}}                          & \bigoplus_{i=1}^{125} (\mathbf{Q}E_{i} \oplus \mathbf{Q}\mathcal{E}_{i})   & \mathbf{Q}\Psi                                                                           & \bigoplus_{i=1}^{101}\mathbf{Q}E_{i}  \\
\hline
s=1            & 0                                                                                                   & 0 & 0                                                                                        & 0     \\
\hline
s=0            & \mathbf{Q}                                                                                          & \bigoplus_{i=1}^{125} \mathbf{Q}1_{i}  & 0                                                                                        & \bigoplus_{i=1}^{124} \mathbf{Q}1_{i}   \\
\hline
\hline
               & H^{s}(\overline{X})                                                                                 & H^{s}(D) & \ker j^{s}                                                                               & \coker{j^{s}}     \\
\hline 
\end{array}  
\]

We let the reader fill in the arrays for the top and zero perversities, we here give the result for the middle perversity $\overline{2}$.
\[
\begin{array}{| c || c c c || c |}
\hline
s = 6          & \bigoplus_{i=1}^{124} \mathbf{Q}\Lambda_{i}       &  0                 &   0                                                    & 0 \\ 
\hline
s = 5          & 0                                                 &  0                 &   0                                                    & \bigoplus_{i=1}^{124} \mathbf{Q}\Lambda_{i}    \\
\hline
s = 4          & \bigoplus_{i=1}^{101} \mathbf{Q}\mathcal{E}_{i}   & \mathbf{Q}\D{\Psi} &   0                                                    & \mathbf{Q}\D{\Psi} \\
\hline
s = 3          & 0                                                 & \mathbf{Q}^{2}     &   0                                                    & \mathbf{Q}^{204}\\
\hline
s = 2          & 0                                                 & \mathbf{Q}\Psi     & \bigoplus_{i=1}^{101}\mathbf{Q}E_{i}                   & \mathbf{Q}\Psi  \\
\hline
s = 1          & 0                                                 &   0                &   0                                                    &\textcolor{red}{\bigoplus_{i=1}^{124} \mathbf{Q}1_{i}} \\
\hline
\hline
s = 0          & 0                                                 & H^{0}(\overline{X})& \textcolor{red}{\bigoplus_{i=1}^{124} \mathbf{Q}1_{i}} & H^{0}(\overline{X})\\
\hline 
\hline
\EIbidg{2}{2}{r,s}{X} & r=-1                                       & r=0                & r=1                                                    & \HI{2}{s}{X} \\
\hline 
\end{array}
\]

\printbibliography

\end{document}